\documentclass{amsart}
\usepackage{amsmath,amssymb,amsfonts,amsthm}
\usepackage{mathtools} 
\usepackage{bbm}
\usepackage{titling}
\usepackage{enumitem}
\usepackage[style=alphabetic,doi=false,isbn=false,url=false,eprint=false,backend=bibtex]{biblatex}

\usepackage[hang,flushmargin]{footmisc}
 \usepackage[foot]{amsaddr}

\usepackage{tikz}
\definecolor{applegreen}{rgb}{0.55, 0.71, 0.0}
\definecolor{bleudefrance}{rgb}{0.19, 0.55, 0.91}
\definecolor{deepcarrotorange}{rgb}{0.91, 0.41, 0.17}
\definecolor{forestgreen}{rgb}{0.13, 0.55, 0.13}
\usepackage{hyperref}
\hypersetup{colorlinks=true,linkcolor=black,citecolor=black}
\usepackage{cleveref}
\usepackage{microtype}
\usepackage{fullpage}
\usepackage{setspace}
\usepackage{newpxmath}
\usepackage{tcolorbox}
\usepackage{framed}

\usepackage[title]{appendix}

\newcommand{\e}{\mathbb{E}}
\newcommand{\eps}{\varepsilon}

\newcommand{\op}[1]{ \operatorname{#1} }

\newcommand{\mc}[1]{ \mathcal{#1} }
\newcommand{\mb}[1]{ \mathbb{#1} }

\newcommand{\den}[1]{\left\lVert#1\right\rVert}

\newcommand{\abs}[1]{\left\lvert#1\right\rvert}

\crefname{thm}{Theorem}{Theorems}
\crefname{lem}{Lemma}{Lemmas}
\crefname{clm}{Claim}{Claims}
\crefname{rk}{Remark}{Remarks}
\crefname{prop}{Proposition}{Propositions}
\crefname{defn}{Definition}{Definitions}
\crefname{cor}{Corollary}{Corollaries}
\crefname{conj}{Conjecture}{Conjectures}
\crefname{question}{Question}{Questions}
\crefname{section}{Section}{Sections}

\theoremstyle{plain}
\newtheorem{thm}{Theorem}[section]
\newtheorem*{thm*}{Theorem}
\newtheorem{lem}[thm]{Lemma}
\newtheorem*{lem*}{Lemma}
\newtheorem{clm}[thm]{Claim}
\newtheorem*{clm*}{Claim}
\newtheorem{cor}[thm]{Corollary}
\newtheorem*{cor*}{Corollary}
\newtheorem{prop}[thm]{Proposition}
\newtheorem*{prop*}{Proposition}

\newtheorem*{conj*}{Conjecture}

\tcbuselibrary{theorems}
\newtcbtheorem[crefname={Theorem}{Theorems}]
{theo}{Theorem}%
{colback=white,colframe=black!75!white,fonttitle=\bfseries}{th}

\newtcbtheorem[number within=section,crefname={Lemma}{Lemmas}]
{lemm}{Lemma}%
{colback=white,colframe=black!75!white,fonttitle=\bfseries}{th}

\newtcbtheorem[use counter from=lemm,crefname={Proposition}{Propositions}]
{propo}{Proposition}%
{colback=white,colframe=black!75!white,fonttitle=\bfseries}{th}

\newtcbtheorem[use counter from=lemm,crefname={Proposition}{Propositions}]
{propo_lematic}{Proposition}%
{colback=white,colframe=black!75!white,fonttitle=\bfseries}{th}

\newtcbtheorem[use counter from=lemm,crefname={Corollary}{Corollaries}]
{corr_lematic}{Corollary}%
{colback=white,colframe=black!75!white,fonttitle=\bfseries}{th}

\newtcbtheorem[use counter from=lemm,crefname={Theorem}{Theorems}]
{theo_lematic}{Theorem}%
{colback=white,colframe=black!75!white,fonttitle=\bfseries}{th}

\newtcbtheorem[crefname={Conjecture}{Conjectures}]
{conje}{Conjecture}%
{colback=white,colframe=black!75!white,fonttitle=\bfseries}{th}

\newtcbtheorem[number within=section,crefname={Definition}{Definitions}]
{deff}{Definition}%
{colback=white,colframe=black!75!white,fonttitle=\bfseries}{th}

\newenvironment{prooff}[1][Proof]
    {\begin{proof}[#1]}
    {\end{proof}}

\theoremstyle{definition}
\newtheorem{defn}[thm]{Definition}
\newtheorem{defn*}{Definition}

\theoremstyle{remark}

\newtheorem*{rk*}{Remark}

\crefname{ineq}{inequality}{inequalities}
\creflabelformat{ineq}{#2{\upshape(#1)}#3} 

\title{Sharpness and locality for percolation \\ on finite transitive graphs}
\author{Philip Easo}

\address{The Division of Physics, Mathematics and Astronomy, California Institute of Technology}
\email{peaso@caltech.edu}

\begin{document}

\bigskip

\begin{abstract}
Let $(G_n) = \left((V_n,E_n)\right)$ be a sequence of finite connected vertex-transitive graphs with uniformly bounded vertex degrees such that $\lvert V_n \rvert \to \infty$ as $n \to \infty$. We say that percolation on $G_n$ has a \emph{sharp} phase transition (as $n \to \infty$) if, as the percolation parameter crosses some critical point, the number of vertices contained in the largest percolation cluster jumps from logarithmic to linear order with high probability. We prove that percolation on $G_n$ has a sharp phase transition unless, after passing to a subsequence, the rescaled graph-metric on $G_n$ (rapidly) converges to the unit circle with respect to the Gromov-Hausdorff metric. We deduce that under the same hypothesis, the critical point for the emergence of a giant (i.e.\! linear-sized) cluster in $G_n$ coincides with the critical point for the emergence of an infinite cluster in the Benjamini-Schramm limit of $(G_n)$, when this limit exists. 
\end{abstract}

\maketitle


\setcounter{tocdepth}{1} 
\tableofcontents


\section{Introduction}

Given a graph $G$, build a random spanning subgraph $\omega$ by independently including each edge of $G$ with
a fixed probability $p \in [0,1]$. The law of $\omega$ is called (Bernoulli bond) percolation and is denoted by $\mathbb P_p^G$. This simple model often undergoes a phase transition: for many natural choices of the underlying graph $G$, as $p$ increases past some critical value $p_c(G)$, the typical behaviour of the connnected components of $\omega$ changes abruptly. The study of this phenomenon has two origins, roughly coming from mathematical physics and combinatorics, respectively.

The first origin is the 1957 work of Broadbent and Hammersley \cite{MR0091567} introducing percolation on the Euclidean lattice $G = \mathbb Z^d$ as a model for the spread of fluid through a porous medium. Note that Euclidean lattices are always (vertex-)\emph{transitive}, meaning that for all vertices $u$ and $v$, there is a graph automorphism that maps $u$ to $v$. This is a way to formalise the notion that a graph is homogeneous or that its vertices are indistinguishable. For example, every Cayley graph of a finitely-generated group is transitive. In 1996, Benjamini and Schramm \cite{MR1423907} launched the systematic study of percolation on general infinite transitive graphs. A cornerstone of this theory is that percolation on an infinite transitive graph $G$ always undergoes a \emph{sharp} phase transition. Let us recall what this means. We will write $o$ to denote an arbitrary vertex in $G$ and write $\abs{K_o}$ to denote the cardinality of its \emph{cluster}, i.e.\! connected component in $\omega$.\footnote{More generally, $K_u$ denotes the cluster containing a vertex called $u$.} There is a trivial sense in which percolation on $G$ always undergoes a phase transition: by Kolmogorov's 0-1 law, there exists some critical point $p_c(G) \in [0,1]$ such that $\mathbb P_p^{G}\left( \text{there exists an infinite cluster}\right)$ equals $0$ for all $p < p_c(G)$ and equals $1$ for all $p > p_c(G)$. Now the phase transition is said to be \emph{sharp} if for all $p < p_c(G)$, not only does $\mathbb P_p^G(\abs{K_o} \geq n) \to 0$ as $n \to \infty$, but in fact there exists a constant $c(G,p) > 0$ such that $\mathbb P_p^G(\abs{K_o} \geq n) \leq e^{-cn}$ for every $n \geq 1$.\footnote{Some people use {sharpness} to mean slightly different things e.g.\! the exponential decay of point-to-point connection probabilities for $p < p_c$ together with the mean-field lower bound for $p > p_c$.} This was first proved in \cite{MR874906,MR852458} and now has multiple modern proofs \cite{duminil2016new,MR3898174,MR4408005,vanneuville2024exponentialdecayvolumebernoulli}.


The second origin is the 1960 work of Erd\H{o}s and R\'{e}nyi \cite{MR125031} investigating percolation on the complete graph $G_n$ with $n$ vertices. This is the celebrated Erd\H{o}s-R\'{e}nyi (or simply \emph{random graph}) model. The fundamental result is that percolation on $G_n$ undergoes a \emph{sharp} phase transition around $p = 1/n$ in the sense that for any fixed $\eps > 0$, the cardinality of the largest cluster of $\omega$ under $\mathbb P_p^G$ jumps from being\footnote{ Given functions $f,g : \mathbb N \to (0,\infty)$, we write $f(n) = \Theta(g(n))$ to mean that there are constants $c > 0$ and $C < \infty$ such that $c g(n) \leq f(n) \leq C g(n)$ for all $n$, i.e.\! $f(n) = O(g(n))$ and $g(n) = O(f(n))$.} $\Theta(\log n)$ at $p = (1-\eps)/n$ to being $\Theta(n)$ at $p = (1+\eps)/n$ with high probability as $n \to \infty$.\footnote{When $p > 1$ or $p < 1$, we define $\mathbb P_p$ to be $\mathbb P_1$ or $\mathbb P_0$ respectively.} Analogous results have since been established for certain other families of finite graphs with diverging degrees. For example, Ajtai, Koml\'{o}s, and Szemer\'{e}di \cite{MR0671140} and Bollob\'{a}s, Kohakayawa, and Łuksak \cite{MR1139488} investigated percolation on the hypercube $H_d = \{0,1\}^d$, which has a sharp phase transition around $p = 1/d$. Note that every complete graph and hypercube is transitive. For a small sample of the vast literature on percolation on finite graphs, see, for example, \cite{alon2004percolation,krivelevich2020asymptotics} on expanders, \cite{MR2020308} on pseudorandom graphs, \cite{MR2155704,MR2165583,MR2260845,MR2570320} on transitive graphs satisfying certain mean-field conditions, \cite{MR2599196} on dense graphs, and \cite{diskin2024percolationisoperimetry,diskin2024componentslargesmalli,diskin2024componentslargesmallii} on general graphs satisfying certain isoperimetric conditions.

Between these two settings lies the less-developed theory of percolation on bounded-degree finite transitive graphs. This theory, which started in 2001, was initiated by Benjamini \cite{benjamini2001percolation} and by Alon, Benjamini, and Stacey \cite{alon2004percolation}. This concerns the asymptotic properties of percolation on a finite transitive graph $G=(V,E)$ as $\abs{V}$ becomes large while the vertex degrees of $G$ remain bounded. As with the Erd\H{o}s-R\'{e}nyi model, here we are primarily interested in the phase transition for the emergence of a \emph{giant} cluster, i.e.\! a cluster containing $\Theta(\abs{V})$ vertices, and we will call the phase transition \emph{sharp} if the size of the largest cluster jumps from $\Theta(\log \abs{V})$ to $\Theta(\abs{V})$. (See \cref{subsec:main_result} for precise definitions.) At the same time, this theory is closely related to percolation on infinite transitive graphs via the local (Benjamini-Schramm) topology on the set of all transitive graphs. Indeed, with respect to this topology, every infinite set $\mathcal G$ of finite transitive graphs with bounded degrees is relatively compact, and every graph in the boundary of $\mc G$ is infinite.

Despite this close relation between infinite transitive graphs and bounded-degree finite transitive graphs, our understanding of percolation on infinite transitive graphs is quite far ahead. Roughly speaking, we can think of the theory of percolation on infinite transitive graphs as the theory of percolation on microscopic (i.e.\! $O(1)$) scales in bounded-degree finite transitive graphs. In this sense, the finite graph theory generalises the infinite graph theory. (A limitation of this maxim is that not every infinite transitive graph can be locally approximated by finite transitive graphs.) In particular, certain basic questions in the finite graph theory have no natural analogues in the infinite graph theory. For example, the uniqueness/non-uniqueness of giant clusters is not directly related to the uniqueness/non-uniqueness of infinite clusters, which is instead related to the microscopic metric distortion of giant clusters \cite[Remark 1.6]{easo2021supercritical}.

In this paper we investigate the following pair of closely related questions. An affirmative answer to the second question provides a direct way to move results and conjectures about infinite transitive graphs to finite transitive graphs.
\begin{enumerate}
	\item Does percolation on a large bounded-degree finite transitive graph $G$ have a sharp phase transition?
	\item If a finite transitive graph $G$ and an infinite transitive graph $H$ are close in the local sense, does the critical point for the emergence of a giant cluster in $G$ approximately coincide with the critical point for the emergence of an infinite cluster in $H$?
\end{enumerate}

Unfortunately, the answer to both of these questions in general is \emph{no}. For example, take the sequence $\left(\mb Z_n \times \mb Z_{f(n)}\right)_{n = 1}^{\infty}$ for any $f : \mb N \to \mb N$ growing fast. This sequence always converges locally to $\mb Z^2$, where the critical point for the emergence of an infinite cluster is $p_c=\frac{1}{2}$. On the other hand, provided that $f$ grows sufficiently fast, the threshold for the emergence of a giant cluster in $\mb Z_n \times \mb Z_{f(n)}$ will be as in the sequence of cycles, around $p_c = 1$. Moreover, for percolation of any fixed parameter $p \in (\frac{1}{2},1)$ on $\mb Z_n \times \mb Z_{f(n)}$, the order of the largest cluster will then typically be much larger than logarithmic but much smaller than linear in the total number of vertices. (See \cite[Example 5.1]{easo2023critical} for some more discussion of these sequences.) The problem is that these graphs are long and thin, coarsely resembling long cycles. In particular, after suitably rescaling, their graph metrics (rapidly) converge in the Gromov-Hausdorff metric to the unit circle. In this paper we prove that this is the only possible obstacle.

\subsection{Locality}
Question (2) above is the finite analogue of \emph{Schramm's locality conjecture}. This conjecture was (equivalently) that for all $\eps > 0$ there exists $R < \infty$ such that for every pair of infinite transitive graphs $G$ and $H$ that are not one-dimensional\footnote{An infinite transitive graph is one-dimensional if and only if the graph is quasi-isometric to $\mb Z$.}, if the ball of radius $R$ in $G$ is isomorphic to the ball of radius $R$ in $H$, then $\abs{p_c(G) - p_c(H)} \leq \eps$. This conjecture formalised the idea that the critical point of an infinite transitive graph should generally be entirely determined by the graph's small-scale, \emph{local} geometry. By building on earlier progress, especially the work of Contreras, Martineau, and Tassion \cite{contreras2022supercritical}, we verified this conjecture in our joint work with Hutchcroft \cite{easo2023critical}. Schramm's locality conjecture for infinite transitive graphs also spurred research on locality in other settings, including much research on the analogue of our question (2) about locality for finite graphs but where the hypothesis that the finite graphs are transitive is replaced by the hypothesis that they are \emph{expanders} \cite{MR2773031,MR4275958,ren2022localitycriticalpercolationexpanding,alimohammadi2023locality}.

It may be surprising, from the perspective of percolation on infinite transitive graphs, that in fact sharpness and locality for finite transitive graphs are equivalent. That is to say, if we restrict ourselves to any particular infinite set $\mc G$ of bounded-degree finite transitive graphs, then the answers to questions (1) and (2) in the introduction will always coincide. (See \cref{prop:equivalent_notions_of_sharpness} for a precise statement.) Indeed, if $\mc G$ satisfies locality, then one can easily extract sharpness for $\mc G$ from the sharpness of the phase transition on every infinite transitive graph that is a local limit of graphs in $\mc G$, and the converse, that sharpness implies locality, can also be established with a little more work. One reason that this equivalence may be surprising is because for infinite transitive graphs, sharpness always holds, even for $\mb Z$, whereas locality requires that the graphs are not one-dimensional. To make sense of this, consider that for infinite transitive graphs, locality corresponds to a version of sharpness that is \emph{uniform} in the choice of the graph, whereas in the context of finite transitive graphs, the only meaningful notion of sharpness is necessarily uniform.

Given the similarity between locality for finite and infinite graphs, one may wonder why the present paper is necessary: why does the proof of locality for infinite transitive graphs not also imply (perhaps after some additional bookkeeping) locality and hence sharpness for finite transitive graphs? The most fundamental reason is that the approach to proving locality in \cite{easo2023critical} relied inherently on the sharpness of the phase transition, which in our setting is what we are trying to prove! Let us be a little more precise. In the proof of \cite{easo2023critical}, we have an infinite transitive graph $G$ and a parameter $p$ that we want to show satisfies $p \geq p_c(G)$. The bulk of the argument in \cite{easo2023critical} involves delicately propagating point-to-point connection lower bounds across larger and larger scales to ultimately establish that for some function $f:\mb N \to (0,1)$ tending to zero slower than exponentially, $\mathbb P_p( u \leftrightarrow v ) \geq f(\op{dist}(u,v))$ for all vertices $u$ and $v$. Since point-to-point connection probabilities are decaying slower than exponentially, the conclusion $p \geq p_c(G)$ then follows from the sharpness of the phase transition on infinite transitive graphs. In a finite graph adaptation of this argument, at this final stage we would need to invoke the sharpness of the phase transition for finite transitive graphs, making the argument circular. One might hope to circumvent this problem by improving the locality argument so that the function $f$ does not tend to zero at all. Unfortunately, $f$ tends to zero because the propagation of point-to-point lower bounds in the locality argument is \emph{lossy}, i.e.\! a lower bound of $\eps_i$ at scale $n_i$ is propagated to a lower bound of $\eps_{i+1}$ at scale $n_{i+1}$ where $\eps_{i+1} \ll \eps_i$, which seems completely unavoidable to us with current technology. 

We will exploit the fact that the locality argument produces an explicit choice for $f$ that decays much slower than exponentially (even slower than algebraically). So for this final step, one only needs a weaker kind of \emph{quasi-sharpness} of the phase transition to conclude. The new idea in the present paper is to directly establish this quasi-sharpness by applying quantitative versions of the proofs of two results that are a priori quite unrelated to locality: the uniqueness of the supercritical giant cluster \cite{easo2021supercritical} and the existence of a percolation threshold \cite{MR4665636} on finite transitive graphs. In short, we can think of the existence of a percolation threshold as the weakest possible kind of \emph{quasi-sharpness}. In general, if we allow graphs to have unbounded degrees (as we did in \cite{MR4665636}), then the implicit rates of convergence can be arbitrarily slow. Luckily, now assuming bounded degrees as we may in the present paper, we can plug into our argument in \cite{MR4665636} a quantitatively strong version of the uniqueness of the supercritical giant cluster from \cite{easo2021supercritical} to get a quantitatively strong quasi-sharpness that suffices to conclude the proof of locality.


	There are also quite serious obstacles to adapting to finite graphs the part of the proof of locality leading up to this application of sharpness. To illustrate, say we tried to run the locality argument on an infinite transitive graph that \emph{is} one-dimensional. What would go wrong? We would encounter a scale where we are unable to efficiently propagate connection lower bounds because two otherwise complementary arguments simultaneously break down. The breakdown of the first argument implies that $G$ cannot be one-ended ($G$ is\footnote{Technically this applies to a certain graph $G'$ that approximates $G$.} the Cayley graph of a  finitely-presented group but its minimal cutsets are not coarsely connected), while the breakdown of the second implies that $G$ must have finitely many ends ($G$ has polynomial growth because $G$ contains a large ball with small tripling). From this we deduce that $G$ is two-ended, thereby successfully identifying that $G$ was one-dimensional. On a finite transitive graph, these end-counting arguments are not applicable. This will require us to make the locality argument more finitary, even in the setting of infinite transitive graphs, which is of independent interest. Unfortunately, this end-counting argument is so deeply embedded in the proof of \cite{easo2023critical} that it will take some work to reorganise the high-level multi-scale induction in \cite{easo2023critical} in order to isolate and make explicit the relevant part. Another obstacle is that the definition of \emph{exposed spheres}, whose special connectivity properties played a pivotal role in \cite{contreras2022supercritical,easo2023critical}, degenerates on finite transitive graphs. As part of our argument, we introduce the \emph{exposed sphere in a finite transitive graph}, justify our definition (\cref{lem:removing_B_r_does_not_disconnect}), and thereby establish that from the perspective of part of our argument, arbitrary finite transitive graphs can be treated like infinite transitive graphs that are one-ended. We hope that these basic geometric objects can be of use in future work on finite transitive graphs, analogously to their infinite counterparts.

\subsection{Statement of the main result} \label{subsec:main_result}

Graphs will always be assumed to be connected, simple, countable, and locally finite. In a slight abuse of language, we identify together all graphs that are isomorphic to each other.\footnote{So a ``graph'' $G$ is really a graph-isomorphism equivalence class of graphs.} Let $\mathcal G$ be an infinite set of finite transitive graphs. Note that $\mc G$ is countable. We will write $\lim_{G \in \mathcal G}$ to denote limits taken with respect to some (and hence every) enumeration of $\mathcal G$. We may omit references to $G$ and $\mc G$ when this does not cause confusion. Given a graph $G$, we will also assume by default that $V$ and $E$ refer to the sets of vertices and edges in $G$.

Given a percolation configuration $\omega$, we write $\abs{K_1}$ to denote the cardinality of the largest cluster. A sequence $p : \mc G \to (0,1)$ is said to be a \emph{percolation threshold} if for every constant $\eps > 0$, we have\footnote{This equation means that for some (and hence every) enumeration $\mc G = \{G_1,G_2,\ldots\}$ where $G_n = (V_n,E_n)$, we have \[\lim_{n\to \infty} \mathbb P_{(1+\eps)p(G_n)}^{G_n}(\abs{K_1} \geq \alpha \abs{V_n}) = 1.\]} $\lim \mathbb P_{(1+\eps)p} ( \abs{K_1} \geq \alpha \abs{V} ) =1$ for some constant $\alpha > 0$, whereas $\lim \mathbb P_{(1-\eps)p} ( \abs{K_1} \geq \beta \abs{V} ) = 0$ for every constant $\beta > 0$. Note that when a percolation threshold exists, it is unique up to multiplication by $1+o(1)$. So in this sense, we may refer to \emph{the} percolation threshold for $\mc G$, when one exists. Now assume that $\mc G$ has bounded degrees, i.e.\! there exists $d \in \mb N$ such that for every $G \in \mc G$, every vertex in $G$ has degree at most $d$. By \cite{MR4665636}, $\mc G$ always has a percolation threshold, say $p$. We say that percolation on $\mc G$ has a \emph{sharp phase transition} if for every constant $\eps > 0$, there exists a constant $A < \infty$ such that
\[
	\lim \mathbb P_{(1-\eps)p} ( \abs{K_1} \geq A \log \abs{V} ) =0.
\]
Conversely, it is not hard to show in general that $\liminf p \geq \frac{1}{d-1} > 0$ (see \cite[Proposition 5]{MR4665636}) and that the complementary bound on $\abs{K_1}$ always holds in the sense that if $\liminf (1-\eps)p >0$ then there exists $A < \infty$ such that $\lim \mathbb P_{(1-\eps)p} ( \abs{K_1} \geq \frac{1}{A} \log \abs{V} ) =1$ (see \cref{prop:local_connections_to_large_cluster}).

Given a transitive graph $G$, we write $o$ to denote an arbitrary vertex, and we write $B_n^G$ to denote the graph-metric ball of radius $n$ centred at $o$, viewed as a rooted subgraph of $G$. We also write $\op{Gr}(n)$ for the number of vertices in $B_n^G$, and define $S_n^G$ to be the sphere\footnote{We generalise these to non-integer $n$ by setting $B_n^G : = B_{\lfloor n \rfloor}^G$ and by defining $S_n^G$ and $\op{Gr}(n)$ analogously.} of radius $n$. The \emph{local} (aka \emph{Benjamini-Schramm}) topology on the set of all transitive graphs is the metrisable\footnote{This topology is induced by the metric $\op{dist}(G,H) := \exp(-\max \{ n : B_n^{G} \cong B_n^{H} \})$, for example.} topology with respect to which a sequence $(G_n)$ converges to $G$ if and only if for $r \in \mb N$, the balls $B_r^{G_n}$ and $B_r^{G}$ are isomorphic for all sufficiently large $n$. For example, the sequence of tori $(\mb Z_n^2)_{n=1}^{\infty}$ converges \emph{locally} to $\mb Z^2$. Given metric spaces $X$ and $Y$, the \emph{Gromov-Hausdorff} distance between $X$ and $Y$, denoted $\op{dist}_{\mathrm{GH}}(X,Y)$, is the infimum over all $\eps > 0$ such that there exists a metric space $Z$ and isometric embeddings $\phi:X \to Z$ and $\psi: Y \to Z$ such that the Hausdorff distance between the images of $\phi$ and $\psi$ in $Z$ is at most $\eps$. Given a graph $G$ and $r > 0$, we write $r G$ for the rescaled graph metric of $G$ where all distances are multiplied by $r$. For example, the sequence of rescaled tori $(\frac{2\pi}{n}\mb Z_n^2)_{n=1}^{\infty}$ Gromov-Hausdorff converges to the continuum torus $S^1 \times S^1$ with the $L^1$ metric, where $S^1$ is the unit circle. The scaling limits that arise like this, as a Gromov-Hausdorff limit of a sequence of diameter-rescaled finite transitive graphs, are explored in \cite{MR3666783}.

The main result of our paper resolves the problems of sharpness and locality for all bounded-degree finite transitive graphs that are not one-dimensional in a certain coarse-geometric sense.

\begin{thm}\label{thm:main}
Let $\mathcal G$ be an infinite set of finite transitive graphs with bounded degrees. Suppose that there does not exist an infinite subset $\mc H \subseteq \mc G$ such that $\left(\frac{\pi}{\op{diam} G} G \right)_{G \in \mc H}$ Gromov-Hausdorff converges to the unit circle. Then both of the following statements hold:
\begin{enumerate}
	\item Percolation on $\mathcal G$ has a sharp phase transition.
	\item If $\mc G$ converges locally to an infinite transitive graph $H$, then the constant sequence $p : G \mapsto p_c(H)$ is the percolation threshold for $\mc G$.
\end{enumerate}
In fact, if either of these two statements is false, then there exists an infinite subset $\mc H \subseteq \mc G$ such that for every $G \in \mathcal H$,
\[
	\op{dist}_{\mathrm{GH}}\left( \frac{\pi}{\op{diam} G} G, S^1 \right) \leq \frac{e ^{\left( \log \op{diam} G\right)^{1/9} } }{\op{diam} G}.
\]
\end{thm}


We interpret the upper bound on Gromov-Hausdorff distance as a bound on the rate of convergence of the large scale geometry of graphs in $\mc H$ towards the unit circle as their diameters tend to infinity. In this sense, graphs in $\mc H$ converge to the unit circle faster than do the tori $\{\mb Z_n \times \mb Z_{e^{(\log n)^8}}\}_{n \geq 1}$, and in particular faster than do the polynomially-stretched tori $\{\mb Z_n \times \mb Z_{n^C}\}_{n \geq 1}$ for any constant $C$. On the other hand, by using arguments specific to Euclidean tori \cite[Example 5.1]{easo2021supercritical}, items 1 and 2 only fail once we reach exponentially-stretched tori $\{ \mb Z_n \times \mb Z_{C^n} \}_{n \geq 1}$ for a constant $C$. So our rate is not sharp in this special case, even if we could improve the exponent $1/9$, which we did not try to optimise. Perhaps these exponentially-stretched tori are worst possible, in which case the optimal bound on the rate should be on the order of $\frac{\log \op{diam} G}{\op{diam} G}$ instead of $\frac{e^{(\log \op{diam} G)^{1/9}}}{\op{diam}G}$. There are also stronger ways that one could hope to describe the ``one-dimensionality'' of $\mathcal H$. (See the discussion at the end of \cref{subsec:further_discussion}.)

\subsection{Previous work and strategy of the proof}
Recall from our earlier discussion that sharpness and locality for finite transitive graphs are equivalent. (See \cref{prop:equivalent_notions_of_sharpness}.) In this paper we will prove sharpness directly. At a high level, our idea is to apply arguments derived from the proofs of four existing results in succession: (1) The sharpness of the phase transition for infinite transitive graphs; (2) The locality of the critical point for infinite transitive graphs; (3) The uniqueness of the supercritical giant cluster on finite transitive graphs; (4) The existence of a percolation threshold on finite transitive graphs. Below we discuss each of these works and how they feature in our argument.

To prove sharpness, we will start with a sequence $p: \mathcal G \to (0,1)$ where $\mathbb P_p$ has a cluster larger than a large multiple of $\log \abs{V}$ with good probability. Then given any $\eps > 0$, we will show that $\mathbb P_{(1+\eps)p}$ has a giant cluster with high probability. Since $\inf p > 0$, we can replace $\mathbb P_{(1+\eps)p}$ by $\mathbb P_{p+\eps}$, or equivalently, $\mathbb P_{p+4\eps}$. We will split the jump $p \to p+4\eps$ into four little hops $p \to p +\eps \to \ldots \to p+4\eps$. After each hop, we will prove something stronger about the connectivity properties of percolation at the current parameter. Each hop is the subject of one section, discussed below, and involves one of the four works listed above.

\subsubsection{Large clusters $\to$ local connections}

The sharpness of the phase transition for infinite transitive graphs is the statement that for every infinite transitive graph $G$ and every $p < p_c(G)$, there is constant $c(G,p) > 0$ such that $\mathbb P_p^G(\abs{K_o} \geq n) \leq e^{-cn}$ for every $n \geq 1$. Some people use slightly different definitions. For example, some replace this exponential tail on $\abs{K_1}$ by the exponential decay of connection probabilities, which is a priori weaker, and some include the mean-field lower bound $\theta((1+\eps)p_c(G)) \geq \frac{\eps}{1+\eps}$ as part of the definition. The analogue of the mean-field lower bound for finite transitive graphs was already established in full generality in our earlier work \cite[Corollary 4]{MR4665636}. This statement of sharpness for infinite transitive graphs looks similar to our definition of sharpness for an infinite set $\mathcal G$ of finite transitive graphs with bounded degrees. Indeed, let $p$ be the percolation threshold for $\mathcal G$. By a simple union bound, if for all $\eps > 0$ there exists $C(\mc G , \eps) < \infty$ such that $\mathbb P_{(1-\eps)p(G)}^G( \abs{K_o} \geq n ) \leq Ce^{-n/C}$ for all $n \geq 1$ and $G \in \mc G$, then percolation on $\mathcal G$ has a sharp phase transition. With a little more work (see \cref{prop:equivalent_notions_of_sharpness}), one can show that the converse holds too.

So a natural approach towards proving sharpness for finite transitive graphs is to try to adapt an existing proof of sharpness for infinite transitive graphs. Before explaining what is wrong with this approach, notice that \emph{something} must go wrong because these arguments are completely general, applying to every infinite transitive graph - including $\mb Z$, whereas as illustrated by the sequence of stretched tori, some hypothesis on the geometry of finite transitive graphs is required for sharpness to hold. The problem is not that the arguments cannot be run, but rather that they do not address the right question. Roughly speaking, the issue is that a cluster that grows faster than every particular microscopic scale is not automatically macroscopic. Slightly more precisely, given an infinite graph $G$ and parameter $p$, if $\inf_{n \geq 1}\mathbb P_p( \abs{K_o} \geq n) > 0$, then under $\mathbb P_p$ there is an infinite cluster almost surely. However, for an infinite set $\mathcal G$ of finite graphs and a sequence of parameters $p$, if $\inf_{n \geq 1} \liminf_{\mc G} \mathbb P_{p}( \abs{K_o} \geq n ) > 0$, then it does not necessarily follow that under $\mathbb P_p$ there is a giant cluster with high probability. 


While proofs of sharpness for infinite transitive graphs do not directly yield \cref{thm:main}, our first step is still to adapt and run one of these proofs on finite transitive graphs. We will also apply Hutchcroft's idea \cite{hutchcroft2019locality} of using his two-ghost inequality to convert point-to-sphere bounds into point-to-point lower bounds. (This was also the first step of \cite{easo2023critical}.) Together, this will establish that after the first hop, $\mathbb P_{p+\eps}$ satisfies a point-to-point lower bound on a large constant scale.

In this section we will also use an elementary spanning tree argument to prove a kind of ``reverse'' implication that if $\mathbb P_p$ was instead assumed to satisfy such a point-to-point lower bound, then it would follow that $\mathbb P_{p}$ has a cluster much larger than $\log \abs{V}$ with high probability. This reverse direction is not relevant to proving \cref{thm:main}, but we will apply it to establish the equivalence of different characterisations of sharpness on finite transitive graphs and in particular to prove the equivalence of items 1 and 2 in \cref{thm:main}.

\subsubsection{Local connections $\to$ global connections}

Earlier we discussed the proof of the locality of the critical point for infinite transitive graphs \cite{easo2023critical} and the obstructions to using the same argument to prove locality for finite transitive graphs. The primary obstruction was the application of sharpness, which we explained could be replaced by a good enough quantitative \emph{quasi-sharpness}. If we had organised the argument in the present paper as a direct proof of locality, rather than of sharpness, then this quasi-sharpness would be supplied by the following two hops ($p+2\eps \to p+ 3 \eps \to p + 4\eps$). For the current hop ($p+\eps \to p + 2\eps$), we will run the part of the proof of locality for infinite graphs leading up to this application of sharpness (after dealing with the challenges that we discussed this entails) to propagate the microscopic point-to-point lower bound at $p+\eps$ to a global point-to-point lower bound at $p+2\eps$. More precisely, we prove that if $\mathcal G$ does not contain a sequence converging rapidly to the unit circle, then for some explicit and slowly-decaying function $f : \mb N \to (0,1)$, all but finitely many graphs $G \in \mc G$ satisfy
\[
	\min_{u,v \in V} \mathbb P_{p+2\eps}( u \leftrightarrow v ) \geq f(\abs{V}).
\]



\subsubsection{Global connections $\to$ unique large cluster}

In the supercritical phase of percolation on a bounded-degree finite transitive graph, there is exactly one giant cluster with high probability. This had been conjectured by Benjamini and was verified in our joint work with Hutchcroft \cite{easo2021supercritical}. It is important to note that this result actually does not rely on the existence of a percolation threshold. To make sense of this, we need a definition of the supercritical phase that is agnostic to the existence of a percolation threshold.

Let $\mathcal G$ be an infinite set of bounded-degree finite transitive graphs, and let $q : \mc G \to (0,1)$ be a sequence of parameters. If $G$ admits a percolation threshold $p$, then the natural definition for $q$ being supercritical is that $\liminf q/p > 1$. To make this independent of the existence of $p$, we say that $q$ is supercritical if there exists a sequence $q' : \mathcal G \to (0,1)$ and a constant $\eps > 0$ such that $\liminf q/q' > 1$ and $\liminf \mathbb P_{q'} ( \abs{K_1} \geq \eps \abs{V} ) \geq \eps$. In this language, the main result of \cite{easo2021supercritical} is that for every supercritical sequence $q$, the number of vertices $\abs{K_2}$ contained in the second largest cluster satisfies $\lim \mathbb P_q( \abs{K_2} \geq \delta \abs{V} ) = 0$ for every constant $\delta > 0$.

The argument in \cite{easo2021supercritical} is fully quantitative. In particular, if we slightly weaken the hypothesis that $q$ is supercritical by replacing the constant $\eps > 0$ in the definition of ``$q$ is supercritical" by a slowly decaying sequence $\eps : \mc G \to (0,1)$, then we can still deduce that under $\mathbb P_q$ the largest cluster is much larger than all other clusters with high probability. What we need is the same conclusion but with the alternative hypothesis that $\delta := \min_{u,v \in V} \mathbb P_{q'}(u \leftrightarrow v)$ tends to zero slowly. This is certainly possible in principle because by Markov's inequality, the lower bound $\min_{u,v \in V}\mathbb P_{q'}(u \leftrightarrow v) \geq 2\eps$ always implies the lower bound $\mathbb P_{q'} ( \abs{K_1} \geq \eps \abs{V} ) \geq \eps$. Unfortunately, this approach ultimately requires that $\delta$ tends to zero extremely slowly, too slowly for our purposes. Fortunately, the argument in \cite{easo2021supercritical} turns out to run much more efficiently if we directly supply the hypothesis that $\min_{u,v \in V} \mathbb P_{q'}(u \leftrightarrow v) \geq \eps$ rather than the hypothesis that $\mathbb P_{q'} ( \abs{K_1} \geq \eps \abs{V} ) \geq \eps$. Indeed, a significant loss in the proof in \cite{easo2021supercritical} is due to the conversion of the latter into the former. We will apply this to deduce from the global point-to-point lower bound at $\mathbb P_{p+2\eps}$ that under $\mathbb P_{p+3\eps}$, the largest cluster is much larger than all other clusters with high probability. However, note that a priori this largest cluster might \emph{not be giant}, i.e.\! we may still have $\abs{K_1} = o(\abs{V})$ with high probability. In particular, our proof is not complete at this stage, which is why we need the fourth hop.

\subsubsection{Unique large cluster $\to$ giant cluster}

Every infinite set of finite transitive graphs with bounded degrees $\mc G$ admits a percolation threshold. We verified this in \cite{MR4665636} by combining \cite{easo2021supercritical,vanneuville2023sharpnessbernoullipercolationcouplings}. The reader may find it surprising that the uniqueness of the supercritical giant cluster comes first, before the existence of a percolation threshold. Indeed, this is opposite to the order in the classical story for the Erd\H{o}s-R\'{e}nyi model, for example. On the other hand, the reader may suspect that the result is obvious because standard sharp threshold techniques imply that for every sequence $\alpha$, the event $\{\abs{K_1} \geq \alpha\abs{V}\}$ always has a sharp threshold.\footnote{We say that a sequence of events $(A(G) )_{G \in \mc G}$ has a {sharp threshold} if there exists a sequence $p$ such that $\limsup q/p < 1$ implies $\mb P_q( A ) = 0$ and $\liminf q/p > 1$ implies $\mb P_q(A) = 1$ for every sequence $q$.} The challenge is to prove that every sequence $\alpha$ that decays sufficiently slowly has a \emph{common} sharp threshold.

To prove this we embedded the fact that the supercritical giant cluster is unique into Vanneuville's proof of the sharpness of the phase transition for infinite transitive graphs. In \cite{MR4665636}, we did not give any explicit bounds because we were working without the hypothesis that $\mathcal G$ has bounded degrees. At this level of generality, there actually exist (very particular) sequences that do not admit a percolation threshold, and even for those that do, the implicit rates of convergence can be arbitrarily bad. However, our argument is itself fully quantitative. In particular, we will explain how it can still be run under an explicit weaker version of the uniqueness of the giant cluster. This will allow us to deduce from the global two-point lower bound under $\mathbb P_{p+2\eps}$ and the uniqueness of the largest cluster under $\mathbb P_{p+3\eps}$ that there is a giant cluster under $\mathbb P_{p+4\eps}$ with high probability, completing our proof of \cref{thm:main}.

\subsection{Further discussion} \label{subsec:further_discussion}

Let us further explore the connection between percolation on finite and infinite transitive graphs. First, let us remark on how to canonically \emph{define} $p_c$ for finite graphs. By \cite{MR4665636}, there exists a universal function $p_c : \mathcal F \to (0,1)$, where $\mathcal F$ is the set of all finite transitive graphs, such that for every infinite set $\mc G$ of finite transitive graphs with bounded degrees, the restriction $p_c\vert_{\mc G}$ is the percolation threshold for $\mc G$. Let us fix such a function $p_c$ for the rest of this section. Now, thanks to \cref{thm:main}, we can roughly\footnote{It is unique (up to $o(1)$) and continuous whenever we restrict to an infinite set $\mc G$ that is compact in the local topology and satisfies $\inf_{G \in \mathcal G} \op{dist}_{\mathrm{GH}}\left(\frac{\pi}{\op{diam}G} G, S^1\right) > 0$.} interpret this as the unique continuous extension with respect to the local topology of the usual percolation threshold $p_c$ for infinite transitive graphs to the set of finite transitive graphs.

Let $H$ be an infinite transitive graph, and let $\mc G$ be an infinite set of finite transitive graphs with bounded degrees that does not contain a sequence approximating the unit circle in the sense that $\mc H$ does in \cref{thm:main}. Let $(V_n)$ be an exhaustion of $H$ by finite sets, and let $K_{\infty}$ denote the set of vertices contained in infinite clusters. By a second-moment calculation, under $\mathbb P_p^H$ for any $p$,
\[
	\frac{\abs{ K_{\infty} \cap V_n }} { \abs{V_n} } \to \theta^H(p) : = \mathbb P_p^H( o \leftrightarrow \infty )
\]
in probability as $n$ tends to infinity. In this sense, $\theta^H(p)$ captures the \emph{density} of the union of the infinite clusters. In a finite graph $G$, we define the giant density to be $\den{K_1} : = \frac{1}{\abs{V}} \abs{K_1}$. In conjunction with the main result of \cite{easo2021supercritical2}, \cref{thm:main} implies that if $\mathcal G$ converges locally to $H$, then for every constant $p \in (0,1) \backslash \{ p_c(H) \}$, the density $\den{K_1}$ under $\mathbb P_{p}^G$ converges in probability to the density $\theta^H(p)$ under $\mathbb P_{p}^H$ as we run through $G \in \mc G$. In this sense, the infinite cluster phenomenon on infinite transitive graphs is a \emph{good model} for the giant cluster phenomenon on finite transitive graphs. Similar ideas are discussed in Benjamini's original work \cite{benjamini2001percolation}.

In light of this, our results let us easily move statements about infinite transitive graphs to the setting of finite transitive graphs. Here are three examples. For all three, remember that $\mc G$ is assumed to be a family of graphs satisfying the hypotheses of \cref{thm:main}. First, it is well-known that $p_c(H) < 1$ if (and only if) $H$ is not one-dimensional \cite{MR4181032}. By the conclusion of \cref{thm:main}, it immediately follows\footnote{One just needs to verify that $\mc G$ cannot converge locally to a one-dimensional infinite transitive graph, e.g.\! by \cref{lem:quotient_of_Z_is_circle}.} that $\sup_{G \in \mc G} p_c(G) < 1$, i.e.\! there exists $\eps > 0$ such that $\mathbb P_{1-\eps} ( \abs{K_1} \geq \eps \abs{V} ) \geq \eps$ every $G \in \mc G$. This conclusion is not new; we simply wish to illustrate how easily it follows from \cref{thm:main}. Indeed, Hutchcroft and Tointon established this fundamental result under a weaker (essentially optimal!) version of the hypothesis that $\mc G$ is not one-dimensional, (almost) fully resolving a conjecture of Alon, Benjamini, and Stacey \cite{MR2073175}. Second, it is a major open conjecture that $\theta^H(\cdot)$ is continuous if (and only if) $H$ is not one-dimensional. Following the discussion in our previous paragraph, this conjecture would immediately imply the following statement, which says that the giant cluster emerges \emph{gradually}: Let $\mathbb P$ denote the law of the standard monotone coupling $(\omega_p : p \in [0,1])$ of the percolation measures $(\mathbb P_p : p \in [0,1])$, and define $\alpha(p) := \den{K_1(\omega_p)}$. Then for all $\eps > 0$ there exists $\delta > 0$ such that
\begin{equation}\label{eq:gradually}
	\lim_{G \in \mc G} \mathbb P\left( \sup_{p} \left[ \alpha(p+\delta) - \alpha(p)\right] \leq \eps \right) = 1.
\end{equation}
For example, thanks to \cite{hutchcroft2016critical}, we can already deduce from this relation between infinite and finite transitive graphs that \cref{eq:gradually} holds whenever our family $\mc G$ has exponential growth on microscopic scales, e.g.\! for the sequence $(\mb Z^3_n \times G_{h(n)})$ where $(G_n)$ is a sequence of transitive expanders and $h:\mb N \to \mb N$ tends to infinity arbitrarily slowly. (See the discussion in \cite[Example 5.1]{easo2023critical} of stretched tori, in which the giant cluster does not emerge gradually.) By the uniqueness of the supercritical giant cluster \cite{easo2021supercritical}, finite transitive graphs satisfying \cref{eq:gradually} also automatically satisfy the conclusion of \cite[Conjecture 1.1]{alon2004percolation}. This links the well-known continuity conjecture for infinite transitive graphs to this conjecture about the uniqueness of the largest cluster in finite transitive graphs. Third, it is conjectured that the uniqueness threshold $p_u(H)$ satisfies $p_c(H) < p_u(H)$ if and only if $H$ is nonamenable. Again by our discussion in previous paragraph, this conjecture would imply that if the Cheeger constant on graphs in $\mc G$ is uniformly bounded below on microscopic scales, then percolation on $\mc G$ has a phase in which there is a giant cluster whose metric distortion tends to infinity. (See \cite[Remark 1.6]{easo2021supercritical}.) What can be said when the Cheeger constant is uniformly bounded below on larger scales? In the limit, this connects the $p_c$ vs $p_u$ question to the existing theory of percolation on expanders.

This opens the door to many directions for future work, adapting questions and techniques from percolation on infinite transitive graphs to finite transitive graphs. For example, what can be said about supercritical sharpness? Since the continuity conjecture for infinite transitive graphs would imply the unique giant cluster conjecture of \cite[Conjecture 1.1]{alon2004percolation} (possibly with a weaker one-dimensionality condition), might \cite[Conjecture 1.1]{alon2004percolation} be a stepping stone towards continuity that is easier to establish? Another direction for future work is to improve the rate of convergence in \cref{thm:main}. One could also explore stronger notions of one-dimensionality. The Gromov-Hausdorff metric only considers the coarse geometry of graphs, ignoring how densely vertices are packed (i.e.\! the volume growth on small scales). It is natural to expect that graphs in which vertices are packed more densely can afford to have a more one-dimensional coarse geometry before percolation arguments break down. For example, consider the product of a torus with a long cycle versus the product of an expander with a long cycle. 
In the work of Hutchcroft and Tointon \cite{hutchcroft2021nontriviality}, one-dimensionality was characterised more stringently\footnote{This is indeed stronger than asking for Gromov-Hausdorff convergence to the unit circle, by the results of \cite{MR3666783}.} in terms of the relationship of volume to diameter, for example, by requiring that $\abs{V} \leq (\op{diam} G)^{1+\eps}$ or $\frac{\abs{V}}{\log \abs{V}} = o(\op{diam} G)$. One could also investigate questions such as sharpness without bounded degrees. For example, \cite{easo2021supercritical,MR4665636,easo2021supercritical2} did not require this hypothesis, thus linking the story of percolation on infinite transitive graphs to the classical Erd\H{o}s-R\'{e}nyi model.

\subsection{Acknowledgement} We thank Tom Hutchcroft for helpful comments on an earlier draft.

\section{Large clusters $\to$ local connections} \label{sec:large_clusters_to_local_connections}
In this section we will adapt a proof of the sharpness of the phase transition for infinite transitive graphs to finite transitive graphs. By combining this with an idea of Hutchcroft \cite{hutchcroft2019locality} to convert voume-tail bounds into point-to-point bounds, we will prove the following proposition. This roughly says that if for some percolation parameter $p$, the largest cluster contains much more than $\log \abs{V}$ vertices with good probability, then for percolation of any higher parameter $p+\eta$, we have a uniform point-to-point lower bound on a divergently large scale. Later, in \cref{subsec:local_connections_to_large_clusters} we will prove a kind of converse to this statement, and in \cref{subsec:equivalent_notions_of_sharpness} we will use this converse to give equivalent characterisations of sharpness for finite transitive graphs.

\begin{prop}\label{prop:large_cluster_to_local_connections}
Let $G$ be a finite transitive graph with degree $d$. Let $\eta > 0$. There exists $c(d,\eta) > 0$ such that for all $p \in (0,1)$ and $\lambda \geq 1$,
\[
	\mathbb P_p(\abs{K_1} \geq \lambda \log \abs{V}) \geq \frac{1}{c\abs{V}^{c}} \quad \implies \quad \min_{u \in B_{c\log \left( \lambda \right) - \frac{1}{c}} }  \mathbb P_{p + \eta} ( o \leftrightarrow u ) \geq \frac{\eta^2}{20}.
\]
\end{prop}

We have chosen to adapt Vanneuville's recent proof of sharpness for infinite graphs \cite{vanneuville2024exponentialdecayvolumebernoulli}. This involves ghost fields. Given a graph $G$, a ghost field of intensity $q \in (0,1)$ is a random set of vertices $g \subseteq V$ distributed according to (Bernoulli) site percolation of parameter $q$.\footnote{Some authors use a slightly different parameterisation. When we write ``a ghost field of intensity $q \in (0,1)$", they write ``a ghost field of intensity $h > 0$" for the same object, where $q = 1-e^{-h}$.} We denote its law by $\mathbb Q_q$ and write $\mathbb P_p \otimes \mathbb Q_q$ for the joint law of independent samples $\omega \sim \mathbb P_p$ and $g \sim \mathbb Q_q$. One reason to introduce ghost fields is that it can be easier to work with the event $\{ o \leftrightarrow g \}$ when $q = 1/n$ than to work with the closely related event $\{ \abs{K_o} \geq n \}$.

The following is \cite[Theorem 2]{vanneuville2024exponentialdecayvolumebernoulli}. This can also be deduced from \cite{MR4408005} with different constants. This says that starting from any percolation parameter $p$, if we decrease $p$ by a suitable amount, then the volume of the cluster at the origin will have an exponential tail under the new parameter. This is proved by a variant of Vanneuville's stochastic comparison technique from \cite{vanneuville2023sharpnessbernoullipercolationcouplings}, which we will describe in more detail in \cref{subsec:coupled_expl}.

\begin{lem}\label{lem:Vanneuville}
Let $G$ be a transitive graph. Given $p \in (0,1)$ and $h > 0$, define
\[
\mu_{p,h} := \mathbb P_p \otimes \mathbb Q_{1-e^{-h}} ( o \xleftrightarrow{\omega} g ).\]
Then for all $m \geq 1$,
\[
	\mathbb P_{(1-\mu_{p,h})p}( \abs{K_o} \geq m ) \leq \frac{\mathbb P_p( \abs{K_o} \geq m )}{1-\mu_{p,h} } e^{-hm}.
\]
\end{lem}

Vanneuville proved this lemma when $G$ is infinite, but his proof also works verbatim when $G$ is finite. The following easy corollary of this lemma says (contrapositively) that if the cluster at the origin is much larger than $\log \abs{V}$ with reasonable probability, then after sprinkling, the cluster at the origin is at least mesoscopic with good probability.

\begin{cor}\label{cor:superlog_implies_meso}
Let $G$ be a finite transitive graph. For all $\eps > 0$ there exists $c(\eps) > 0$ such that for all $p \in (0,1)$ and $n,\lambda \geq 1$,
\[
	\mathbb P_p(\abs{K_o} \geq n) \leq \eps \quad \implies \quad \mathbb P_{(1-2\eps)p} ( \abs{K_o} \geq \lambda \log \abs{V} ) \leq \frac{1}{c \abs{V}^{ \frac{c \lambda}{n}} }.
\]
\end{cor}

\begin{proof}
Suppose that $\mathbb P_p(\abs{K_o} \geq n) \leq \eps$. We may assume that $\eps < 1/2$, otherwise the result is trivial. Define $h := \frac{1}{n}\log \frac{1}{1-\eps}$ and $q := 1-e^{-h}$. By a union bound,
\[
	\mu_{p,h} := \mathbb P_p \otimes \mathbb Q_{q} ( o \xleftrightarrow{\omega} g ) \leq \mathbb P_p ( \abs{K_o} \geq n) + \mathbb P_p \otimes \mathbb Q_q ( o \xleftrightarrow{\omega} g \mid \abs{K_o} < n ).
\]
We now bound these two terms individually. By hypothesis, $\mathbb P_p ( \abs{K_o} \geq n) \leq \eps$. By our choice of $h$ and $q$,
\[
	\mathbb P_p \otimes \mathbb Q_q ( o \xleftrightarrow{\omega} g \mid \abs{K_o} < n ) \leq 1 - e^{-hn} = \eps.
\]
Therefore $\mu_{p,h} \leq 2\eps$. So by \cref{lem:Vanneuville},
\[
	\mathbb P_{(1-2\eps) p} (\abs{K_o} \geq \lambda \log \abs{V}) \leq \frac{ \mathbb P_p( \abs{K_o} \geq \lambda \log \abs{V} ) } { 1 - 2\eps } e^{-h \lambda \log \abs{V}} \leq \frac{1}{1-2\eps} \abs{V}^{-\frac{\lambda}{n} \log \frac{1}{1-\eps} }.
\]
So the claim holds with $c := \min\left \{ 1-2\eps, \log \frac{1}{1-\eps} \right\}$.
\end{proof}

To convert the fact that the cluster at the origin is at least mesoscopic with good probability into a uniform point-to-point lower bound on a divergently large scale, we will apply Hutchcroft's volumetric two-arm bound \cite[Corollary 1.7]{hutchcroft2019locality}, stated below as \cref{thm:two-ghost}. This applies in our setting because every finite transitive graph is unimodular, and in this case we can trivially drop the hypothesis that at least one of the clusters is finite in the definition of $\mathcal T_{e,n}$. This tells us that it is always unlikely that the endpoints of a given edge belong to distinct large clusters.

\begin{thm}\label{thm:two-ghost}
Let $G$ be a unimodular transitive graph with degree $d$. There exists $C(d) < \infty$ such that for all $e \in E$, $n \geq 1$, and $p \in (0,1)$
\[
	\mathbb P_{p}( \mathcal T_{e,n} ) \leq C\left[ \frac{1-p}{p n} \right]^{1/2},
\]
where $\mathcal T_{e,n}$ is the event that the endpoints of $e$ belong to distinct clusters, each of which contains at least $n$ vertices, and at least one of which is finite.
\end{thm}

Hutchcroft showed in \cite{hutchcroft2019locality} that this can be used to convert volume-tail bounds into point-to-point bounds. This was also used in \cite{easo2023critical}. Here is the quantitative output of his argument, stated in the case of finite graphs.

\begin{cor}\label{cor:volume_tail_implies_two-point_tail}
Let $G$ be a finite transitive graph with degree $d$. There exists $C(d) < \infty$ such that for all $n,r \geq 1$ and $p \in (0,1)$,
\[
	\min_{u \in B_r }\mathbb P_{p}(o \leftrightarrow u) \geq \mathbb P_p( \abs{K_o} \geq n )^2 - \frac{ Cr }{ p^{ r+1 } n^{1/2} }.
\]
\end{cor}

\begin{proof}
Let $u \in B_r$. By Harris' inequality and a union bound,
\[
	\mathbb P_p( o \leftrightarrow u ) \geq \mathbb P_p( \abs{K_o} \geq n )^2 - \mathbb P_{p}( \abs{K_o} \geq n \text{ and } \abs{K_u} \geq n \text{ but } o \not\leftrightarrow u ). 
\]
The second term on the right can now be bounded by \cite[Lemma 2.6]{easo2023critical}. (In that lemma the hypothesis that $G$ is infinite and $p < p_c$ can be replaced by the hypothesis that $G$ is finite.)
\end{proof}

We now combine \cref{cor:superlog_implies_meso} and \cref{cor:volume_tail_implies_two-point_tail} to establish \cref{prop:large_cluster_to_local_connections}.

\begin{proof}[Proof of \cref{prop:large_cluster_to_local_connections}]

Fix $p \in (0,1)$, $\lambda \geq 1$, and $\eta > 0$. Let $\eps : = \frac{\eta}{4}$ and let $c_1\left( \frac{\eta}{4} \right) > 0$ be the corresponding constant from \cref{cor:superlog_implies_meso}. We may assume that $\eta \leq \frac{1}{2}$, $c_1 < 1$, and $\abs{V} > 1$. Suppose that $\mathbb P_p( \abs{K_1} \geq \lambda \log \abs{V} ) \geq \frac{1}{c \abs{V}^c}$. By a union bound, $\mathbb P_p( \abs{K_o} \geq \lambda \log \abs{V} ) \geq \frac{1}{c \abs{V}^{c+1}}$. Let $n := \frac{c_1 \lambda}{2}$. Since $\frac{c_1 \lambda}{n}-1 > c_1$ and $(1-2\eps)(p+\eta) \geq p$, it follows by \cref{cor:superlog_implies_meso} that
\[
	\mathbb P_{p+\eta}\left( \abs{K_o} \geq \frac{c_1 \lambda}{2} \right) \geq \frac{\eta}{4}.
\]
Let $C_1(d) < \infty$ be the constant from \cref{cor:volume_tail_implies_two-point_tail}.
Let $r \geq 1$ be arbitrary. Then by \cref{cor:volume_tail_implies_two-point_tail},
\[
	\min_{u \in B_{r}} \mathbb P_{p+\eta}( o \leftrightarrow u ) \geq \left(\frac{\eta}{4}\right)^2 - \frac{C_1r}{ (p+\eta)^{r+1} \left(\frac{c_1 \lambda}{2}\right) ^{1/2} }.
\]
Note that $r \leq \eta^{-r}$ because $\eta \leq \frac{1}{2}$. So there is a constant $C_2(d,\eta) < \infty$ such that
\[
	\frac{C_1r}{ (p+\eta)^{r+1} \left(\frac{c_1 \lambda}{2}\right) ^{1/2}  } \leq \frac{C_2}{\eta^{2r} \lambda^{1/2}}.
\]
Now there exists $c_2(d,\eta) > 0$ such that $r := c_2 \log( \lambda) - \frac{1}{c_2}$ satisfies $\frac{C_2}{\eta^{2r} \lambda^{1/2}} \leq \frac{\eta^2}{80}$. Then by our above work (when $r \geq 1$, otherwise the inequality anyway holds trivally),
\[
	\min_{u \in B_r} \mathbb P_{p+\eta}(o \leftrightarrow u)  \geq \left(\frac{\eta}{4}\right)^2 - \frac{\eta^2}{80} = \frac{\eta^2}{20}.
\]
Therefore the claim holds with $c := \min\{ c_1,c_2 \}$.
\end{proof}

\subsection{Local connections $\to$ large clusters} \label{subsec:local_connections_to_large_clusters}

In this subsection we prove the following proposition. This imples that if there is a uniform point-to-point lower bound on a divergently large scale, then the largest cluster contains much more than $\log \abs{V}$ vertices with high probability. This is a kind of converse to \cref{prop:large_cluster_to_local_connections}. This will be used in the next subsection to prove the equivalence of different notions of sharpness.

\begin{prop} \label{prop:local_connections_to_large_cluster}
Let $G$ be a finite transitive graph. For all $\delta > 0$ there exists $c(\delta) > 0$ such that for all $p \in (0,1)$ and $r \geq 1$ with $\abs{B_r} \leq \abs{V}^{1/10}$,
\[
	\min_{u \in B_r}\mathbb P_p\left( o \leftrightarrow u) \geq \delta \quad \implies \quad \mathbb P_p(\abs{K_1} \geq c \abs{B_r} \log \abs{V} \right) \geq 1 - \frac{1}{\abs{V}^{3/4}}.
\]
\end{prop}

The next lemma converts point-to-point connection lower bounds on one scale into volume-tail lower bounds on all scales. The idea is to approximately cover the graph by a large number of balls on which the point-to-point lower bound holds then glue together large clusters from multiple balls.

\begin{lem}\label{lem:two_point_bound_implies_volume_tail_at_all_scales}
Let $G$ be a finite transitive graph. For all $\delta > 0$ there exists $c(\delta) > 0$ such that for all $p \in (0,1)$ and $n,r \geq 1$ satisfying $n \leq \frac{c \abs{V}}{\abs{B_r}}$,
\[
	\min_{u \in B_r}\mathbb P_p( o \leftrightarrow u) \geq \delta \quad \implies \quad \mathbb P_p\left( \abs{K_o} \geq n \right) \geq c e^{-\frac{n}{c \abs{B_r} } }.
\]
\end{lem}

\begin{proof}
Fix $\delta > 0$, $p \in (0,1)$, and $n,r \geq 1$. Suppose that $\min_{u \in B_r}\mathbb P_p( o \leftrightarrow u) \geq \delta$. Let $W$ be a maximal (with respect to inclusion) set of vertices such that $o \in W$ and $\operatorname{dist}_G(u,v) \geq 2r$ for all distinct $u,v \in W$. Build a graph $H$ with vertex set $W$ by including the edge $\{u,v\}$ if and only if $\operatorname{dist}_G(u,v) \leq 5r$ and $u\not=v$. By maximality of $W$, the graph $H$ is connected. Let $T$ be a spanning tree for $H$. Let $f : W \backslash \{o\} \to W$ be a function encoding $T$ where `$f(u) = v$' means that the edge $\{u,v\}$ is present in $T$ and $\operatorname{dist}_{T}(o,v) < \operatorname{dist}_T(o,u)$. Extend this to a function $f:W \to W$ by setting $f(o) := o$. By Markov's inequality, every $u \in V$ satisfies $\mathbb P_{p}\left( \abs{K_u \cap B_r(u)} \geq \frac{\delta}{2} \abs{B_r} \right) \geq \frac{\delta}{2}$.\footnote{In this proof, $\mathbb P_p$ and $\abs{B_r}$ refer to $G$, not to $T$ or $H$.} By Harris' inequality, every edge $\{u,v\}$ in $H$ satisfies $\mathbb P_{p}( u\leftrightarrow v ) \geq \delta^5$. So by Harris' inequality again, for every $u \in W$, the event $A_u$ that $u \leftrightarrow f(u)$ and $ \abs{K_u \cap B_r(u)} \geq \frac{\delta}{2} \abs{B_r}$ satisfies $\mathbb P_{p}( A_u ) \geq \delta^5 \cdot \frac{\delta}{2} = \frac{\delta^6}{2}$.

Let $c(\delta) > 0$ be a small constant to be determined. Suppose that $n \leq \frac{c\abs{V}}{\abs{B_r}}$. By maximality of $W$, the balls $\{ B_{2r}(u) : u \in W \}$ cover $V$, and hence $\abs{V} \leq \abs{W} \cdot \abs{B_{2r}}$. So provided that $c$ is sufficiently small,
\[
	\abs{W} \geq \frac{\abs{V}}{\abs{B_{2r}}} \geq \frac{\abs{V}}{\abs{B_r}^2} \geq \frac{n}{c\abs{B_r}} \geq \frac{2 n}{\delta \abs{B_r}}.
\]
In particular, we can find a $T$-connected set of vertices $U \subseteq W$ such that $o \in U$ and $\abs{U} = \left \lceil \frac{2 n}{\delta \abs{B_r}} \right \rceil$. If $A_u$ holds for every $u \in U$ then $\abs{K_o} \geq \frac{\delta}{2} \cdot \abs{B_r} \cdot \abs{U} \geq n$. So by Harris' inequality, provided $c$ is sufficiently small,
\[
	\mathbb P_{p}\left( \abs{K_o} \geq n \right) \geq \mathbb P_{p}\left( \bigcap_{u \in U} A_u \right) \geq \left( \frac{\delta^6}{2} \right)^{\left \lceil \frac{2 n}{\delta \abs{B_r}} \right \rceil} \geq c e^{-\frac{n}{c\abs{B_r}}}. \qedhere
\]
\end{proof}

The following is a second-moment calculation for the number of vertices contained in large clusters.\footnote{We were inspired by a weaker (degree-dependent) version of this argument that arose during joint work with Hutchcroft towards \cite{easo2021supercritical2}, which was made redundant and thus did not appear in the final version of that work.} In the proof, it will be convenient to introduce partial functions to encode partially-revealed percolation configurations. Recall that a \emph{partial function} $f:A \rightharpoonup B$ is a function $A' \to B$ for some $A' \subseteq A$, i.e.\! for every $a \in A$, either $f(a) \in B$ or $f(a) = \text{`undefined'}$. We denote this set $A'$ on which $f$ is defined by $\operatorname{dom}(f)$. Given partial functions $f$ and $g$, the \emph{override} $f \sqcup g$ is the partial function with $\operatorname{dom}(f \sqcup g) = \operatorname{dom}(f) \cup \operatorname{dom}(g)$ that is equal to $f$ on $\operatorname{dom}(f)$ and is equal to $g$ on $\operatorname{dom}(g)\backslash \operatorname{dom}(f)$. We write $\op{Var}_p$ to denote the variance of a random variable under $\mathbb P_p$.

\begin{lem}\label{lem:concentration_of_small_cluster_density}
Let $G$ be a finite transitive graph. For all $n \geq 0$ and $p \in (0,1)$, the random set $X := \{ u \in V : \abs{K_u} \geq n \}$ satisfies
\[
	\operatorname{Var}_p \abs{X} \leq n^2 \cdot \mathbb E_p \abs{X}.
\]
\end{lem}

\begin{proof}
Let $\mathbb P$ be the joint law of a uniformly random automorphism of $G$, denoted $\phi$, and three configurations $\omega_1,\omega_2,\omega_3$ sampled according to $\mathbb P_p$, where all four of these random variables are independent. Given a configuration $\omega : E \to \{0,1\}$, let $\hat \omega : E \rightharpoonup \{0,1\}$ be the partial function encoding the edges revealed in an exploration of the cluster at $o$ from inside (with respect to an arbitrary fixed ordering of $E$) that is halted as soon as the event $\{ \abs{K_o(\omega)} \geq n \}$ is determined by the states of the revealed edges. Define $\omega:= \left(\hat \omega_1 \sqcup  \phi(\hat \omega_2)\right) \sqcup \omega_3$. By transitivity, the law of $\phi(o)$ is uniform on $V$, and by a standard cluster-exploration argument, $\mathbb P( \omega = \cdot \mid \phi ) = \mathbb P_p$ almost surely. This lets us rewrite $\op{Var}_p \abs{X}$ as
\begin{equation}\label{eq:variance_calc}\begin{split}
	\operatorname{Var}_p \abs{X} &=  \sum_{u,v} \left[\mathbb P_p( u,v \in X ) - \mathbb P_p( u \in X) \cdot \mathbb P_p( v \in X )\right]\\
	&=  \abs{V} \mathbb E_p \abs{X} \cdot \left[\frac{1}{\abs{V}} \sum_{ u } \mathbb P_p( u \in X \mid o \in X ) - \mathbb P_p(o \in X)\right]  \\
	& = \abs{V} \mathbb E_p \abs{X} \cdot \left[ \mathbb P\left( \phi(o) \in X(\omega) \mid o \in X(\omega) \right) - \mathbb P_p(o \in X)\right].
\end{split}\end{equation}

Consider the sets of vertices $A_1 := K_o( \hat \omega_1 )$ and $A_2 := K_o ( \hat \omega_2 )$, which are defined purely in terms of the open edges in $\hat \omega_1$ and $\hat \omega_2$ respectively, i.e.\! all edges with `undefined' state are treated as closed. Note that $o \in X(\omega)$ if and only if $o \in X( \omega_1)$. Moreover, given that $o \in X(\omega_1)$, if $\phi(o) \in X(\omega)$ then either $o \in X(\omega_2)$ or $A_1 \cap \phi(A_2) \not= \emptyset$. So by a union bound and independence,
\begin{equation} \label{eq:union_bound_in_var_calc}
	\mathbb P\left( \phi(o) \in X(\omega) \mid o \in X(\omega) \right) \leq \mathbb P_p(o \in X) + \mathbb P\left( A_1 \cap \phi(A_2) \not=\emptyset \mid o \in X(\omega_1) \right).
\end{equation}
In particular, by \cref{eq:variance_calc}, it suffices to verify that
\begin{equation} \label{eq:stp_in_var_calc}
	 \mathbb P\left( A_1 \cap \phi(A_2) \not=\emptyset \mid \omega_1,\omega_2,\omega_3 \right) \leq \frac{n^2}{\abs{V}} \qquad \text{a.s.}
\end{equation}

Consider arbitrary deterministic sets of vertices $B_1$ and $B_2$. By transitivity, the law of $\phi(u)$ for any fixed vertex $u$ is uniform over $V$. So by a union bound,
\[
	\mathbb P( B_1 \cap \phi(B_2) \not = \emptyset) \leq \sum_{u \in B_2} \mathbb P( \phi(u) \in B_1 ) = \sum_{u \in B_2} \frac{\abs{B_1}}{\abs{V}} = \frac{\abs{B_1} \abs{B_2}}{\abs{V}}.
\]
\Cref{eq:stp_in_var_calc} now follows by applying this to the sets $A_1$ and $A_2$, which almost surely satisfy $\abs{A_1},\abs{A_2} \leq n$.
\end{proof}

We now combine \cref{lem:two_point_bound_implies_volume_tail_at_all_scales,lem:concentration_of_small_cluster_density} to prove \cref{prop:local_connections_to_large_cluster}.

\begin{proof}[Proof of \cref{prop:local_connections_to_large_cluster}]
Let $\delta > 0$, $p \in (0,1)$, and $r \geq 1$. Suppose that $\abs{B_r} \leq \abs{V}^{1/10}$ and $\min_{u \in B_r} \mathbb P_p(o \leftrightarrow u) \geq \delta$. Let $c_1(\delta) > 0$ be the constant from \cref{lem:two_point_bound_implies_volume_tail_at_all_scales}. Let $n := c \abs{B_r} \log \abs{V}$ for a small constant $c(\delta)  >0$ to be determined. Since $\abs{B_r}^2 \leq \abs{V}^{2/10}$, provided $c$ is small,
\[
	n := c \abs{B_r} \log \abs{V} \leq \frac{c_1 \abs{V}}{\abs{B_r}}.
\]
So \cref{lem:two_point_bound_implies_volume_tail_at_all_scales} yields
\[
	\mathbb P_p\left( \abs{K_o} \geq n \right) \geq c_1 e^{ - \frac{n}{c_1 \abs{B_r}}} = c_1 \abs{V}^{ - \frac{c}{c_1} } \geq c_1 \abs{V}^{-1/100},
\]
provided $c$ is small. By transitivity, it follows that the random set $X := \{ u \in V : \abs{K_o} \geq n \}$ satisfies $\e_p \abs{X} \geq c_1 \abs{V}^{99/100}$. So by Chebychev's inequality and \cref{lem:concentration_of_small_cluster_density},
\[
	\mathbb P_p( \abs{K_1} \geq n ) = 1-\mathbb P_p( \abs{X} = 0 ) \geq 1 - \frac{ \op{Var}_p \abs{X} }{ \left(\e_p \abs{X}\right)^2 } \geq 1 - \frac{n^2}{ c_1 \abs{V}^{99/100} }.
\]
The conclusion follows because, provided $c$ is small,
\[
	\frac{n^2}{ c_1 \abs{V}^{99/100} } = \frac{\left(c \abs{B_r} \log \abs{V}\right)^2 }{ c_1 \abs{V}^{99/100} } \leq \frac{\left(c \abs{V}^{1/10} \log \abs{V}\right)^2 }{ c_1 \abs{V}^{99/100} } \leq \frac{1}{\abs{V}^{3/4}}. \qedhere
\]
\end{proof}

\subsection{Equivalent notions of sharpness} \label{subsec:equivalent_notions_of_sharpness}

In this subsection we apply results from earlier in \cref{sec:large_clusters_to_local_connections} to prove the following proposition. In the statement and the proof, we take for granted that $\mc G$ always admits a percolation threshold \cite{MR4665636}. Item 2 is analogous to the standard definition of sharpness for percolation on an infinite transitive graph. The fact that items 1 and 2 are equivalent is why we decided to label our version of ``sharpness'' for finite transitive graphs as such. Item 3 is analogous to the locality of the critical parameter for infinite transitive graphs. It is perhaps surprising that sharpness and locality are equivalent for finite graphs but not for infinite graphs. One way to make sense of this is that locality for infinite graphs is equivalent to a \emph{uniform} (in the choice of graph) version of sharpness for infinite graphs, and for finite graphs, the only meaningful notion of sharpness is necessarily uniform.

\begin{prop}\label{prop:equivalent_notions_of_sharpness}
For every infinite set $\mathcal G$ of finite transitive graphs with bounded degrees, the following are equivalent:
\begin{enumerate}
	\item Percolation on $\mc G$ has a sharp phase transition.
	\item For every subcritical sequence of parameters $p$, there exists a constant $C(\mc G, p)<\infty$ such that for all $G \in \mc G$ and all $n \geq 1$,
	\[
		\mathbb P_p \left( \abs{K_o} \geq n \right) \leq C e^{-n/C}.
	\]
	\item If an infinite subset $\mc H \subseteq G$ converges locally to an infinite transitive graph $H$, then the constant sequence $G \mapsto p_c(H)$ is the percolation threshold for $\mc H$.
\end{enumerate}
\end{prop}

We will prove that $3 \implies 2 \implies 1 \implies 3$. For the first step, we apply \cref{cor:superlog_implies_meso} and compactness.

\begin{proof}[Proof that item 3 implies item 2]
	Assume that item 3 holds. Our goal is to prove that item 2 holds. Since $\mc G$ has bounded degrees, $\mc G$ is relatively compact in the local topology. In particular, we may assume without loss of generality that $\mc G$ converges locally to some infinite transitive graph $G$.  (If item $2$ is false, then we can find an infinite subset $\mc H \subseteq \mc G$ such that item 2 is false for every sequence in $\mc H$.) Now fix a subcritical sequence $p$ for $\mc G$. By item 3, after passing to a tail of $\mc G$ if necessary, there exists a constant $\eps > 0$ such that $p(G) \leq (1-\eps)p_c(H)$ for every $G \in \mc G$. Pick $r \geq 1$ such that $\mathbb P_{(1-\eps/2) p_c(H)}^{H} ( \abs{K_o} \geq r ) \leq \eps/4$. By passing to a further tail of $\mc G$ if necessary, we may assume that $B_r^G \cong B_r^H$ for every $G \in \mc G$. Then $\mathbb P_{(1-\eps/2)p_c(H)}^{G} ( \abs{K_o} \geq r ) \leq \eps/4$ for every $G \in \mc G$. Let $c(\eps) > 0$ be the constant from \cref{cor:superlog_implies_meso}. For every $n \geq 1$ and $G \in \mc G$, \cref{cor:superlog_implies_meso} with $\lambda := \frac{n}{\log \abs{V}}$ tells us that
	\[
		\mathbb P_{(1-\eps) p_c(H)}^G( \abs{K_o} \geq n ) \leq \mathbb P_{(1-\eps/2)(1-2\cdot \eps/4) p_c(H)}^G( \abs{K_o} \geq n ) \leq \frac{1}{c e^{cn / r}}.
	\]
	Take $C := r/c$. The conclusion now follows by monotonicity because $p(G) \leq (1-\eps) p_c(H)$ for every $G \in \mc G$.
\end{proof}

The second step is a simple union bound.

\begin{proof}[Proof that item 2 implies item 1]
	Given a subcritical sequence $p$, let $C(\mc G, p) < \infty$ be the constant guaranteed to exist by item 2. Then for every $G \in \mc G$,
	\[
		\mathbb P_p( \abs{K_o} \geq 2C \log \abs{V} ) \leq C e^{- \frac{ 2C \log \abs{V} }{ C } } = \frac{C}{\abs{V}^2},
	\]
	and hence by a union bound,
	\[
		\mathbb P_p( \abs{K_1} \leq 2C \log \abs{V} ) \geq 1 - \abs{V}\frac{C }{ \abs{V}^{2}} = 1 -\frac{C}{\abs{V}}.
	\]
	So $\lim \mathbb P_p( \abs{K_1} \leq 2C \log \abs{V} ) = 1$, as required.
\end{proof}

We now turn to the third step, $1 \implies 3$. Fix a choice of percolation threshold $p_c : \mc G \to (0,1)$, and think of this as an extension of the usual critical points $p_c$ for percolation on the infinite transitive graphs that make up the boundary of $\mc G$. Then our goal is to show that, assuming item 1, the function $p_c$ is continuous as we approach the boundary of $\mc G$ from the interior. We split this into two parts: upper- and lower-semicontinuity. For lower-semicontinuity, we will apply a finite graph version of an argument of Pete \cite[Section 14.2]{Gabor}, which was based on the mean-field lower bound for infinite transitive graphs. For upper-semicontinuity, we will combine \cref{cor:volume_tail_implies_two-point_tail} and \cref{prop:local_connections_to_large_cluster}. 

\begin{proof}[Proof that item 1 implies item 3]
	Suppose for contradiction that $\mc H \subseteq G$ is an infinite subset that converges locally to some infinite transitive graph $H$, but the constant sequence $G \mapsto p_c(H)$ is not a percolation threshold for $\mc H$. By passing to a subsequence, we may assume without loss of generality that there is a constant $\eps > 0$ such that either $p_c(G) \leq (1-\eps) p_c(H)$ for every $G \in \mc H$, or $p_c(G) \geq (1+\eps)p_c(H)$ for every $G \in \mc H$. Call these Case 1 and Case 2, corresponding to (a violation of) lower- and upper-semicontinuity in our discussion above.

	\paragraph{(Case 1)}
	Since $p_c : \mc G \to (0,1)$ is a percolation threshold, there exists a constant $\delta > 0$ such that $\mathbb P_{(1+\eps)p_c(G)}^G( \abs{K_o} \geq \delta \abs{V} ) \geq \delta$ for every $G \in \mc H$. So by monotonicity, $\mathbb P_{(1-\eps^2)p_c(H)}^G( \abs{K_o} \geq \delta \abs{V} ) \geq \delta$ for every $G \in \mc H$. For every $r \geq 1$, there exists $G \in \mc H$ such that $\delta \abs{V} > \abs{B_r^H}$ and $B_r^G \cong B_r^H$, and hence
	\[
		\mathbb P_{(1-\eps^2) p_c(H)}^{H} ( \abs{K_o} \geq r ) \geq \mathbb P_{(1-\eps^2) p_c(H)}^{G} ( o \leftrightarrow S_r ) \geq \mathbb P_{(1-\eps^2)p_c(H)}^G( \abs{K_o} \geq \delta \abs{V} ) \geq \delta.
	\]
	In particular, $\mathbb P_{(1-\eps^2) p_c(H)}^{H} ( \abs{K_o} = \infty ) > 0$, a contradiction. 

	\paragraph{(Case 2)} Let $A \geq 1$ be a given arbitrary constant. It suffices to prove that the parameter $p:=(1+\eps/2)p_c(H)$ satisfies $\lim_{\mc G} \mathbb P_p^G(\abs{K_1} \geq A \log \abs{V}) = 1$. Set $\delta := \mathbb P_p^H( o \leftrightarrow \infty ) > 0$. Let $d$ be the vertex degree of $H$, and note that $p_c(H) > 1/d$, as this is well-known to hold for every infinite transitive graph. So by \cref{cor:volume_tail_implies_two-point_tail}, there is a constant $C(d) < \infty$ such that for all $n,r \geq 1$ and all $G \in \mc H$ with $B_n^{G} \cong B_n^H$,
	\[
		\min_{u \in B_r^G} \mathbb P_p^G( o \leftrightarrow u ) \geq \delta^2 - \frac{C r d^{r+1}}{n^{1/2}},
	\]
	and in particular, (using that $r \leq d^r$ for all $r \geq 1$) the radius $r(n) := \log_d \left(\frac{ \delta^2 n^{1/2} } { 2C d^2 }\right)$ satisfies
	\[
		\min_{u \in B_{r(n)}^G} \mathbb P_p^G( o \leftrightarrow u ) \geq \delta^2 - \frac{\delta^2}{2} = \frac{\delta^2}{2}.
	\]
	Let $c(\delta^2/2) > 0$ be the constant from \cref{prop:local_connections_to_large_cluster}. Fix $n$ sufficiently large that $c \cdot r(n) \geq A$. By passing to a tail of $\mc H$ if necessary, let us assume that $B_n^G \cong B_n^H$ and $\lvert B_{r(n)}^G \rvert \leq \abs{V}^{1/10}$ for every $G \in \mc H$. Then by \cref{prop:local_connections_to_large_cluster}, for all $G \in \mc H$,
	\[
		\mathbb P_p^G\left( \abs{K_1} \geq c \lvert B_{r(n)}^G \rvert \log \abs{V} \right) \geq 1 - \frac{1}{\abs{V}^{3/4}}.
	\]
	In particular, since $c \lvert B_{r(n)}^H \rvert \geq c r(n) \geq A$, we deduce that $\lim_{\mc G} \mathbb P_p^G ( \abs{K_1} \geq A \log \abs{V} ) = 1$ as required.
\end{proof}


\section{Local connections $\to$ global connections} \label{sec:local_connections_to_global_connections}

In this section, we will apply the proof from \cite{easo2023critical} that the critical point for percolation on (non-one-dimensional) infinite transitive graphs is local. As explained in the introduction, we need to both make this argument more finitary and adapt it to finite transitive graphs. We can roughly think of the proof of \cite{easo2023critical} in two parts: First, if $G$ does not satisfy certain geometric properties around scale $n$, which include that $G$ is finitely-ended, then $G$ must satisfy a certain statement $\mathcal I_n$ about the propagation of connection bounds around scale $n$. Second, if $G$ does satisfy these geometric properties and $G$ is one-ended, then $G$ again satisfies $\mathcal I_n$. Together, these two parts imply that if $G$ does not satisfy $\mathcal I_n$, then $G$ must actually be two-ended and hence one-dimensional. By looking at the proof of the second part, we can pinpoint where one-endedness is used, namely as a hypothesis in \cite[Lemma 5.8]{easo2023critical}.

\cite[Lemma 5.8]{easo2023critical} concerns certain $(o,\infty)$-cutsets called \emph{exposed spheres}. The lemma says that if $G$ satisfies nice geometric properties around scale $n$ and is one-ended, then the exposed spheres around scale $n$ are in some sense well-connected. We took this from \cite[Lemma 2.1 and 2.7]{contreras2022supercritical}, where the authors deduced it from a theorem of Babson and Benjamini \cite{MR1622785}. By reading Timar's proof \cite{timar2007cutsets} of this theorem of Benjamini and Babson, we see that if an exposed sphere is not well-connected, then not only is $G$ multiply-ended, but this is actually witnessed by the exposed sphere itself in the sense that its removal from $G$ would create multiple infinite components. From this we can conclude that $G$ must in fact start to look one-dimensional from around scale $n$. This is how we will make this step from \cite{easo2023critical} finitary. To adapt the argument to finite transitive graphs, we will additionally need to introduce the notion of the exposed sphere in a finite transitive graph and prove that finite transitive graphs can, for the purpose of part of our argument, be treated like infinite transitive graphs that are one-ended.

Unfortunately, this application of Babson-Benjamini is deeply embedded in the proof of \cite{easo2023critical} as it is currently written. So it will take some work to restructure the multi-scale induction in \cite{easo2023critical} to isolate the relevant part. To avoid repetition, we have deferred the details of arguments that are implicit in \cite{easo2023critical} to the appendix, thereby keeping many of the arguments in this section high-level. Ultimately we will prove the following proposition, which contains this finite-graph finitary refinement of locality. While we have written this for finite graphs, the same argument yields the analogous finitary refinement for infinite graphs.

\begin{prop}\label{prop:local_to_global_connections}
Let $G$ be a finite transitive graph with degree $d$. Define
\[
	\gamma :=  \op{dist}_{\mathrm{GH}}\left(\frac{\pi}{\op{diam} G}G,S^1 \right) \cdot \op{diam}G \qquad \text{and} \qquad \gamma^+ := e^{(\log \gamma)^9}.
\]
For all $\eps,\eta > 0$ there exists $\lambda(d,\eps,\eta) < \infty$ such that for all $p \in (0,1)$,
\[
	\min_{u \in B_\lambda} \mathbb P_p( o \leftrightarrow u ) \geq \eta \quad \implies \quad \min_{u \in B_{ \gamma^+ } } \mathbb P_{p+\eps}( o \leftrightarrow u) \geq e^{-(\log \log \gamma^+ )^{1/2}}.
\]
\end{prop}

This proof of \cref{prop:local_to_global_connections} is by induction. In \cref{subsec:logic_of_induction}, we describe the high-level structure of this induction, which is essentially the same as in \cite[Section 3.2]{easo2023critical}, except for two differences. The first difference is that we have reworded the induction to say ``we can keep propagating until and unless we reach a scale where the geometry is bad'', with a separate lemma that says ``if the geometry is bad at scale $n$, then $G$ starts to look one-dimensional from around scale $n$''. In contrast, the induction in the earlier work simply says ``if $G$ is not one-dimensional, then we can keep propagating forever''. The second difference is that the induction in the earlier work is slightly coarser in the sense that it groups multiple inductive steps of the argument we present here into a single inductive step. The additional detail in the present version is necessary to close the gap between the last scale from which we can propagate connection bounds and the first scale at which we can prove that the geometry ``is bad''.

The individual inductive steps are all implicit in \cite[Sections 4 and 6]{easo2023critical}. We will justify these in the appendix. In \cref{subsec:base_case_of_the_induction}, we will prove something like the ``base case" of the induction. This follows by a compactness argument from some intermediary results in \cite{easo2023critical} and \cite{contreras2022supercritical}. In \cref{subsec:degenerate_geometry}, we will prove that ``if the geometry is bad at scale $n$, then $G$ starts to look one-dimensional from around scale $n$''. This subsection is a refinement of \cite[Section 5]{easo2023critical}, but for the reasons discussed, it will require some new ideas.

\subsection{The logic of the induction} \label{subsec:logic_of_induction}

For the entirety of this subsection, fix a finite transitive graph $G$ with degree $d$, and define $\gamma$ as in the statement of \cref{prop:local_to_global_connections}. We will describe the repeated-sprinkling multi-scale induction argument used to prove \cref{prop:local_to_global_connections} (which is adapted from \cite{easo2023critical}) as a deterministic colouring process evolving over time. At every time $t \in \mathbb R$, every scale\footnote{It would have been more natural to consider scales $n \in \mb N$ rather than $n \in [3,\infty)$. We chose the latter to avoid rounding issues and so that $\log \log n$ is always positive.} $n \in [3,\infty)$ can be coloured orange or green (or both, or neither - i.e.\! uncoloured), encoding a statement\footnote{Formally, this colouring can be encoded as a function $\operatorname{colour}: [3,\infty) \times \mathbb R \to \mathcal P( \{ \mathrm{orange}, \mathrm{green} \} )$, where $\mathcal P(X)$ means the powerset of $X$. We say ``$n$ is green at time $t$'' to mean that $\operatorname{colour}(n,t) \ni \operatorname{orange}$. Similar statements are formalised analogously.} about the connectivity properties of percolation of parameter\footnote{This choice of parameterisation appears implicitly in \cite{easo2023critical} as the natural choice for arguments that involve repeated sprinkling. Indeed, our function $\phi$ is the function $\op{Spr}(p;\lambda)$ from \cite[Section 3.1]{easo2023critical} evaluated at $(1/2;t)$.} $\phi(t):= 1-2^{-e^t}$ over distances of approximately $n$. To lighten notation, let $\delta(n) := e^{-(\log \log n)^{1/2}}$ denote the standard small-quantity associated to each scale $n$. Now we colour a scale $n$ \emph{orange} at time $t$ to mean that
\[
	\min_{u \in B_n} \mathbb P_{\phi(t)}( o \leftrightarrow u ) \geq \delta(n).
\]
We also define the move-right (aka increase-scale) function $R : n \mapsto e^{(\log n)^{9}}$, and write $R^k := R\circ\ldots \circ R$ for the $k$-fold composition of $R$ with itself. Now to prove \cref{prop:local_to_global_connections}, it suffices to prove the following lemma.

\begin{lem}\label{lem:orange_up_to_circleness}
For all $\eps > 0$ and $n_1 \geq 3$ there exists $n_2(d,\eps,n_1) < \infty$ such that for all $t \in \mb R$ with $t \leq \frac{1}{\eps}$, if $[n_1,n_2]$ is orange at time $t$, then $R(\gamma)$ is orange at time $t+\eps$. 
\end{lem}
\begin{prooff}[Proof of \cref{prop:local_to_global_connections} given \cref{lem:orange_up_to_circleness}]
Fix $\eps,\eta > 0$. Let $n_1(\eta)$ be the smallest integer satisfying $n_1 \geq 3$ and $\delta(n_1) \leq \eta$. Let $\alpha(\eps) > 0$ be the unique real satisfying $\phi(1/\alpha) = 1 - \eps$. Let $n_2(d,  \alpha \wedge \eps,n_1) < \infty$ be the constant that is guaranteed to exist by \cref{lem:orange_up_to_circleness}. We claim that we can take $\lambda := n_2$. Indeed, let $p \in (0,1)$ and suppose that $\min_{u \in B_{n_2}} \mathbb P_p( o \leftrightarrow u ) \geq \eta$. Define $t := \phi^{-1}(p)$. The claim is trivial if $p \geq 1 - \eps$, so we may assume that $t \leq 1/\alpha$. By monotonicity of the function $\delta(\cdot)$, the interval $[n_1,n_2]$ is orange at time $t$.  So by applying \cref{lem:orange_up_to_circleness}, $R(\gamma)$ is orange at time $t+\eps$. Since (by calculus) $\phi$ is $1$-Lipschitz, $R(\gamma)$ is also orange at time $\phi^{-1}(p+\eps)$, which is the required conclusion.
\end{prooff}

We say that a set $M \subseteq \mathbb N$ is a certain colour if every $m \in M$ is that colour. Given a statement $A$ about a colouring at an implicit time $t$, we define $s(A) := \inf\{ t :  A \text{ is true at time }t \}$ where $\inf \emptyset := +\infty$. For example,
\[
	s\left(\{10,12\} \text{ is orange}\right) := \inf\{ t :  \text{$10$ and $12$ are both orange at time } t\} \in [-\infty,+\infty].
\]
As a first approximation to our induction, imagine we knew that for every scale $n$ with $n \leq \gamma$,
\begin{equation} \label{eq:imaginary_orange_propagation}
	s(\text{$R(n)$ is orange}) \leq s(\text{$n$ is orange}) + \delta(n),
\end{equation}
Suppose for simplicity that some positive integer $r$ satisfies $R^r(n_2) = \gamma$ and that $n_2$ exceeds some large universal constant. Then by repeatedly applying \cref{eq:imaginary_orange_propagation}, we could deduce that
\[\begin{split}
	s\left( R(\gamma) \text{ is orange} \right) - s\left( [n_1,n_2] \text{ is orange} \right) &\leq \sum_{k=0}^{r} \left[ s\left( R^{k+1}(n_2) \text{ is orange} \right) - s\left( R^k(n_2) \text{ is orange} \right) \right] \\
	&\leq \sum_{k=0}^{r} \delta\left( R^k( n_2 ) \right) \leq \sum_{k=0}^{\infty} e^{- \left( 9^k \log \log n_2 \right)^{1/2} } \leq 2 \delta(n_2). 
\end{split}\]
Since $\delta(n_2) \to 0$ as $n_2 \to \infty$, this would certainly imply \cref{lem:orange_up_to_circleness}. Rather than prove something as direct as \cref{eq:imaginary_orange_propagation}, we will have to bring into play a new colour, \emph{green}.

Given a finite path $\nu = (\nu_i)_{i=0}^{k}$, we write $\op{start}(\nu) := \nu_0$ and $\op{end}(\nu) := \nu_k$ for its start and end vertices, $\abs{\nu}:= k$ for its length, and given $r  \geq 0$, we write $B_r(\nu) := \bigcup_{i=0}^k B_r(\nu_i)$ for the associated tube. Given $m,n \geq 1$ and $p \in (0,1)$, define the \emph{corridor function} (which we take from \cite{contreras2022supercritical}),
\[\kappa_p(m,n) := \inf_{\nu : \abs{\nu} \leq m} \mathbb P_p \left( \op{start} (\nu) \xleftrightarrow{\omega \cap B_n(\nu)} \op{end}(\nu) \right).\]
Notice that we always have $\kappa_{p}(m,n) \leq \min_{u \in B_m} \mathbb P_p( o \leftrightarrow u )$. Let us also define the set of low-growth scales $\mathbb L := \{n \geq 3 : \operatorname{Gr}(n) \leq e^{(\log n)^{100}}\}$. Now we colour a scale $n$ \emph{green} at time $t$ to mean that $n$ is orange and either $n \not\in \mb L$ or $\kappa_{\phi(t)}(R^2(n),n) \geq \delta(R(n))$. We will use this new colour to help us propagate orange by controlling the time taken for orange scales to turn green and for green scales to turn nearby scales orange.

The next lemma says that green scales quickly turn nearby scales orange. If we see that $n$ is green at some time $t$ but $n\in \mb L$, then $\kappa_{\phi(t)}(R^2(n),n) \geq \delta(R(n))$, which trivially implies that $[R(n),R^2(n)]$ is already orange. So the content of this lemma is that if instead $n \not\in \mb L$, then we can efficiently propagate a point-to-point connection lower bound from scale $n$ to scales in $[R(n),R^2(n)]$. The proof of this is implicit in \cite[Section 4]{easo2023critical}. The argument uses some ghost-field technology that works more efficiently around scales $n$ where $\op{Gr}(n)$ is large. See the appendix for details.

\begin{lem}\label{lem:green_propagates_orange}
There exists $n_0(d) < \infty$ such that for all $n \geq n_0$,
\[
	s(\text{$[R(n),R^2(n)]$ is orange}) \leq s(\text{$n$ is green}) + \delta(n).
\]
\end{lem}

The next lemma sometimes lets us control how long it takes for a scale $n$ to become green after turning orange. If $n \not\in \mathbb L$, then this time is trivially zero. So it suffices to consider $n \in \mb L$. We might hope for a statement like the following: there exists $n_0(d) < \infty$ such that for all $n \geq n_0$ with $n \in \mb L$,
\begin{equation} \label{eq:optimistic_orange_to_green}
	s( n \text{ is green}) \leq s(n \text{ is orange}) +\delta(n).
\end{equation}
Our next lemma is less satisfying in two ways.\footnote{We need to use the non-one-dimensionality hypothesis somewhere.} First, we can only prove an upper bound like \cref{eq:optimistic_orange_to_green} when $n$ belongs to a particular distinguished subset $\mb T(c,\lambda)$ of $\mb L$. Second, our upper bound is in terms of a mysterious quantity $\Delta$ rather than something explicit like $\delta(n)$. So to use this lemma, we will need to: (1) Deal with scales $n \in \mb L \backslash \mb T(c,\lambda)$, and (2) Find a way to upper bound $\Delta$ explicitly. For completeness, we will now define $\mb T(c,\lambda)$ and $\Delta$, but the reader should feel free to skip these definitions for now because they are not necessary to follow the high-level induction argument being developed in this subsection.

Given constants $c,\lambda > 0$, let $\mathbb T(c,\lambda)$ be the set of scales $n \in \mathbb L$ such that $G$ has $(c,\lambda)$-polylog plentiful tubes at every scale in an interval of the form $[m,m^{1+c}]$ that is contained in $[n^{1/3},n^{1/(1+c)}]$. We will recall the definition of plentiful polylog tubes, taken from \cite[Section 5]{easo2023critical}, in \cref{subsec:degenerate_geometry}. Given $m,n \geq 1$, we define $\op{Piv}[m,n]$ to be the event that in the restricted configuration $\omega \cap B_n$, there are at least two clusters that each contain an open path from $B_m$ to $S_n$. For each scale $n$ and time $t$, let $U_t(n)$ be the \emph{uniqueness zone} defined to be the maximum integer $b \leq \frac{1}{8}n^{1/3}$ satisfying $\mathbb P_{\phi(t)}(\operatorname{Piv}[4b,n^{1/3}]) \leq (\log n)^{-1}$. The associated \emph{cost} is
\[
	\Delta_t(n) := \left[\frac{\log \log n}{(\log n) \wedge \log \operatorname{Gr}(U_t(n))}\right]^{1/4}.
\]
Note that the cost is small if $\operatorname{Gr}(U_t(n)) \geq (\log n)^C$ for a big constant $C$. The proof of the next lemma is implicit in \cite[Section 6]{easo2023critical}, where, together with Hutchcroft, we used plentiful tubes to run an orange-peeling argument inspired by the one in \cite{contreras2022supercritical}. See the appendix for details.

\begin{lem}\label{lem:green_to_orange_via_K_Delta}
	For all $c > 0$ there exist $\lambda(d,c),n_0(d,c),K(d,c)< \infty$ such that the following holds for all $n \geq n_0$ with $n \in \mathbb T(c,\lambda)$. For all $t \in \mb R$, if $n$ is orange at time $t$ and $K\Delta_t(n) \leq 1$ then
	\[
		s(\text{$n$ is green}) \leq t + K \Delta_t(n).
	\]
\end{lem}

We now turn to the problem of finding an explicit upper bound on the cost $\Delta = \Delta_t(n)$ of a scale $n$ at a time $t$. Define the move-left (aka decrease-scale) function $L : n \mapsto (\log n)^{1/2}$. The next lemma provides such an upper bound if the much smaller scale $L(n)$ happens to already be green at time $t$. The proof of this is implicit in \cite[Section 6.3]{easo2023critical}. Notice that to upper bound $\Delta_t(n)$ is to lower bound $\op{Gr}(U_t(n))$. It is easy to check that in the setting of this lemma, $U_t(n) \geq L(n)$. The proof of the lemma establishes that when $L(n)$ is green and $n \in \mb L$, either $\op{Gr} (L(n))$ is big (as a function of $L(n)$) or we can find a better lower bound on $U_t(n)$ than the trivial bound that is $L(n)$. See the appendix for details.

\begin{lem}\label{lem:green_at_log_implies_small_delay}
There exists $n_0(d) < \infty$ such that the following holds for all $n \in \mathbb L$ with $n \geq n_0$. For all $t \in \mathbb R$, if $L(n)$ is green at time $t$ then
\[
	\Delta_t(n) \leq \frac{1}{\log \log n}.
\]
\end{lem}

\Cref{lem:green_to_orange_via_K_Delta,lem:green_at_log_implies_small_delay} together provide an explicit upper bound on the time it takes for a scale $n$ that is orange to become green if the much smaller scale $L(n)$ happens to already be green, at least until we encounter a scale $n \in \mb L \backslash \mb T(c,\lambda)$. Of course, this says nothing about how long we have to wait for for at least one orange scale to become green in the first place. We will return to this shortly, but for now, consider the following method for rapidly propagating orange once we have a big interval of green. Suppose that at some time $t$, some interval of the form $[ L(n), n ]$ is green. By \cref{lem:green_propagates_orange}, since $R^{-1}(n) \in [L(n),n]$, we will not have to wait long for $[n,R(n)]$ to turn orange. By \cref{lem:green_to_orange_via_K_Delta,lem:green_at_log_implies_small_delay}, since $L(m) \in [ L(n) , n ]$ for every $m \in [n,R(n)]$, we will not then have to wait long for $[n,R(n)]$ to turn green. By \cref{lem:green_propagates_orange} again, we will not then have to wait long for $[R(n),R^2(n)]$ to turn orange, and so forth. We can repeat this indefinitely until and unless we encounter a scale $n \in \mb L \backslash \mathbb T (c,\lambda)$.

In conjunction with \cref{lem:green_to_orange_via_K_Delta}, the next lemma lets us control long it takes to get this big interval of green in the first place, starting from an even bigger interval of orange. We prove this in \cref{subsec:base_case_of_the_induction}. We will use a compactness argument to reduce this to an analogous statement about an arbitrary infinite unimodular transitive graph $G$, which is then addressed by results in either \cite{contreras2022supercritical} or \cite{easo2023critical} according to whether $G$ has polynomial or superpolynomial growth.

\begin{lem}\label{lem:orange_interval_yields_low-cost_block}
For all $\eps > 0$ and $n_1 \geq 3$, there exists $n_2(d,\eps,n_1) < \infty$ such that the following holds for all $t \in \mb R$ with $t \leq \frac{1}{\eps}$. If $[n_1,n_2]$ is orange at time $t$, then for some $m$ satisfying $I:= [ L(m) , m ] \subseteq [n_1,n_2]$, we have
\[
	\sup_{n \in I \cap \mathbb L} \Delta_{t+\eps}(n) \leq \eps.
\]
\end{lem}

At this point, the lemmas we have accumulated allow us to rapidly propagate orange, starting from a big interval of orange, until and unless we encounter a scale $n \in \mb L \backslash \mb T(c,\lambda)$. To prove \cref{lem:orange_up_to_circleness}, we need this propagation to keep going until we encounter the scale $\gamma$. The next lemma lets us ensure that we will encounter $\gamma$ before we encounter $\mb L \backslash \mb T(c,\lambda)$. We will prove this in \cref{subsec:degenerate_geometry}. This is the analogue in our setting of \cite[Section 5]{easo2023critical}. While the random walk arguments that make up \cite[Subection 5.2]{easo2023critical} work equally well in our setting, the geometric arguments in \cite[Subsection 5.1]{easo2023critical} will require some new ideas.

\begin{lem}\label{lem:degenerate_geometry_implies_circleness}
	There exist $c(d) > 0$ such that for all $\lambda \geq 1$, there exists $n_0(d,\lambda) < \infty$ such that \[\inf \{ n \in \mb L \backslash \mb T(c,\lambda) : n \geq n_0 \} \geq \gamma.\]
\end{lem}

Let us conclude by formalising the above sketch of the fact that \cref{lem:orange_up_to_circleness}, which we know implies \cref{prop:local_to_global_connections}, can be reduced to the rest of the lemmas introduced in this subsection.

\begin{proof}[Proof of \cref{lem:orange_up_to_circleness} given \cref{lem:green_propagates_orange,lem:green_to_orange_via_K_Delta,lem:green_at_log_implies_small_delay,lem:orange_interval_yields_low-cost_block,lem:degenerate_geometry_implies_circleness}]

Fix $\eps > 0$ and $n_1 \geq 3$. We may assume that $\eps < 1$. Let $c(d) > 0$ be the constant from \cref{lem:degenerate_geometry_implies_circleness}. Let $\lambda(d),u_0(d),K(d)$ be the constants ``$\lambda(d,c),n_0(d,c),K(d,c)$" from \cref{lem:green_to_orange_via_K_Delta} for this choice of $c$. We may assume that $K \geq 1$. Let $u_1(d)$ be the constant ``$n_0(d)$'' from \cref{lem:green_propagates_orange}. Let $u_2(d)$ be the constant ``$n_0(d)$'' from \cref{lem:green_at_log_implies_small_delay}. Let $u_3(d)$ be the constant ``$n_0(d,\lambda)$'' from \cref{lem:degenerate_geometry_implies_circleness}, with the above choice of $\lambda$. Note that $\sum_i \delta  \left(R^i(3) \right)< \infty$ and $\sum_i \frac{1}{\log \log R^i(3)} < \infty$. Let $i_0(d,\eps)$ be the smallest non-negative integer such that
\[
	\sum_{i=i_0}^{\infty} \delta  \left(R^i(3) \right)\leq \frac{\eps}{5} \quad \text{and} \quad \sum_{i=i_0}^{\infty} \frac{1}{\log \log R^i(3)}  \leq \frac{\eps}{5K},
\]
and set $u_4(d,\eps) := R^{i_0}(3)$. Set $u_5(d,\eps,n_1) := \max \{u_0, u_1,u_2,u_3 , u_4 , n_1\}$. Let $u_6(d,\eps,n_1)$ be the constant ``$n_2\left( d, \frac{\eps}{5K} , u_5 \right)$" from \cref{lem:orange_interval_yields_low-cost_block}. We claim that the conclusion holds with $n_2(d,\eps,n_1) : = R( u_6)$.

Let $t \in \mb R$ with $t \leq \frac{1}{\eps}$, and suppose that $[n_1,n_2]$ is orange at time $t$. By \cref{lem:orange_interval_yields_low-cost_block}, there exists $m$ with $I := [L(m) , m] \subseteq [u_5,u_6]$ such that $\sup_{n \in I \cap \mb L} K \Delta_{t + \frac{\eps}{5}}(n) \leq \frac{\eps}{5} \leq 1$. Consider the possibility that $I \cap \left(\mb L \backslash \mb T(c,\lambda)\right) \not= \emptyset$. Then by \cref{lem:degenerate_geometry_implies_circleness}, $\gamma \leq u_6$. In particular, $R(\gamma) \leq n_2$, and hence $R(\gamma)$ is already orange at time $t$. Since we are trivially done in that case, let us assume to the contrary that $I \cap \left(\mb L \backslash \mb T(c,\lambda)\right) = \emptyset$. Then by \cref{lem:green_to_orange_via_K_Delta},
\[
	s\left( I \text{ is green} \right) \leq t + \frac{\eps}{5} +  \sup_{n \in I \cap \mb L} K \Delta_{t + \frac{\eps}{5}}(n) \leq t + \frac{2\eps}{5}. 
\]

Let $k$ be the largest non-negative integer such that $R^k(m) < \gamma$. (We may assume that such an integer exists, otherwise $\gamma \leq u_6$ and hence we are trivially done as above.) We claim that for all $i \in \{0,\ldots,k-1\}$,
\begin{equation} \label{eq:propagate_green_one_increment}
	s\left( [L(m) , R^{i+1}(m) ] \text{ is green} \right) \leq s\left( [ L(m) , R^{i}(m) ] \text{ is green} \right) + \delta \left(R^{i-1}(m)\right) + \frac{K}{ \log \log R^i(m) }.
\end{equation}
Indeed, fix an arbitrary index $i \in \{ 0, \ldots , k-1\}$ and an arbitrary time $s \in \mb R$ at which $[L(m) , R^{i}(m) ]$ is green. By \cref{lem:green_propagates_orange}, the interval $\left[R^{i}(m),R^{i+1}(m)\right]$ is orange at time $s+\delta\left( R^{i-1}(m) \right)$. By \cref{lem:green_at_log_implies_small_delay}, since $L\left( \left[ R^i(m), R^{i+1}(m) \right] \right) \subseteq \left[L(m) , R^{i}(m) \right]$,
\[
	\sup_{ n \in \left[ R^i(m) , R^{i+1}(m) \right] \cap \mb L } \Delta_{ s+\delta\left( R^{i-1}(m) \right) }(n) \leq \frac{1}{\log \log R^{i}(m) }.
\]
By \cref{lem:degenerate_geometry_implies_circleness}, we know that $\left[ R^i(m) , R^{i+1}(m) \right] \cap \left( \mb L \backslash \mb T(c,\lambda) \right) = \emptyset$, and since $R^i(m) \geq u_4$, we know that $\frac{K}{ \log \log R^i(m) } \leq \frac{\eps}{5} \leq 1$. So by \cref{lem:green_to_orange_via_K_Delta},
\[
	s\left( \left[ R^i(m) , R^{i+1}(m) \right] \text{ is green}  \right) \leq s + \delta \left( R^{i-1}(m) \right) + \frac{K}{\log \log  R^i(m) },
\]
establishing \cref{eq:propagate_green_one_increment}. By repeated applying \cref{eq:propagate_green_one_increment}, it follows by induction that
\[
	s\left( \left[ L(m) , R^{k}(m) \right] \text{ is green} \right) \leq t + \frac{2 \eps}{5} + \sum_{i=0}^{\infty} \left( \delta\left( R^{i-1}(m) \right) + \frac{K} { \log \log R^{i}(m) } \right) \leq t + \frac{4 \eps}{5},
\]
where in the second inequality we used the fact that $R^{-1}(m) \geq u_4$. By maximality of $k$, we know that $R^{-1}(\gamma) \in \left[ L(m) , R^{k}(m) \right]$. So by \cref{lem:green_propagates_orange},
\[\begin{split}
	s\left( R(\gamma) \text{ is orange} \right) - s\left( \left[ L(m) , R^{k}(m) \right] \text{ is green} \right) &\leq
	s\left( \left[ \gamma , R(\gamma) \right] \text{ is orange} \right) - s\left( R^{-1}(\gamma) \text{ is green} \right)\\
	&\leq \delta\left( R^{-1}(\gamma) \right) \leq \delta\left( u_4 \right) \leq \frac{\eps}{5}.
\end{split}\]
Therefore, as required,
\[
	s\left( R(\gamma) \text{ is orange} \right) \leq t + \frac{4\eps}{5} + \frac{\eps}{5} = t + \eps. \qedhere
\]
\end{proof}

\subsection{Base case of the induction} \label{subsec:base_case_of_the_induction}
In this subsection we prove \cref{lem:orange_interval_yields_low-cost_block}. By a compactness argument, we will reduce this to the following simpler statement about individual infinite transitive graphs. Although we defined $\Delta_t(n)$ and $\mb L$ in the context of \emph{finite} transitive graphs, let us use the exact same definitions for infinite transitive graphs.

\begin{lem}\label{small_cost_on_infinite_graphs}
Let $G$ be a unimodular infinite transitive graph. For every $t \in \mb R$ with $\phi(t) > p_c(G)$,
\[
	\lim_{n \to \infty}  \sup_{s \geq t} \Delta_s(n) \mathbbm{1}_{\mb L}(n) = 0.
\]
\end{lem}

Our first goal is to prove this lemma. As mentioned earlier, to show that $\Delta_t(n)$ is small, we need to show that $\op{Gr}(U_t(n)) \geq (\log n)^{C}$ for a large constant $C$. The following lemma\footnote{In the version in \cite{easo2023critical}, we also required that $n$ is larger than some constant depending on $d,\eta,\eps$, but that is redundant.} from \cite[Corollary 2.4]{easo2023critical} tells us in particular that if $\op{Gr}(n)$ is not too big with respect to $n$, then the uniqueness zone for $n$ is always at least of order $\log n$. The proof of this result was essentially already contained in \cite{contreras2022supercritical}, which in turn was inspired by \cite{Cerf_2015}.

\begin{lem}\label{lem:trivial_a_priori_uniqueness_zone}
Let $G$ be a unimodular transitive graph of degree $d$. Fix $\eta \in (0,1)$ and $\eps \in (0,1/2)$. There exists $c(d,\eta,\eps) > 0$ such that for every $n \geq 1$ and $p \in [\eta,1]$,
\[
	\mathbb P_p\left( \op{Piv}[c \log n , n] \right) \leq \left( \frac{ \log \op{Gr}(n) }{cn} \right)^{\frac{1}{2}-\eps}.
\]
\end{lem}

In an infinite transitive graph, if $U_t(n) \gtrsim \log n$ then trivially $\op{Gr}(U_t(n)) \gtrsim \log n$, and hence $\Delta_t(n)$ is bounded above by a (possibly large) constant. Our goal is to improve this argument so that this constant can be made arbitrarily small. We will do this by improving either the bound on $U_t(n)$ or the bound on $\op{Gr}(U_t(n))$ given $U_t(n)$. When $G$ has superpolynomial growth, this is easy: we can use the trivial bound on $U_t(n)$, but then use the fact that $\op{Gr}(U_t(n)) \gtrsim \left(U_t(n)\right)^C$ for any particular constant $C$\footnote{This was the idea in \cite[Section 3]{easo2023critical}, where it sufficed to consider sequences converging to graphs of superpolynomial growth.}. When $G$ has polynomial growth, we will apply the following more delicate result from \cite[Proposition 6.1]{contreras2022supercritical} to improve our bound on $U_t(n)$. (For background on transitive graphs of polynomial and superpolynomial growth, see \cite{MR4253426}.)

\begin{lem}\label{non-trivial_a_priori_uniqueness_zone_for_poly}
Let $G$ be an infinite transitive graph of polynomial growth. For every $p > p_c(G)$ there exist $\chi (G,p) \in (0,1)$ and $C(G,p) < \infty$ such that for every $n \geq 1$ and $q \in [p,1]$,
\[
	\mathbb P_q\left( \op{Piv}\left[ e^{(\log n)^\chi} , n \right] \right) \leq C n^{-1/4}.
\]
\end{lem}

\begin{proof}[Proof of \cref{small_cost_on_infinite_graphs}]
Suppose that $t \in \mb R$ satisfies $\phi(t) > p_c(G)$. Note that $\phi(t) > 1/d$ because $p_c(G) \geq 1/(d-1)$ (as this holds for every infinite transitive graph). Let $c_1(d, 1/d,1/6) > 0$ be the constant from \cref{lem:trivial_a_priori_uniqueness_zone}. Then for every sufficiently large $n \in \mb L$,
\[
	\sup_{s \geq t} \mathbb P_{\phi(s)}\left( \op{Piv} \left[c_1 \log \left(n^{1/3} \right) , n^{1/3} \right] \right) \leq \left( \frac{ \log \op{Gr}(n^{1/3}) }{ c_1n^{1/3} } \right)^{\frac{1}{2}-\frac{1}{6}} \leq \left( \frac{ (\log n)^{100} }{c_1 n^{1/3}} \right)^{\frac{1} {3} } \leq \frac{1}{\log n},
\]
and hence $\inf_{s \geq t} U_s(n) \geq  \lfloor \frac{1}{4}c_1 \log (n^{1/3}) \rfloor = \frac{c_1}{13} \log n$. In particular,
\[
	\limsup_{n \to \infty} \sup_{s \geq t} \Delta_s(n) \mathbbm{1}_{\mb L}(n) \leq \limsup_{n \to \infty} \left[\frac{\log \log n}{(\log n) \wedge \log \operatorname{Gr} \left( \frac{c_1}{13} \log n \right)}\right]^{1/4}.
\]
If $G$ has superpolynomial growth, then $\frac{\log \operatorname{Gr} \left( \frac{c_1}{13} \log n \right)}{\log \log n} \to \infty$ and hence $\sup_{s \geq t} \Delta_s(n) \mathbbm{1}_{\mb L}(n) \to 0$ as $n \to \infty$. So we may assume to the contrary that $G$ has polynomial growth. Let $\chi(G,\phi(t))$ and $C(G,\phi(t))$ be the constants from \cref{non-trivial_a_priori_uniqueness_zone_for_poly}. Then for every sufficiently large $n \geq 1$,
\[
	\sup_{s \geq t} \mathbb P_{\phi(s)}\left( \op{Piv}\left[ e^{\left(\log \left(n^{1/3}\right) \right)^{\chi}} , n^{1/3} \right] \right) \leq C (n^{1/3})^{-1/4} \leq \frac{1}{\log n},
\]
and hence $\inf_{s \geq t} U_s(n) \geq \frac{1}{4} e^{\left(\log \left(n^{1/3}\right) \right)^{\chi}} \geq e^{(\log n)^{\chi/2}}$. In particular, using the trivial bound $\op{Gr}(U_s(n)) \geq U_s(n)$,
\[
	\limsup_{n \to \infty} \sup_{s \geq t} \Delta_s(n) \mathbbm{1}_{\mb L}(n) \leq \limsup_{n \to \infty} \left[\frac{\log \log n}{(\log n) \wedge \log \left(e^{(\log n)^{\chi/2}}\right) }\right]^{1/4} = 0.
\]
\end{proof}

Next we will use a compactness argument to deduce \cref{lem:orange_interval_yields_low-cost_block} from \cref{small_cost_on_infinite_graphs}. Given $d \in \mb N$, let $\mc U_d$ be the space of all unimodular transitive graphs with degree $d$ endowed with the local topology. Recall that every finite transitive graph is unimodular. By \cite[Corollary 5.5]{hutchcroft2019locality}, $\mc U_d$ is a closed subset of the space $\mc T_d$ of all transitive graphs with degree $d$ endowed with the local topology. In particular (recalling that the local topology is metrisable), since $\mc T_d$ is compact, so is $\mc U_d$.

\begin{proof}[Proof of \cref{lem:orange_interval_yields_low-cost_block}]
Suppose for contradiction that the statement is false. Then we can find $\eps > 0$ and $n_1 \geq 3$ such that for all $N \in \mb N$ there exists $t_N \leq \frac{1}{\eps}$ and a finite transitive graph $G_N$ with degree $d$ such that in $G_N$, the interval $[n_1,N]$ is orange at time $t_N$, but for every $m$ with $[L(m),m] \subseteq [n_1,N]$, there exists $n \in [L(m),m] \cap \mb L(G_N)$ with $\Delta_{t_N+\eps}^{G_N}(n) > \eps$. (We write $\Delta^{G}$, $U^G$, $\mb L(G)$ to denote $\Delta$, $U$, $\mb L$ defined with respect to a specific graph $G$.) By compactness, there exists an infinite subset $\mb M \subseteq \mb N$ and a unimodular transitive graph $G$ such that $G_N \to G$ as $N \to \infty$ with $N \in \mb M$.

First consider the case that $G$ is finite. Then trivially, there exists $n_0(G) < \infty$ such that for all $n \geq n_0$ and for all $s \in \mb R$, we have $U_s^G(n) = \lfloor \frac{1}{8} n^{1/3} \rfloor$. In particular, $\lim_{n \to \infty} \sup_{s \in \mb R} \Delta_s^{G}(n) = 0$. So there exists $m$ with $L(m) \geq n_1$ such that
\[
\sup_{n \in [L(m) , m]} \sup_{s\in \mb R} \Delta_s^G(n) \leq \eps.
\]
Pick $N \in \mb M$ sufficiently large that $N \geq m$ and $B_m^{G_N} \cong B_m^{G}$ (or even that $G \cong G_N$). Then we have a contradiction because there exists $n \in [L(m) , m]$ such that $\Delta_{t_N+\eps}^G(n) =\Delta_{t_N+\eps}^{G_N}(n) > \eps$.

So we may assume that $G$ is infinite. We claim that
\begin{equation} \label{eq:liminf_of_ts_is_high}
	\liminf_{ \substack{ N \to \infty \\ N \in \mb M } } \phi(t_N) \geq p_c(G).
\end{equation}
Indeed, suppose that $q \in (0,p_c(G))$. By the sharpness of the phase transition for percolation on infinite transitive graphs, there exists $C(q,G) < \infty$ such that $\mathbb P_q^G( o \leftrightarrow S_n ) \leq Ce^{-n/C}$ for all $n \geq 1$. Pick $m \geq n_1$ such that $C e^{-m/C} < \delta(m)$. Pick $N_0 \geq m$ such that for all $N \geq N_0$ with $N \in \mb M$, we have $B_m^{G_N} \cong B_m^G$. Then for all $N \geq N_0$ with $N \in \mb N$,
\[
	\min_{ u \in B_m^{G_N}} \mathbb P_q^{G_N} ( o \leftrightarrow u ) \leq \mathbb P_q^{G_N} ( o \leftrightarrow S_m ) = \mathbb P_q^{G} ( o \leftrightarrow S_m ) < \delta(m),
\]
so $m$ is not orange for $G_N$ at time $\phi^{-1}(q)$, and hence $q \leq \phi(t_N)$. Since $q$ was arbitrary, this establishes \cref{eq:liminf_of_ts_is_high}. Now by hypothesis, $t_N \leq \frac{1}{\eps}$ for every $N \geq 1$. So by \cref{eq:liminf_of_ts_is_high}, we know that $p_c(G) \leq \phi(1/\eps) < 1$. (We also know that $p_c(G) > 0$ since this holds for every infinite transitive graphs.) By passing to a further subsequence, we may assume that for all $N \in \mb M$,
\[
	t_N + \eps \geq \phi^{-1}( p_c(G) ) + \frac{\eps}{2} =: t.
\]
Note that $\phi(t) > p_c(G)$. So by \cref{small_cost_on_infinite_graphs},
\[
	\lim_{n \to \infty} \sup_{s \geq t} \Delta_s^G(n) \mathbbm 1_{\mb L(G)}(n) = 0.
\]
Pick $m$ with $L(m) \geq n_1$ such that
\begin{equation} \label{eq:low_cost_block_in_infinity_unimod}
	\sup_{n \in [L(m) , m]} \sup_{s \geq t} \Delta_s^G(n) \mathbbm 1_{\mb L(G)}(n) \leq \eps.
\end{equation}
Pick $N \in \mb M$ such that $N \geq m$ and $B_m^{G_N} \cong B_m^{G}$. Then $[L(m), m ] \subseteq [n_1,N]$, and the same inequality as \cref{eq:low_cost_block_in_infinity_unimod} holds with $G_N$ in place of $G$.
This contradicts the existence of $n \in [L(m) , m] \cap \mb L(G_N)$ satisfying $\Delta_{t_N+\eps}^{G_N}(n) > \eps$ because $t_N + \eps \geq t$. \qedhere
\end{proof}

\subsection{The obstacles are circles} \label{subsec:degenerate_geometry}

In this subsection we prove \cref{lem:degenerate_geometry_implies_circleness}.
Our argument is a finitary refinement of the argument in \cite[Section 5]{easo2023critical}. Our first step is to isolate the part of that previous argument that needs to be improved. For this we need to introduce the definition of plentiful tubes.

\paragraph{Plentiful tubes} Let $G$ be a transitive graph and fix a scale $n \geq 1$. We call the $r$-neighbourhood $B_r(\gamma) := \bigcup_{i} B_r(\gamma_i)$ of a path $\gamma \in \Gamma$ a \emph{tube}. Given constants $k,r,l \geq 1$, we say that $G$ has \emph{$(k,r,l)$-plentiful tubes} at scale $n$ if the following always holds. Let $A$ and $B$ be sets of vertices such that $(A,B) = (S_n,S_{4n})$ or such that $A$ and $B$ both contain paths from $S_n$ to $S_{3n}$. Then there is a set $\Gamma$ of paths from $A$ to $B$ such that $\abs{\Gamma} \geq k$, each path has length at most $l$, and $B_r(\gamma_1) \cap B_r(\gamma_2) = \emptyset$ for all pairs of distinct paths $\gamma_1, \gamma_2 \in \Gamma$. Note that the property of having $(k,r,l)$-plentiful tubes gets stronger as we increase $k$ (the number of tubes), increase $r$ (the thickness of tubes), or decrease $l$ (the lengths of tubes). We will be concerned mainly with the following two-parameter subset of this three-parameter family of properties. Given constants $c,\lambda > 0$, we say that $G$ has \emph{$(c,\lambda)$-polylog plentiful tubes} at scale $n$ if $G$ has $(k,r,l)$-plentiful tubes at scale $n$ with \[(k,r,l) := \left( [\log n]^{c \lambda} , n[\log n]^{-\lambda/c} , n[\log n]^{\lambda/c} \right).\]
We think of $c$ as representing a fixed exchange rate for the tradeoff between asking for more tubes that are long and thin vs fewer tubes that are short and thick, which we can realise by varying $\lambda$. Finally, recall from \cref{subsec:logic_of_induction} that $\mb T(c,\lambda)$ is defined to be the set of all scales $n \geq 3$ such that $\op{Gr}(n) \leq e ^{(\log n)^{100}}$ (i.e.\! $n \in \mb L$) and there exists $m$ satisfying $[m,m^{1+c}] \subseteq [n^{1/3} , n^{1/(1+c)}]$ such that $G$ has $(c,\lambda)$-polylog plentiful tubes at every scale in $[m,m^{1+c}]$.

Now suppose in the context of proving \cref{lem:degenerate_geometry_implies_circleness} that we have a large scale $n \in \mb L$ with $n < \gamma$, and we want to build the required plentiful tubes to establish that $n \in \mb T(c,\lambda)$. We split our argument into two cases, slow growth and fast growth, according to the rate of change of $\op{Gr}$ near $n$, as measured by whether \! $\op{Gr}(3m)/\op{Gr}(m)$ for $m \approx n$ exceeds some particular constant\footnote{In \cite{easo2023critical} we considered ratios of triplings $\op{Gr}(3n)/\op{Gr}(n)$ rather than of doublings $\op{Gr}(2n)/\op{Gr}(n)$ because only the former was known at the time to be sufficient to invoke the structure theory of transitive graphs of polynomial growth. Tointon and Tessera have since proved that small doublings imply small triplings \cite{tessera2023smalldoublingimpliessmall}, so it is now possible to work with doublings $\op{Gr}(2n)/\op{Gr}(n)$ instead, which is slightly more natural. However, since this does not significantly simplify our arguments, we have chosen to stay with triplings to avoid some repetition of work from \cite{easo2023critical}.}. The next lemma says that if $G$ has fast growth throughout a sufficiently large interval around scale $n$, then for some fixed exchange rate $c$, we have $(c,\lambda)$-polylog plentiful tubes for every choice of $\lambda$ whenever $n$ is sufficiently large. This is \cite[Proposition 5.4]{easo2023critical}, which was originally stated for infinite unimodular transitive graphs, but as we will justify in the appendix, exactly the same proof also works for finite transitive graphs.

\begin{lem}\label{lem:tubes_from_fast_tripling}
Let $G$ be a unimodular transitive graph of degree $d$. Suppose that
\[
	\operatorname{Gr}(m) \leq e^{(\log m)^D} \quad \text{and} \quad \operatorname{Gr}(3m) \geq 3^5 \operatorname{Gr}(m)
\]
for every $m \in [n^{1-\eps},n^{1+\eps}]$, where $\eps,D,n > 0$. Then there is a constant $c(d,D,\eps) > 0$ with the following property. For every $\lambda \geq 1$, there exists $n_0(d,D,\eps,\lambda) < \infty$ such that if $n \geq n_0$ then $G$ has $(c,\lambda)$-polylog plentiful tubes at scale $n$.
\end{lem}

The next lemma says that if $G$ has slow growth at some scale $n$, then outside of a bounded number of small problematic intervals, $G$ has plentiful tubes with good constants $(k,r,l)$ unless $G$ is one-dimensional. This is equivalent to \cite[Proposition 5.3]{easo2023critical}.

\begin{lem}\label{lem:tubes_from_small_tripling_old}
Let $G$ be an infinite transitive graph of degree $d$. Suppose that $\operatorname{Gr}(3n) \leq 3^\kappa \operatorname{Gr}(n)$, where $n,\kappa > 0$. There exists $C(d,\kappa) < \infty$ such that the following holds if $n \geq C$:

\noindent There is a set $A \subseteq [1,\infty)$ with $\abs{A} \leq C$ such that for every $k \geq 1$ and every $m \in [Ckn,\infty) \backslash \bigcup_{a \in A}[a,2ka]$, if $G$ does not have $(C^{-1}k,C^{-1}k^{-1}m,Ck^Cm)$-plentiful tubes at scale $m$, then $G$ is one-dimensional.
\end{lem}

We need to improve this lemma in two ways. First, we need the conclusion to be that ``$G$ looks one-dimensional from around scale $m$'', rather than just ``$G$ is one-dimensional''. Second, we need to allow $G$ to be finite. Here is the modified version of \cref{lem:tubes_from_small_tripling_old} that we will prove.\footnote{Although we have chosen to write everything for finite graphs, our proof also yields the analogous finitary refinement of \cref{lem:tubes_from_small_tripling_old} when $G$ is infinite.}

\begin{lem}\label{lem:tubes_from_small_tripling_new}
Let $G$ be a finite transitive graph of degree $d$. Suppose that $\operatorname{Gr}(3n) \leq 3^\kappa \operatorname{Gr}(n)$, where $n,\kappa > 0$. There exists $C(d,\kappa) < \infty$ such that the following holds if $n \geq C$:

\noindent There is a set $A \subseteq [1,\infty)$ with $\abs{A} \leq C$ such that for every $k \geq 1$ and every $m \in [Ckn,\infty) \backslash \bigcup_{a \in A}[a,2ka]$, if $G$ does not have $(C^{-1}k,C^{-1}k^{-1}m,Ck^Cm)$-plentiful tubes at scale $m$, then
\[
	\op{dist}_{\mathrm{GH}}\left( \frac{\pi}{\op{diam} G} G,S^1 \right) \leq \frac{Cm}{\operatorname{diam}G}.
\]

\end{lem}

In the next two subsections we will prove \cref{lem:tubes_from_small_tripling_new}. Before that, let us quickly check that \cref{lem:tubes_from_fast_tripling,lem:tubes_from_small_tripling_new} together do imply \cref{lem:degenerate_geometry_implies_circleness}. This is essentially the same as the proof of \cite[Proposition 5.2]{easo2023critical} given \cite[Propositions 5.3 and 5.4]{easo2023critical}.

\begin{proof}[Proof of \cref{lem:degenerate_geometry_implies_circleness} given \cref{lem:tubes_from_fast_tripling,lem:tubes_from_small_tripling_new}]

Let $c(d) > 0$ be a small positive constant to be determined. Let $\lambda \geq 1$ be arbitrary. Suppose that $n \geq 3$ satisfies $n \in \mb L \backslash \mb T(c,\lambda)$. We will freely (and implicitly) assume that $n$ is large with respect to $d$ and $\lambda$. Our goal is to show that if $c$ is sufficiently small, it then necessarily follows that $\gamma \leq n$.

First consider the possibility that $\op{Gr}(3m) \geq 3^5 \op{Gr}(m)$ for all $m \in [n^{1/3},n^{1/2}]$. Let $\eta := 1/100$, and let $c_1\left(d,101,\eta\right) > 0$ be given by \cref{lem:tubes_from_fast_tripling}. Note that for all $m \in [n^{1/3 + \eta},(n^{1/3+\eta})^{1+\eta}]$, we have $[m^{1-\eta} , m^{1 + \eta}] \subseteq [n^{1/3} , n^{1/2}]$ and $\op{Gr}(m) \leq \op{Gr}(n) \leq e^{[\log n]^{100}} \leq e^{[\log m]^{101}}$. So by construction of $c_1$, we know that $G$ has $(c_1,\lambda)$-polylog tubes at every scale $m \in [n^{1/3 + \eta},(n^{1/3+\eta})^{1+\eta}]$. In particular, if we pick $c(d) \leq c_1 \wedge \eta$, then $n \in  \mb T(c,\lambda)$ - a contradiction.

So we may assume that there exists $m \in [n^{1/3},n^{1/2}]$ such that $\op{Gr}(3m) \leq 3^5 \op{Gr}(m)$. Let $C(d,5) < \infty$ be as given by \cref{lem:tubes_from_small_tripling_new}. Without loss of generality, assume that $C$ is an integer and $C \geq 2$. Let $A \subseteq [1,\infty)$ with $\abs{A} \leq C$ be the set guaranteed to exist for our particular small-tripling scale $m$, and apply the conclusion of \cref{lem:tubes_from_small_tripling_new} with $k := (\log n)^{\lambda}$. Define $\eps(d) := \frac{1}{3C} \log \frac{3}{2}$, and consider the sequence $(u_i : 1 \leq 1 \leq 3C)$ defined by \[u_i : = \left(n^{1/2}\right)^{(1+\eps)^i}.\] Note that thanks to our choice of $\eps$,
\[
	u_{3C} = \left( n^{1/2} \right)^{(1+\eps)^{3C}} \leq \left( n^{1/2} \right)^{ \left( e^{\eps} \right) ^{3C }} = n^{3/4}. 
\]
As we are assuming that $n$ is large with respect to $d$ and $\lambda$, we also have that $C k m \leq u_1$, and for all $a \in A$, the interval $[a,2ka]$ contains at most one of the $u_i$'s. So by the pigeonhole principle, there exists $i$ such that $[u_i,u_{i+1}] \subseteq [C k m , \infty) \backslash \bigcup_{a \in A} [ a ,2ka]$. By construction of $A$, we know that either (1) for every scale $l \in [u_i,u_{i+1}]$, the graph $G$ has $(C^{-1}k,C^{-1} k^{-1} m , C k^{C} m)$-plentiful tubes at scale $l$, and in particular $(\frac{1}{2C} , \lambda)$-polylog plentiful tubes at scale $l$, or (2) there exists $l \in [u_i , u_{i+1}]$ such that $\gamma \leq C l $ and hence $\gamma \leq n$. If (1) holds, then by picking $c(d) \leq \eps \wedge \frac{1}{2C}$, we can guarantee that $n \in \mb T(c,\lambda)$ - a contradiction. So (2) holds, i.e.\! $\gamma \leq n$ as required. 
\end{proof}

\subsubsection{Cutsets and cycles}

In this subsection we reduce \cref{lem:tubes_from_small_tripling_new} to \cref{lem:two_bad_graphs_make_a_circle}, which is a less technical statement about cutsets and cycles. We will prove \cref{lem:two_bad_graphs_make_a_circle} in the next section.

\paragraph{Cutsets} Let $A,B,C$ be sets of vertices in a graph $G$. We write $A \xleftrightarrow{B} C$ to mean that there exists a finite path $(\gamma_k)_{k=0}^n$ such that $\gamma_0 \in A$; $\gamma_1,\ldots,\gamma_{n-1} \in B$; and $\gamma_n \in C$, and we write $A \xleftrightarrow{B} \infty$ to mean that there exists an infinite self-avoiding path $(\gamma_k)_{k=0}^{\infty}$ such that $\gamma_0 \in A$ and $\gamma_1,\gamma_2,\ldots \in B$. We write $\not\xleftrightarrow{B}$ to denote the negations of these properties. Now we say that $B$ is an \emph{$(A,C)$-cutset} to mean that $A \not\xleftrightarrow{B^C} C$, and we say that $B$ is a \emph{minimal $(A,C)$-cutset} if no proper subset of $B$ is also an $(A,C)$-cutset. We extend all of these definitions in the obvious way to allow $A$ or $C$ to be vertices rather than set of vertices. Now suppose that $G$ is an infinite transitive graph. Of course the spheres $S_n$ for $n \in \mb N$ are all $(o,\infty)$-cutsets, but interestingly, they are not always minimal $(o,\infty)$-cutsets because some transitive graphs contain dead-ends, i.e.\! a vertex that is at least as far from $o$ as all of its neighbours. The exposed sphere $S_n^{\infty}$ is defined to be the unique minimal $(o,\infty)$-cutset contained in the usual sphere $S_n$, which is given concretely by
\[
	S_n^{\infty} = \{ u \in S_n : u \xleftrightarrow{B_n^c} \infty \}.
\]
Thanks to the following result of Funar, Giannoudovardi, and Otera \cite[Proposition 5]{MR3291630}, exposed spheres also admit the following finitary characterisation: $S_n^{\infty}$ is the unique minimal $(o,S_{2n+1})$-cutset contained in $S_n$. We have included the short and elegant proof from their paper for the reader to appreciate that it does not adapt well to finite graphs. Specifically, it does not yield \cref{lem:removing_B_r_does_not_disconnect}, which is what we will need. We like to call this the \emph{inflexible geodesic} argument.

\begin{lem}\label{lem:finitary_exposed_sphere}
Let $G$ be an infinite transitive graph. Let $r \in \mathbb N$. Then every vertex $u \in B_{2r}^c$ satisfies $u \xleftrightarrow{B_r^c} \infty$. 
\end{lem}

\begin{proof}[Proof of \cref{lem:finitary_exposed_sphere}: The inflexible geodesic argument]
This proof uses the well-known fact that every infinite transitive graph contains a bi-infinite geodesic $\gamma = (\gamma_n)_{n \in \mathbb Z}$. Here is a sketch of how to prove this: There exist geodesic segments $\gamma^N = (\gamma_{n}^N)_{n = 0 }^{2N}$ for every $N \in \mathbb N$. By transitivity, we can pick these with $\gamma_N^N = o$ for every $N$. Then $\gamma$ is any local limit of these geodesic segments rooted at $o$, which exists by compactness.

Now fix $u \in B_{2r}^c$. Let $\gamma=(\gamma_n)_{n \in \mathbb Z}$ be a bi-infinite geodesic with $\gamma_0 = u$. Suppose for contradiction that there exist $s,t \in \mathbb N$ such that $\gamma_{-s},\gamma_{t} \in B_r$. Note that $\op{dist}( \gamma_{-s} , \gamma_t ) \leq \op{diam} B_r \leq 2r$. Since $\operatorname{dist}(o,u) > 2r$, we have $\gamma_{-s},\gamma_t \not\in B_r(u)$. Since $\gamma$ is a path, it follows that $s,t > r$, and in particular, $s+t > 2r$. On the other hand, since $\gamma$ is a geodesic, $s+t = \operatorname{dist}(\gamma_{-s},\gamma_{t})$. Therefore $2r < s+t \leq 2r$, a contradiction. So either $\{\gamma_n : n \geq 0\}$ or $\{ \gamma_n : n \leq 0 \}$ is disjoint from $B_r$ and therefore forms a path witnessing that $u \xleftrightarrow{B_r^c} \infty$.
\end{proof}

The usual definition of the exposed sphere is clearly inappropriate when working with finite transitive graphs. We propose that the exposed sphere in a finite transitive graphs should instead be defined according to this alternative finitary characterisation. Since \cref{lem:finitary_exposed_sphere} only applies to infinite transitive graphs, there is no reason for now that the reader should believe us that this is a good definition. We will fix this later by proving a finite graph analogue of \cref{lem:finitary_exposed_sphere}, namely \cref{lem:removing_B_r_does_not_disconnect}. As with (usual) spheres and balls, we extend the definition of exposed spheres to non-integer $n$ by setting $S_n^{\infty} := S_{\lfloor n \rfloor}^{\infty}$.

\begin{defn} \label{def:exposed_sphere}
Let $G$ be a transitive graph, which may be finite or infinite. Let $n \in \mb N$. We define the exposed sphere $S_n^{\infty}$ to be the unique minimal $(o,S_{2n+1})$-cutset contained in $S_n$, or equivalently,
\[
	S_n^{\infty} := \{ u \in S_n : u \xleftrightarrow{B_n^c} S_{2n+1} \}.
\]
\end{defn}

\paragraph{Cycles} Let $G$ be a graph. Recall that we identify spanning subgraphs of $G$ with functions $E \to \{0,1\}$. Pointwise addition and scalar multiplication of these functions makes the set of all spanning subgraphs into a $(\mathbb Z / 2\mathbb Z)$-vector space. Recall that a cycle is finite path that starts and ends at the same vertex and visits no other vertex more than once. We identify cycles (ignoring orientation) with spanning subgraphs and hence with elements of this $(\mathbb Z / 2\mathbb Z)$-vector space. Now let $\delta(G)$ be the minimal $n \in \mb N$ such that every cycle can be expressed as the linear combination of cycles having (extrinsic) diameter $\leq n$, if such an $n$ exists, and set $\delta(G) := +\infty$ otherwise. It is natural to ask whether cycles with diameter $\leq \delta(G)$ also generate every bi-infinite geodesic $\gamma$ in the sense that there is a sequence $(\gamma_n)_{n =1}^{\infty}$ of cycles each having diameter $\leq \delta(G)$ such that $\gamma_n \to \gamma$ pointwise as $n \to \infty$. Notice that this is equivalent to $G$ being one-ended. Benjamini and Babson \cite{MR1622785} (see also the proof by Timar \cite{timar2007cutsets}) proved that if $G$ is one-ended, then for all $u \in V$ and $v \in V \cup \{\infty\}$, every minimal $(u,v)$-cutset $A$ is $\delta(G)$-connected in the sense that $\op{dist}_G(A_1,A_2) \leq \delta(G)$ for every non-trivial partition $A = A_1 \sqcup A_2$. We will use this in the next section.



We claim that \cref{lem:tubes_from_small_tripling_new} can be reduced to the following statement about cutsets and cycles by applying the structure theory of groups and transitive graphs of polynomial growth. Since this step is essentially identical to the proof of \cite[Proposition 5.3]{easo2023critical}, we have chosen to defer the details to the appendix. For the same reason, we will not give an overview of the rich theory of polynomial growth or even the definition of a virtually nilpotent group. The relevant background can be found in \cite{easo2023critical,EHStructure,MR4253426}.

\begin{lem}\label{lem:two_bad_graphs_make_a_circle}
Let $r,n \geq 1$. Let $G$ be a finite transitive graph such that $S_n^{\infty}$ is not $r$-connected. Let $H$ be a (finite or infinite) transitive graph with
$\delta(H) \leq r$ that does not have infinitely many ends. If $B_{50n}^H \cong B_{50n}^G$, then
\[
	\op{dist}_{\mathrm{GH}}\left(\frac{\pi}{\op{diam} G}G,S^1 \right) \leq \frac{200n}{\op{diam} G}.
\]
\end{lem}

\subsubsection{Solving the reduced problem}

In this section we prove \cref{lem:two_bad_graphs_make_a_circle}. Benjamini and Babson \cite{MR1622785} tell us that if in an infinite transitive graph $H$, the exposed sphere $S_n^{\infty}$ is not $\delta(H)$ connected\footnote{meaning that there is a non-trivial partition $\delta(H) = A_1 \sqcup A_2$ with $\op{dist}(A_1,A_2) > \delta(H)$}, then $H$ must not be one-ended. If $H$ is not infinitely-ended either, then $H$ must in fact be two-ended and hence one-dimensional. This is how the argument (implicitly) went in \cite{easo2023critical}. To make this more finitary, let us start by noting that the proof that $H$ not one-ended actually also tells us that this is witnessed by $S_n^{\infty}$ itself in the sense that $H \backslash S_n^{\infty}$ has multiple infinite components. Equivalently (by \cref{lem:finitary_exposed_sphere}), the exposed sphere $S_n^{\infty}$ disconnects\footnote{We say that a set of vertices $A$ disconnects another set of vertices $B$ if there exist vertices $b_1,b_2 \in B$ such that $b_1 \not\xleftrightarrow{A^c} b_2$.} $S_{2n+1}$. (This alternative phrasing has the benefit that it also makes sense when $H$ is finite.) Indeed, this follows from the next lemma with $(A,B) := (\{o\}, S_{2n+1})$. This is also an instance of Benjamini-Babson, just phrased slightly differently in terms of \emph{sets of} vertices, \emph{vertex} cutsets, and (extrinsic) \emph{diameter} rather than length of generating cycles. For completeness, we have written Timar's proof of Benjamini-Babson with the necessary tiny adjustments in the appendix.

\begin{lem}\label{lem:timar_improved}
Let $G$ be a graph. Let $A$ and $B$ be sets of vertices. Let $\Pi$ be a minimal $(A,B)$-cutset that does not disconnect $A$ or $B$. Then $\Pi$ is $\delta(G)$-connected.
\end{lem}

The following elementary lemma lets us conclude from this that $H$ must begin to look one-dimensional \emph{already from scale $n$}. We say that a path $\gamma = (\gamma_t : t \in I)$ is \emph{$n$-dense} if $\sup_{v \in V} \op{dist}_G(v,\gamma) \leq n$, where $\op{dist}_G(v,\gamma) := \op{dist}_G(v , \{ \gamma_t : t \in I\})$.

\begin{lem}\label{lem:two_ends_implies_Z_quantitatively}
Let $G$ be an infinite transitive graph. Let $n \geq 1$. If $G$ is two-ended and $G \backslash B_n$ has two infinite components, then $G$ contains an $n$-dense bi-infinite geodesic.
\end{lem}

\begin{proof}
Let $A$ and $B$ be the two infinite components of $G \backslash B_n$. For each integer $N \geq n+1$, let $\gamma^N = (\gamma^N_{t} : -a_N \leq t \leq b_N)$ be a shortest path among those that start in $S_{N} \cap A$ and end in $S_N \cap B$, indexed such that $\gamma^N_0 \in B_n$. By compactness, there exists a bi-infinite geodesic $\gamma = (\gamma_t : t \in \mb Z)$ and a subsequence $(\gamma^N : N \in \mb M)$ such that for every $t \in \mb Z$, we have $\gamma_t = \gamma^N_t$ for all sufficiently large $N \in \mb M$. As in the \emph{inflexible geodesic} argument used to prove \cref{lem:finitary_exposed_sphere} (i.e.\! by the triangle inequality), a geodesic can never visit $B_{2n}^c$ in between two visits to $B_{n}$. It follows that there exists $t_0$ such that $\gamma^- := (\gamma_{-t} : t \geq t_0)$ is entirely contained in $A$, and $\gamma^+ := (\gamma_{t} : t \geq t_0)$ is entirely contained in $B$.

Suppose for contradiction that $\gamma$ is not $n$-dense. Pick $u \in V$ with $\op{dist}(u,\gamma) > n$. Since $B_n(u)$ does not intersect $\gamma$, the path $\gamma$ must be entirely contained in one of the two infinite components of $G \backslash B_n(u)$, say $C$. Since $B_n(o)$ disconnects $\gamma^-$ from $\gamma^+$, there are at least two infinite components in $C \backslash B_n(o)$. So there are at least three infinite components in $G \backslash (B_n(o)\cup B_n(u))$, contradicting the fact that $G$ is two-ended.
\end{proof}


What happens if instead $H$ is finite? \Cref{lem:timar_improved} still tells us that if $S_n^{\infty}$ is not $\delta(H)$ connected, then $S_n^{\infty}$ disconnects $S_{2n+1}$. When $H$ was infinite, this had a nice interpretation in terms of ends because we could go back to the original infinitary definition of $S_{n}^{\infty}$ as a minimal $(o,\infty)$-cutset. The problem when $H$ is finite is that we are stuck with our artificial finitary definition of $S_{n}^{\infty}$ as a minimal $(o,S_{2n+1})$-cutset. The next lemma justifies our definition by establishing that $S_n^{\infty}$ is automatically a minimal $(o,u)$-cutset for every vertex $u \in B_{2n}^c$. Thanks to this lemma, it is simply impossible that $S_n^{\infty}$ is not $\delta(H)$-connected when $H$ is finite. 

The analogous statement for one-ended infinite transitive graphs follows from \cref{lem:finitary_exposed_sphere}\footnote{This was the motivation for \cite{MR3291630} to prove \cref{lem:finitary_exposed_sphere}.}. In this sense, \cref{lem:removing_B_r_does_not_disconnect} lets us treat finite transitive graphs as if they were infinite transitive graphs that are one-ended. Note that a naive finite-graph adaptation of the inflexibe geodesic argument used to prove \cref{lem:finitary_exposed_sphere} would not yield \cref{lem:removing_B_r_does_not_disconnect}. (It would just say that every vertex in $S_{2n+1}$ belongs to a cluster in $G \backslash B_{n}$ of large diameter.) Our argment also yields a new proof of \cref{lem:finitary_exposed_sphere}.

\begin{lem}\label{lem:removing_B_r_does_not_disconnect}
Let $G$ be a finite transitive graph. Let $r \in \mb N$. Then $B_r$ does not disconnect $B_{2r}^c$. 
\end{lem}

\begin{proof}[Proof of \cref{lem:finitary_exposed_sphere} and \cref{lem:removing_B_r_does_not_disconnect}]
Suppose that $B_r$ does disconnect $B_{2r}^c$. (For \cref{lem:removing_B_r_does_not_disconnect} we assume this for sake of contradiction, whereas for \cref{lem:finitary_exposed_sphere}, we may assume this otherwise the conclusion is trivial.) Let $C$ be a component of $G \backslash B_{r}$ intersecting $B_{2r}^c$. It suffices to prove that $C$ is infinite. (For \cref{lem:removing_B_r_does_not_disconnect}, this establishes the required contradiction because $G$ is finite, whereas for \cref{lem:finitary_exposed_sphere}, this is the desired conclusion.)

Suppose for contradiction that $C$ is finite. Then we can pick a vertex $u \in C$ maximising $\op{dist}(o,u)$. Since $\op{dist}(o,u) \geq 2r+1$ and (by transitivity) $B_r(u)$ disconnects $B_{2r}(u)^c$, there exists a vertex $v \in B_{2r}(u)^{c}$ such that $o \not\xleftrightarrow{B_r(u)^c} v$. Since $B_r(u) \cap B_r(o) = \emptyset$ and the subgraph induced by $B_r(o)$ is connected, $B_r(o) \not\xleftrightarrow{B_r(u)^c} v$. Since $G$ is connected, $v \xleftrightarrow{ B_r(o)^c } B_r(u)$. Since $B_r(u) \cap B_r(o) = \emptyset$ and the subgraph induced by $B_r(u)$ is connected, $v \xleftrightarrow{ B_r(o)^c } u$, i.e. $v \in C$. However, since every path from $o$ to $v$ must visit $B_r(u)$,
\[\begin{split}
	\op{dist}(o,v) &\geq \op{dist}(o , B_r(u)) + \op{dist}(B_r(u) ,v )  \\
	&\geq \left( \op{dist}(o,u) - r \right) + \left( \op{dist}(u,v) - r \right) \geq \op{dist}(o,u) +1,
\end{split}\]
contradicting the maximality of $\op{dist}(o,u)$.
\end{proof}


By applying our work up to this point, under the hypothesis of \cref{lem:two_bad_graphs_make_a_circle}, we can prove that the graph $H$ must be infinite and begin to look one-dimensional from around scale $n$. By the next lemma, it follows that $G$ looks like a circle from around scale $n$.

\begin{lem}\label{lem:quotient_of_Z_is_circle}
	Let $G$ and $H$ be transitive graphs. Suppose that $G$ is finite whereas $H$ contains an $n$-dense bi-infinite geodesic for some $n \geq 1$. If $B_{50n}^G \cong B_{50n}^H$ then
	\[
		\op{dist}_{\mathrm{GH}} \left( \frac{\pi} { \op{diam} G } G,S^1 \right) \leq \frac{200 n}{\op{diam} G}.
	\]
\end{lem}

\begin{proof}
	Let $\gamma = (\gamma_t)_{t \in \mb Z}$ be an $n$-dense bi-infinite geodesic in $H$. Without loss of generality, assume that $\gamma_0 = o_H$. We will break our proof into a sequence of small claims.

	\bigskip

	\begin{clm*}
		$B_{2n}^H$ disconnects $S_{2n}^{\infty,H}$.
	\end{clm*}
	\begin{proof}[Proof of claim]
		Let $\zeta = (\zeta_t)_{t=0}^k$ be an arbitrary path from $\zeta_0 = \gamma_{-2n}$ to $\zeta_k = \gamma_{2n}$. Since $\gamma$ is $n$-dense, for all $t \in \{0,\ldots,k\}$, there exists $g_t \in \mb Z$ such that $\op{dist}( \zeta_t , \gamma_{g_t} ) \leq n$. We can of course require that $g_0 := -2n$ and $g_k := 2n$. Since $\gamma$ is geodesic, for all $t \in \{0,\ldots,k-1\}$,
		\[\begin{split}
			\abs{g_{t+1} - g_{t} } &= \op{dist}\left( \gamma_{ g_{ t+1 }  } , \gamma_{ g_{t} } \right) \\
			&\leq \op{dist}\left( \gamma_{g_t} , \zeta_t \right) + \op{dist}\left( \zeta_{t} , \zeta_{t+1} \right) + \op{dist}\left( \zeta_{t+1} , \gamma_{g_{t+1}} \right) \\
			&\leq n + 1 + n = 2n+1.
		\end{split}\]
		In particular, since $g_0 \leq -(n+1)$ but $g_k \geq n+1$, there must exist $t \in \{1,\ldots,k-1\}$ such that $-n \leq g_t \leq n$. Then $\gamma_{g_t} \in B_n^H$ and hence $\zeta_t \in B_{2n}^H$. Since $\zeta$ was arbitrary, this establishes that $B_{2n}^H$ disconnects $\gamma_{-2n}$ from $\gamma_{2n}$. Since $\gamma$ is a geodesic, $\gamma_s \in S_{s}^{H}$ for all $s \in \mb Z$. So the path $(\gamma_{s} : s \geq 2n)$ witnesses the fact that $\gamma_{2n} \in S_{2n}^{\infty,H}$. Similarly, $\gamma_{-2n} \in S_{2n}^{\infty,H}$. So $B_{2n}^H$ disconnects $S_{2n}^{\infty,H}$.
	\end{proof}

	Fix a non-trivial partition $S_{2n}^{\infty,H} = A \sqcup B$ such that $B_{2n}^H$ is an $(A,B)$-cutset. Now suppose that there is a graph isomorphism $\psi: B_{50n}^G \to B_{50n}^H$. Note that $\psi$ induces a bijection $S_{2n}^{\infty,G} \leftrightarrow S_{2n}^{\infty,H}$. In particular, $S_{2n}^{\infty,G} = \psi(A) \sqcup \psi(B)$. By definition of exposed sphere, $B_{2n}^G$ does not disconnect $\psi(A)$ or $\psi(B)$ from $(B_{4n}^G)^c$. So by \cref{lem:removing_B_r_does_not_disconnect}, $B_{2n}^G$ is not a $(\psi(A),\psi(B))$-cutset. Consider a shortest path from $\psi(A)$ to $\psi(B)$ that witnesses this, then connect the start and end of this path to $o_G$ by geodesics. Let $\lambda = (\lambda_k)_{k \in \mathbb Z_l}$ be the resulting cycle, labelled such that $\lambda_0 = o_G$. We will write $\abs{s}_l$ for the distance from $s$ to $0$ in the cycle graph $\mb Z_l$. The next three claims establish that $\lambda$ is roughly dense and geodesic.

	\bigskip

	\begin{clm*}
	For all $s \in \mb Z_l$, if $\abs{s}_l > 2n$ then $\op{dist}_G(\lambda_s,B_{2n}^G) = \abs{s}_l - 2n$
	\end{clm*}
	\begin{proof}[Proof of claim]
	Fix $s \in \mb Z_l$ with $\abs{s}_l > 2n$. Since $B_{2n}^H$ is an $(A,B)$-cutset, the segment $(\lambda_t : \abs{t}_l > 2n)$ must intersect $(B_{4n}^G)^c$ (it must exit the ball $B_{50n}$, on which $G$ and $H$ are isomorphic), but by construction, this segment does not intersect $B_{2n}^G$. So every path from $\gamma_s$ to $B_{2n}^G$ must intersect $S_{2n}^{\infty,G}$. In particular, by minimality in the construction of $\lambda$,
	\[
		\op{dist}_G(\gamma_s,B_{2n}^G) = \op{dist}_G(\gamma_s,\psi(A)) \wedge \op{dist}_G(\gamma_s,\psi(B)) = \abs{s}_l-2n. \qedhere
	\]
	\end{proof}

	\bigskip
 
	\begin{clm*}
	For all $s,t \in \mb Z_l$, we have $\abs{s-t}_l - 4n \leq \op{dist}_G(\lambda_s,\lambda_t) \leq \abs{s-t}_l$.
	\end{clm*}
	\begin{proof}[Proof of claim]
	The second inequality is trivial, and the first inequality is trivial when $\abs{s}_l \vee \abs{t}_l \leq 2n$. By our previous claim, if $\abs{s}_l > 2n$ and $\abs{t} \leq 2n$, then
	\[
		\op{dist}_G(\lambda_s,\lambda_t) \geq \op{dist}_G(\lambda_s,B_{2n}^G) = \abs{s}_l - 2n \geq \abs{s-t}_l - 4n.
	\]
	Similarly, the first inequality also holds if instead $\abs{t}_l > 2n$ and $\abs{s} \leq 2n$. So let us consider $s$ and $t$ satisfying $\abs{s}_l \wedge \abs{t}_l > 2n$, and fix an arbitrary path $\eta$ from $\lambda_s$ to $\lambda_t$. If $\eta$ intersects $B_{2n}^G$, then by our previous claim, $\eta$ has length at least
	\[
		\op{dist}_G(\lambda_s,B_{2n}^G) + \op{dist}_G(\lambda_t,B_{2n}^G) =( \abs{s}_l - 2n) + ( \abs{t}_l - 2n ) \geq \abs{s-t}_l - 4n.
	\]
	If $\eta$ does not intersect $B_{2n}^G$, then by minimality in the construction of $\lambda$, the length of $\eta$ is at least $\abs{s-t}_l$. Either way, $\op{dist}_G(\lambda_s,\lambda_t) \geq \abs{s-t}_l - 4n$.
	\end{proof}

	\bigskip

	\begin{clm*}
	$\lambda$ is $10n$-dense
	\end{clm*}
	\begin{proof}[Proof of claim]
	Suppose for contradiction that $u$ is a vertex with $\op{dist}_G(u,\lambda) > 10n$. Let $v$ be a vertex in $\lambda$ that is closest to $u$. Let $z$ be a vertex in $S_{10n}(v)$ that lies along a geodesic from $u$ to $v$. Since $\lambda$ visits $o_G$ but must exit $B_{50n}^G$ (on which $G$ and $H$ are isomorphic), we know that $\lambda$ has (extrinsic) diameter $> 50n$. By our previous claim, it follows that $\lambda$ visits vertices $x$ and $y$ in $S_{10n}^G(v)$ satisfying $\op{dist}_G(x,y) \geq 2 \cdot 10n - 4n \geq 10n$. Now $x,y,z$ are three vertices in $S_{10n}^G(v)$ such that the distance between any pair is at least $10n$. By transitivity and the fact that $B_{50n}^G \cong B_{50n}^H$, three such vertices can also be found in $S_{10n}^H(o)$, say $v_1,v_2,v_3$. Since $\gamma$ is $n$-dense, there exist integers $k_1,k_2,k_3$ such that $\op{dist}_H( v_i , \gamma_{k_i} ) \leq n$ for each $i \in \{1,2,3\}$. Notice that since $\gamma_0 = o_{H}$ and $\gamma$ is geodesic, $\abs{k_i} \in [9n,11n]$ for all $i$. So by the pigeonhole principle, either $[-9n,-11n]$ or $[9n,11n]$ contains $k_i$ for at least two distinct values of $i$. On the other hand, for all $i \not= j$, since $\gamma$ is geodesic,
	\[
		\abs{k_i - k_j} = \op{dist}_H( \gamma_{k_i}, \gamma_{k_j} ) \geq \op{dist}_H(v_i,v_j) - 2n \geq 8n.
	\]
	So an interval of width $2n$ can never contain $k_i$ for at least two distinct values of $i$, a contradiction.
	\end{proof}

	Thanks to the previous two claims, the map $\mb Z_l \to G$ sending $t \mapsto \lambda_t$ is a $(1,10n)$-quasi-isometry. So (by exercise 5.10 (b) in \cite{MR4696672}, for example), $\op{dist}_{\mathrm{GH}}(\mb Z_l , G) \leq 10n$. By the obvious $1$-dense isometric embedding of $\mb Z_l$ into $\frac{l}{2 \pi} S^1$, we know that $\op{dist}_{\mathrm{GH}}(\mb Z_l , \frac{l}{2 \pi} S^1) \leq 1$. Let $D := \op{diam} G$. By the previous two claims $\abs{D - \frac{l}{2} } \leq 20n$. So by considering the identity map from $G$ to itself,
	\[
	\begin{split}
		\op{dist}_{\mathrm{GH}}\left(\frac{1}{D}G , \frac{2}{l} G\right) &\leq \sup_{u,v \in V(G)}
		\abs{ \frac{1}{D} \op{dist}_G(u,v) - \frac{2}{l} \op{dist}_G(u,v)  }\\ &\leq D \cdot \abs{ \frac{1}{D}- \frac{2}{l} } \leq \frac{40n}{ l}.
	\end{split}
	\]
	Putting these bounds together,
	\begin{equation} \label{eq:QI-calc}
	\begin{split}
		\op{dist}_{\mathrm{GH}}\left( \frac{\pi}{D} G , S^1 \right) &\leq \op{dist}_{\mathrm{GH}}\left( \frac{\pi}{D} G , \frac{2\pi}{l} G  \right) + \op{dist}_{\mathrm{GH}}\left( \frac{2\pi}{l} G , \frac{2\pi}{l} \mb Z_l  \right) + \op{dist}_{\mathrm{GH}}\left( \frac{2\pi}{l} \mb Z_l , S^1  \right) \\
		&\leq \pi \cdot \frac{ 40n }{l} + \frac{2\pi}{l} \cdot 10n + \frac{2 \pi}{l} \cdot 1 \\
		& \leq \frac{200n}{l} = 100n \cdot \frac{2}{l} \leq 100n \cdot \frac{1}{D - 20n} = 5 \cdot \frac{20n/D}{1 - 20n / D}.
	\end{split}
	\end{equation}
	We may assume that $D \geq 40n$, other the conclusion of the lemma holds trivially because $\op{dist}_{\mathrm{GH}}(A,B) \leq 1$ for all non-empty compact metric spaces $A$ and $B$ each having diameter at most 1. In particular, $\frac{20n}{D} \leq \frac{1}{2}$. Since $\frac{x}{1-x} \leq 2x$ for all $x \in [0,1/2]$, it follows from \cref{eq:QI-calc} that $\op{dist}_{\mathrm{GH}}\left( \frac{\pi}{D} G , S^1 \right) \leq 5 \cdot 2 \cdot 20n/D = 200n/D$ as required.
\end{proof}


We now combine these lemmas to formalise this sketch of a proof of \cref{lem:two_bad_graphs_make_a_circle}, thereby concluding our proof of \cref{prop:local_to_global_connections}.

\begin{proof}[Proof of \cref{lem:two_bad_graphs_make_a_circle}]
Suppose that $B_{50n}^G \cong B_{50n}^H$. Note that $S_n^{\infty,G}$ is trivially $2n$-connected. So $r \leq 2n$. In particular, in any transitive graph, the statement ``$S_n^{\infty}$ is not $r$-connected'' is determined by the subgraph induced by $B_{50n}$. So $S_n^{\infty,H}$ is not $r$-connected either. By definition, $S_n^{\infty,H}$ is a minimal $(o_H,S_{2n+1}^H)$-cutset. So by \cref{lem:timar_improved}, since $\delta(H) \leq r$, the exposed sphere $S_n^{\infty,H}$ must disconnect $S_{2n+1}^H$. In particular, $B_n^H$ disconnects $(B_{2n}^H)^c$. So by \cref{lem:removing_B_r_does_not_disconnect}, $H$ is infinite, and by \cref{lem:finitary_exposed_sphere}, $H \backslash B_n^H$ contains at least two infinite components. Since $H$ has at most finitely many ends, $H \backslash B_n^H$ must contain exactly two infinite components and $H$ must be exactly two-ended. So by \cref{lem:two_ends_implies_Z_quantitatively}, $H$ contains an $n$-dense bi-infinite geodesic. The conclusion follows by \cref{lem:quotient_of_Z_is_circle}.
\end{proof}

\section{Global connections $\to$ unique large cluster} \label{sec:global_connections_to_unique_large_cluster}

In this section we apply the methods of \cite{easo2021supercritical}. It will be convenient to adopt the following notation from that paper: given a set of vertices $A$ in a graph $G$, we define its \emph{density} to be $\den{A} := \frac{\abs{A}}{\abs{V(G)}}$. In \cite{easo2021supercritical}, together with Hutchcroft, we showed that the supercritical giant cluster for percolation on bounded-degree finite transitive graphs is always unique with high probability. More precisely, for every infinite set $\mathcal G$ of finite transitive graphs with bounded degrees, for every supercritical sequence of parameters $p$, and for every constant $\eps > 0$, the density of the second largest cluster $\den{K_2}$ satisfies
\[
	\lim \mathbb P_p\left( \den{K_2} \geq \eps \right) = 0.
\]
The following proposition contains a quantitative version of this statement that is useful even if we slightly weaken the hypothesis that $p$ is supercritical. We think of this as saying that if at some parameter $p$ we have a point-to-point lower bound that is only slightly worse than constant as $\abs{V} \to \infty$, then after passing to $p+\eps$, we can still pretend that we are actually in the supercritical phase and still prove that the second largest cluster is typically much smaller than the largest cluster. Note that this largest cluster is not necessarily a giant cluster because we are not (a priori) really in the supercritical phase.\footnote{This is reminiscent of \cite[Section 6]{easo2023critical}. There we used the hypothesis of a point-to-point lower bound on a large scale to enable us to run arguments from \cite{contreras2022supercritical}, which were ostensibly about supercritical percolation, to study subcritical percolation.}

\begin{prop}\label{prop:uniqueness}
Let $G$ be a finite transitive graph with degree $d$. Define $\delta := ( \log \abs{V} )^{-1/20}$. There exists $C(d) < \infty$ such that if $\abs{V} \geq C$, then for all $p,q \in (0,1)$ with $q-p \geq \delta$,
\[
	\min_{u,v \in V} \mathbb P_p( u \leftrightarrow v ) \geq 2\delta \quad \implies \quad \mathbb P_{q}\left( \den{K_1} \geq \delta \text{ and } \den{K_2} \leq \delta^2 \right) \geq 1-\delta^4.
\]
\end{prop}

In \cref{subsec:sandcastles} we will explain why this proposition is implied by the \emph{sandcastles}\footnote{We thank Coales for suggesting this name.} argument of \cite{easo2021supercritical}. In fact, \cite{easo2021supercritical} already explicitly contains a very similar quantitative statement, namely \cite[Theorem 1.5]{easo2021supercritical}. Unfortunately this statement is not quantitatively strong enough for our purposes. One could alternatively prove a version of \cref{prop:uniqueness} by applying the ghost-field technology developed in \cite[Section 4]{easo2023critical}. (See the discussion at the end of \cite[Section 7.1]{easo2023critical}.) This approach would be less elementary and less generalisable\footnote{The ghost-field arguments ultimately rely on two-arm bounds, which are not elementary and which break down when working with graphs with rapidly diverging vertex degrees.} but quantitatively stronger.


\subsection{Proof via sandcastles} \label{subsec:sandcastles}

Let $G$ be a finite transitive graph. In \cite{easo2021supercritical} we made the definition of ``supercritical sequence" finitary as follows. Given a constant $\eps > 0$, we say that a parameter $p \in (0,1)$ is \emph{$\eps$-supercritical} if
\[
	\mathbb P_{(1-\eps)p}( \den{K_1} \geq \eps ) \geq \eps
\]
and $\abs{V} \geq 2 \eps^{-3}$, the latter being a technical condition that the reader may like to ignore. Note that a sequence of parameters is supercritical if and only if there exists a constant $\eps > 0$ such that all but finitely many of the parameters are $\eps$-supercritical. On the other hand, in the present paper the more relevant finitary notion of supercriticality concerns point-to-point connection probabilities, i.e.\!
\[
	\min_{u,v} \mathbb P_{(1-\eps)p} \left( u \leftrightarrow v \right) \geq \eps.
\]
These properties are equivalent up to changing the constant $\eps$. Indeed, for every parameter $p \in (0,1)$ and every constant $\eps > 0$ satisfying the technical condition $\abs{V} \geq 2\eps^{-3}$,
\begin{equation} \label{eq:two_point_bd_vs_giant}
	\min_{u,v} \mathbb P_p(u \leftrightarrow v) \geq 2\eps \quad \implies \quad \mathbb P_p\left( \den{K_1} \geq \eps \right) \geq \eps \quad \implies \quad \min_{u,v} \mathbb P_p(u \leftrightarrow v) \geq e^{-10^{5} \eps^{-18}}.
\end{equation}
The first implication is an easy application of Markov's inequality, and the second implication is \cite[Theorem 2.1]{easo2021supercritical}. A version of the second implication assuming an upper bound on the degree of $G$ is originally due to Schramm. Notice that the second implication quantitatively loses much more than the first. In this sense, we can think of the hypothesis ``$\min_{u,v} \mathbb P_{(1-\eps)p}(u \leftrightarrow v) \geq \eps$'' as being quantitatively much stronger than the hypothesis ``$\mathbb P_{(1-\eps)p}(\den{K_1} \geq \eps ) \geq \eps$''.

Below is \cite[Theorem 1.5]{easo2021supercritical}, which contains a finitary uniqueness statement similar to \cref{prop:uniqueness}. Unfortunately, the terrible $e^{-C \eps^{-18}}$ dependence on $\eps$ is not good enough for our purposes. Fortunately, it turns out that in the proof of this theorem, the source of this poor dependence is a conversion from the hypothesis of a giant cluster bound (implicit in $p$ being $\eps$-supercritical) into a point-to-point bound, i.e.\! an application of the second implication in \cref{eq:two_point_bd_vs_giant}. This saves us because in the present setting we actually start with the ``stronger'' hypothesis of a point-to-point bound.

\begin{thm}
Let $G$ be a finite transitive graph with degree $d$. There exists $C(d) < \infty$ such that for every $\eps > 0$, every $\eps$-supercritical parameter $p$, and every $\lambda \geq 1$,
\[
	\mathbb P_p\left( \den{K_2} \geq \lambda e^{C \eps^{-18}}\left( \frac{ \log d } { \log \abs{V} } \right)^{1/2} \right) \leq \frac{1}{\lambda}.
\]
\end{thm}

A key ingredient in the sandcastles argument of \cite{easo2021supercritical} is the sharp density property, which measures the extent to which the events $\{\den{K_1} \geq \alpha\}$ for each $\alpha \in (0,1)$ have uniformly-in-$\alpha$ sharp thresholds. Let $\Delta : (0,1) \to (0,1/2]$ be a decreasing function. For all $\alpha, \delta \in (0,1)$, let $p_c(\alpha,\delta) \in (0,1)$ be the parameter satisfying $\mathbb P_{p_c(\alpha,\delta)}( \den{K_1} \geq \alpha) = \delta$, which is unique by the strict monotonicity of this probability with respect to $p$. We say that $G$ has the \emph{$\Delta$-sharp density property} if for all $\alpha \in (0,1)$ and $\delta \in [ \Delta(\alpha),1/2]$,
\[
	\frac{ p_c(\alpha,1-\delta) } { p_c(\alpha,\delta) } \leq e^{\delta}.
\]
The following lemma establishes a sharp density property for graphs with bounded degrees. This is \cite[Proposition 3.2]{easo2021supercritical} and is an easy consequence of Talagrand's well-known sharp threshold theorem \cite{MR1303654}.

\begin{lem}\label{lem:sharp_density}
Let $G$ be a finite transitive graph with degree $d$. There exists $C(d) < \infty$ such that $G$ has the $\Delta$-sharp density property for the function $\Delta : (0,1) \to (0,1/2]$ given by
\[
	\Delta(\alpha) :=
	\begin{cases}
		\frac{1}{2} \wedge \frac{C}{ (\log \abs{V})^{1/2} } &\text{ if } \alpha \geq \left( \frac{2}{\abs{V}} \right)^{1/3} \\
		\frac{1}{2} &\text{ otherwise.}
	\end{cases}
\]
\end{lem}

The sandcastles argument combines a sharp density property with a point-to-point bound to establish the uniqueness of the largest cluster. Here is the technical output of that argument.

\begin{lem}\label{lem:technical_sandcastles_output}
Let $G$ be a finite transitive graph. Let $\eps \in (0,1)$ and suppose that $p \in (0,1)$ is $\eps$-supercritical. Suppose that $G$ satisfies the $\Delta$-sharp density property for some decreasing function $\Delta : (0,1) \to (0,1/2]$. Then for all $\lambda \geq 1$,
\[
	\mathbb P_p\left( \den{K_2} \geq \lambda \left( \frac{200 \Delta(\eps)}{\eps^{3} \tau} + \frac{25}{\eps^2 \tau \abs{V}} \right) \right) \leq \frac{\eps}{\lambda},
\]
where
\[
	\tau := \min_{u,v} \mathbb P_{(1-\eps)p}( u \leftrightarrow v ).
\]
\end{lem}
\begin{proof}
\Cite[Theorem 3.3]{easo2021supercritical} is the same statement but where $\tau$ is instead defined to be
\[
	\tau := e^{- 10^{5} \eps^{-18}},
\]
which is the function appearing in \cref{eq:two_point_bd_vs_giant}. We claim that the proof of \cite[Theorem 3.3]{easo2021supercritical} actually also establishes \cref{lem:technical_sandcastles_output}. First note that in the statement of \cite[Lemma 3.5]{easo2021supercritical}, we can require that $q \in ( (1-\eps)p , p )$ rather than just $q \in ( p_c(\eps,\eps) ,p  )$. Indeed, the exact same proof works, using $q_j := e^{ j \Delta(\eps) } (1-\eps)p$ instead of $q_j := e^{ j \Delta(\eps) } p_c(\eps,\eps)$, because $(1-\eps)p \geq p_c(\eps,\eps)$. So in the proof of \cite[Theorem 3.3]{easo2021supercritical}, we may assume that the parameter called $q$, which is provided by \cite[Lemma 3.5]{easo2021supercritical}, satisfies $q \geq (1-\eps)p$. In particular, when we later apply \cite[Theorem 2.1]{easo2021supercritical} to lower bound $\min_{u,v} \mathbb P_{q}( u \leftrightarrow v )$ by $e^{- 10^{5} \eps^{-18}}$, we could instead simply lower bound $\min_{u,v} \mathbb P_{q}( u \leftrightarrow v )$ by $\min_{u,v} \mathbb P_{(1-\eps)p}( u \leftrightarrow v )$. Running the rest of the proof of \cite[Theorem 3.3]{easo2021supercritical} exactly as written, except for the new definition ``$\tau := \min_{u,v} \mathbb P_{(1-\eps)p}( u \leftrightarrow v )$'' in place of ``$\tau := e^{- 10^{5} \eps^{-18}}$'', yields the desired conclusion.
\end{proof}

The uniqueness part of \cref{prop:uniqueness} will follow from \cref{lem:sharp_density,lem:technical_sandcastles_output}. The existence part will follow from the following well-known and (again) easy consequence of Talagrand's sharp threshold theorem \cite{MR1303654}.

\begin{lem}\label{lem:russo_tala}
Let $G$ be a finite transitive graph. Let $A$ be a non-trivial increasing event that is invariant under all graph automorphisms of $G$. Let $0 < p_1 < p_2 < 1$ and set $\delta := p_2 - p_1$. There exists a universal constant $c > 0$ such that
\[
	\mathbb P_{p_1}(A) \leq \frac{1}{\abs{V}^{c\delta}} \quad \text{or} \quad \mathbb \mathbb P_{p_2}(A) \geq 1-\frac{1}{\abs{V}^{c\delta}}.
\]
\end{lem}

\begin{proof}
For every edge $e$, let $\op{Orb}(e)$ denote the orbit of $e$ under the action of the automorphism group of $G$. By \cite[Theorem 3.10]{easo2021supercritical}, there is a universal constant $c_1 > 0$ such that for all $p \in (0,1)$, the function $f(p) := \mathbb P_p(A)$ satisfies
\[
	f'(p) \geq \frac{c_1}{p(1-p) \log \frac{2}{p(1-p)} } \cdot f(p) \left( 1 - f(p) \right) \cdot \log \left( 2 \min_{e \in E} \abs{ \op{Orb}(e) } \right).
\]
Since $G$ is (vertex-)transitive, $\abs{\op{Orb}(e)} \geq \frac{\abs{V}}{2}$ for every $e \in E$. Also, by calculus, $\sup_{p \in (0,1)} p(1-p) \log \frac{2}{p(1-p)} < \infty$. Therefore, there is another universal constant $c > 0$ such that for all $p \in (0,1)$,
\[
	\left[ \log \frac{f}{1-f} \right]' = \frac{f'}{f (1-f)} \geq 2c \log \abs{V}.
\]
The result follows by integrating this differential inequality.
\end{proof}

\begin{proof}[Proof of \cref{prop:uniqueness}]
Suppose that $q,p \in (0,1)$ satisfy $q - p \geq \delta$ and $\min_{u,v} \mathbb P_{p}( u \leftrightarrow v) \geq 2\delta$. We will assume throughout this proof that $\abs{V}$ is as large as we like with respect to $d$. Let us start with the existence of a large cluster. Let $c > 0$ be the universal constant from \cref{lem:russo_tala}. By the first implication in \cref{eq:two_point_bd_vs_giant}, we know that $\mathbb P_{p} ( \den{K_1} \geq \delta ) \geq \delta$. Since $\abs{V}$ is large,
\begin{equation}\label{eq:comparing_log_with_almost_log}
	\delta = e^{-\frac{1}{20} \log \log \abs{V} }  \geq e^{- c (\log \abs{V}) (\log \abs{V})^{-1/20} } = \abs{V}^{-c\delta}.
\end{equation}
So by \cref{lem:russo_tala} with $A:= \{ \den{K_1} \geq \delta \}$, since $\frac{\delta^4}{2} \geq \abs{V}^{-c\delta}$ (by a calculation like \cref{eq:comparing_log_with_almost_log}),
\begin{equation} \label{eq:existence}
	\mathbb P_{q}( \den{K_1} \geq \delta ) \geq 1 - \frac{1}{\abs{V}^{c\delta}} \geq 1-\frac{\delta^4}{2}.
\end{equation}

We now turn to the uniqueness of the largest cluster. The parameter $q$ is $(\delta/2)$-supercritical because $\abs{V} \geq 2(\delta/2)^{-3}$ and
\[
	\mathbb P_{(1-\delta/2)(p+\delta)} ( \den{K_1} \geq \delta ) \geq \mathbb P_{p} ( \den{K_1} \geq \delta ) \geq \delta.
\]
By \cref{lem:sharp_density}, since $\delta/2 \geq (2/\abs{V})^{1/3}$, there is a constant $C_1(d) < \infty$ such that $G$ has the $\Delta$-sharp density property for some $\Delta$ satisfying
\[
	\Delta(\delta/2) \leq \frac{C_1}{(\log \abs{V})^{1/2}}.
\]
So by \cref{lem:technical_sandcastles_output}, there is a constant $C_2(d) < \infty$ such that for every $\lambda \geq 1$,
\[
	\mathbb P_{q}\left( \den{K_2} \geq \frac{C_2 \lambda}{\delta^4 (\log \abs{V} )^{1/2}} \right) \leq 
	\mathbb P_{q}\left( \den{K_2} \geq \lambda \left[ \frac{ 200 \frac{C_1}{ (\log \abs{V} )^{1/2}  } } { (\delta/2)^3 \cdot (2\delta) } + \frac{25}{(\delta/2)^2 \cdot (2\delta) \cdot \abs{V}} \right] \right) \leq \frac{\delta/2}{\lambda}.
\]
By picking $\lambda$ such that $\frac{C_2 \lambda}{\delta^4 (\log \abs{V} )^{1/2}    } = \delta^2$ (which satisfies $\lambda \geq 1$ when $\abs{V}$ is large), it follows that
\begin{equation} \label{eq:uniqueness}
	\mathbb P_{q}( \den{K_2} \geq \delta^2 ) \leq \frac{ C_2 } { 2 \delta^5 (\log \abs{V})^{1/2} } = \frac{C_2 \delta^5}{2} \leq \frac{\delta^4}{2}.
\end{equation}
The conclusion follows by combining \cref{eq:existence,eq:uniqueness} with a union bound.
\end{proof}

\section{Unique large cluster $\to$ giant cluster}
In this section we apply the methods of \cite{MR4665636}. We will again use the notation $\den{K_1},\den{K_2}$ introduced in \cref{sec:global_connections_to_unique_large_cluster}. Let $\mathcal G$ be an infinite set of finite transitive graphs with possibly unbounded degrees. Recall that $\mathcal G$ is said to have a percolation threshold if there is a fixed sequence $p_c : \mathcal G \to (0,1)$ such that for every sequence $p : \mc G \to (0,1)$, if $\limsup p / p_c < 1$ then $\lim \mathbb P_p( \den{K_1} \geq \eps ) = 0$ for all $\eps > 0$, and if $\liminf p/p_c > 1$ then $\lim \mathbb P_p( \den{K_1} \geq \eps ) = 1$ for some $\eps > 0$. In \cite{MR4665636} we showed that $\mathcal G$ has a percolation threshold unless and only unless $\mathcal G$ contains a very particular family of pathological sequences of dense graphs. This might appear to be simply a matter of proving that some nice event has a sharp threshold, perhaps by a simple application of \cref{lem:russo_tala} in the bounded-degree case. The subtle problem is that ``$\{K_1 \text{ is a giant}\}$'' is not an event. Really the challenge is to prove that multiple events of the form $\{\den{K_1} \geq \alpha\}$, for different choices of $\alpha$, all have sharp thresholds that in fact coincide with each other.

The bulk of our proof consisted in proving that if the supercritical giant cluster for $\mathcal G$ is unique (as given by \cite{easo2021supercritical}), then we can embed this fact into Vanneuville's new proof of the sharpness of the phase transition for infinite transitive graphs \cite{vanneuville2023sharpnessbernoullipercolationcouplings,vanneuville2024exponentialdecayvolumebernoulli} to deduce a kind of mean-field lower bound for the supercritical giant cluster density. This mean-field-like lower bound implies that for every $\delta > 0$ and sequence $p$, if there is a giant whose density exceeds some constant $\alpha > 0$ at $p$, i.e.\! $\lim \mathbb P_p( \den{K_1} \geq \alpha ) = 1$, then there is a giant whose density exceeds some constant $c(\delta) > 0$ at $(1+\delta)p$, i.e.\! $\lim \mathbb P_{(1+\delta)p}( \den{K_1} \geq c(\delta) ) = 1$, where, crucially, $c(\delta)$ is independent of $\alpha$. By a diagonalisation argument, it is clear that $\alpha$ can be allowed to decay slowly rather than remain constant. However, in general, the slowest allowable rate of decay can be arbitrarily slow\footnote{Consider sequences that approximate sequences that do not have percolation thresholds.}. This is why there was no discussion of rates of convergence in \cite{easo2023critical}. Luckily, this is not the case in our restricted setting where graphs have bounded degrees.

In the following subsection we simply note that the argument in \cite{easo2023critical} is fully quantitative in the sense that $\alpha$ can decay at any particular rate provided that we supply a sufficiently strong bound on the uniqueness of the largest (possibly non-giant) cluster. This is the content of the following proposition.

\begin{prop}\label{prop:mflb}
Let $G$ be a finite transitive graph. Let $p,\eps \in (0,1)$ and $\alpha \in (0,\frac{\eps}{10})$. There is a universal constant $c > 0$ such that the following holds whenever $\abs{V}^{c\eps} \geq \frac{1}{c \eps}$.

\noindent If
\[
	\mathbb P_{p-\eps}( \den{K_o} \geq \alpha ) \geq \alpha \quad \text{and} \quad \mathbb P_{p}\left( \den{K_1} \geq \alpha \text{ and } \den{K_2} < \frac{\alpha}{2} \right) >  1- \alpha^2,
\] then
\[
	\mathbb P_{p+\eps}\left(\den{K_1} \geq c \eps \right) \geq 1-\frac{1}{\abs{V}^{c\eps}}.
\]
\end{prop}



\subsection{Proof via coupled explorations} \label{subsec:coupled_expl}
At the heart of Vanneuville's new proof of the sharpness of the phase transition for infinite transitive graphs is a stochastic comparison lemma. This says that starting with percolation of some parameter $p$, decreasing from $p$ to $p-\eps$ for a certain $\eps > 0$ has more of an effect than conditioning on a certain disconnection event $A$, roughly in the sense that
\[
	\mathbb P_{p-\eps}(\omega = \cdot \;) \leq_{\mathrm{st}} \mathbb P_p(\omega = \cdot \mid A ),
\]
where $\leq_{\mathrm{st}}$ denotes stochastic dominance with respect to the usual partial ordering $\{ 0,1 \}^E$. This is proved by coupling two explorations of the cluster at the origin, sampled according to each of the two laws. In \cite{MR4665636} we modified Vanneuville's argument to prove the following lemma (\cite[Lemma 8]{MR4665636}). Note that here the stochastic dominance only holds approximately, i.e.\! only on the complement of an event with small probability.

\begin{lem}\label{lem:coupling_lem}
Let $G$ be a finite transitive graph. Let $p,\alpha \in (0,1)$. Define
\[
	\theta := \mathbb E_p \den{K_1}, \quad h := \mathbb P_p\left( \den{K_1} < \alpha \text{ or } \den{K_2} \geq \frac{\alpha}{2} \right), \quad \delta := \frac{2 h^{1/2}}{1-\theta - h},
\]
and assume that $\theta + h < 1$ (so that $\delta$ is well-defined and positive). Then there is an event $A$ with $\mathbb P_p( A \mid \den{K_o} < \alpha ) \leq h^{1/2}$ such that
\[
	\mathbb P_{(1-\theta-\delta)p} (\omega = \mathbf{\cdot} \; ) \leq_{\mathrm{st}} \mathbb P_p( \omega \cup \mathbf{1}_A = \mathbf{\cdot} \mid \den{K_o} < \alpha  ),
\]
where $\mathbf 1_A$ denotes the random configuration with every edge open on $A$ and every edge closed on $A^c$.
\end{lem}

To prove \cref{prop:mflb}, we will simply combine this lemma together with \cref{lem:russo_tala}, which was a standard application of Russo's formula and Talagrand's inequality. 

\begin{proof}[Proof of \cref{prop:mflb}]
Define $\theta$ and $h$ as in \cref{lem:coupling_lem}. Suppose that $\mathbb P_{p-\eps}( \den{K_o} \geq \alpha ) \geq \alpha$ and $\mathbb P_{p}\left( \den{K_1} \geq \alpha \text{ and } \den{K_2} < \frac{\alpha}{2} \right) >  1- \alpha^2$, i.e.\! $h < \alpha^2$. First consider the case that $\theta+h \geq \frac{\eps}{2}$. Then by hypothesis and the fact that $\alpha \leq \frac{\eps}{10}$,
\begin{equation} \label{eq:easy_bound_on_theta}
	\theta \geq \frac{\eps}{2}-h \geq \frac{\eps}{2}- \alpha^2 \geq \frac{\eps}{4}.
\end{equation}
Now consider the case that $\theta+h < \frac{\eps}{2}$. Define $\delta$ as in \cref{lem:coupling_lem}. By \cref{lem:coupling_lem}, there is an event $A$ such that
\[
	\mathbb P_{(1-\theta-\delta)p} ( \den{K_o} \geq \alpha ) \leq \mathbb P_p( A \mid \den{K_o} < \alpha ) \leq h^{1/2}.
\]
On the other hand, by our hypotheses,
\[h^{1/2} < \alpha \leq \mathbb P_{p-\eps}( \den{K_o} \geq \alpha ).\] So by monotonicity, we must have $(1-\theta-\delta)p \leq p - \eps$. In particular, $\theta + \delta \geq \eps$. We can upper bound $\delta$ by
\[
	\delta = \frac{2 h^{1/2}}{ 1- (\theta+h)} \leq \frac{2 \alpha}{ 1 - \frac{\eps}{2} } \leq \frac{2 \cdot \frac{\eps}{10}}{1 - \frac{\eps}{2}} \leq \frac{2\eps}{5},
\]
where the last inequality used the fact that $\eps \in (0,1)$. Therefore, again, $\theta \geq \eps - \delta \geq \frac{\eps}{4}$, as in \cref{eq:easy_bound_on_theta}.

Let $c>0$ be the constant from \cref{lem:russo_tala}. Without loss of generality, assume that $c < \frac{1}{8}$. Suppose that $\abs{V}^{c\eps} \geq \frac{1}{c \eps}$. By Markov's inequality, $\mathbb P_p\left( \den{K_1} \geq \frac{\eps}{8} \right) \geq \frac{\eps}{8}$ because $\theta \geq \frac{\eps}{4}$. Therefore,
\[
	\mathbb P_p\left( \den{K_1} \geq c\eps \right) \geq \mathbb P_p \left( \den{K_1} \geq \frac{\eps}{8} \right) \geq \frac{\eps}{8} > c \eps \geq \frac{1}{\abs{V}^{c\eps}}.
\]
So by applying \cref{lem:russo_tala}, $\mathbb P_{p+\eps}( \den{K_1} \geq c \eps ) \geq 1- \abs{V}^{-c\eps}$, as required.
\end{proof}

\section{Proof of \cref{thm:main}}

	Let $\mc G$ be an infinite set of finite transitive graphs with bounded degrees. Suppose that for all but at most finitely many $G \in \mc G$,
	\begin{equation}\label{eq:non-one-dimensionality}
		\op{dist}_{\mathrm{GH}}\left( \frac{\pi}{\op{diam} G} G , S^1 \right) > \frac{ e ^{( \log \op{diam} G)^{1/9} } }{ \op{diam} G }.
	\end{equation}
	Our goal is to prove that both statements (1) and (2) are true. By \cref{prop:equivalent_notions_of_sharpness}, statement (1) implies statement (2). So it suffices to prove statement (1), i.e.\! percolation on $\mc G$ has a sharp phase transition. We will assume without loss of generality that there exists $d \in \mb N$ such that every $G \in \mc G$ has degree exactly $d$. We will again adopt the notation $\den{K_1},\den{K_2}$ from \cref{sec:global_connections_to_unique_large_cluster}.

	\begin{clm}
		For every constant $\eps > 0$, there exist constants $c(\eps) > 0$ and $\mu(d,\eps) < \infty$ such that for every infinite subset $\mc H \subseteq \mc G$ and every sequence $p : \mc H \to (0,1)$,
		\[
			\liminf_{G \in \mc H} \mathbb P_p( \abs{K_1} \geq \mu \log \abs{V} ) > 0 \quad \implies \quad \lim_{G \in \mc H} \mathbb P_{p + 4\eps}( \den{K_1} \geq c ) = 1.
		\]
	\end{clm}

	Before proving this claim, let us explain how to conclude from it. For each $G \in \mc G$, pick a parameter $q(G) \in (0,1)$ satisfying $\mathbb P_{q(G)}^G( \abs{K_1} \geq \abs{V}^{2/3} ) = \frac{1}{2}$. We will prove that percolation on $\mc G$ has a sharp phase transition with percolation threshold given by $q : \mc G \to (0,1)$. First notice that $\liminf q \geq \frac{1}{2d} > 0$. Indeed, this follows from the proof of \cite[Lemma 2.8]{easo2021supercritical}, but let us explain the elementary argument here for completeness. For every $G \in \mc G$ and $n \geq 1$, there are at most $d^n$ self-avoiding paths starting from $o$. So by a union bound, every $G \in \mc G$ satisfies
	\[
		\mathbb E_{\frac{1}{2d}} \abs{K_o} \leq \sum_{n=0}^{\infty} \frac{d^n}{(2d)^n} = 2.
	\]
	On the other hand, by transitivity, every $G \in \mc G$ satisfies
	\[\begin{split}
		\mathbb E_{\frac{1}{2d}} \abs{K_o} &\geq \abs{V}^{2/3} \mathbb P_{\frac{1}{2d}} \left( \abs{K_o} \geq \abs{V}^{2/3} \right) \geq \abs{V}^{1/3 }\mathbb P_{\frac{1}{2d}}\left( \abs{K_1} \geq \abs{V}^{2/3} \right). 
	\end{split}\]
	Therefore for all but finitely many $G \in \mc G$,
	\[
		\mathbb P_{\frac{1}{2d}}\left( \abs{K_1} \geq \abs{V}^{2/3} \right) \leq 2\abs{V}^{-1/3} < \frac{1}{2},
	\]
	and hence by monotonicity, $q(G) \geq \frac{1}{2d}$.

	Now fix a constant $\eps > 0$. Since $\liminf q \geq \frac{1}{2d} > 0$, there exists a constant $\delta(\eps,d) > 0$ such that $(1-\eps) q \leq q - \delta$ and $q + \delta \leq (1+\eps)q$ for all but finitely many $G \in \mc G$. Let $c\left( \delta/4 \right) > 0$ and $\mu(d, \delta / 4) < \infty$ be the constants provided by the claim. For all but finitely many $G \in \mc G$, we have $\mu \log \abs{V} < \abs{V}^{2/3}$. So by applying the claim with ``$\mc H$'' being the whole of $\mc G$ and ``$p$'' being $q$,
	\[
		\lim_{G \in \mc G}  \mathbb P_{q + \delta} \left( \den{K_1} \geq c \right) = 1. 
	\]
	On the other hand, for all but finitely many $G \in \mc G$, we have $c \abs{V} > \abs{V}^{2/3}$. So by applying the claim (contrapositively) with ``$p$'' being $q - \delta$, for every infinite subset $\mc H \subseteq \mc G$,
	\[
		\liminf_{G \in \mc H} \mathbb P_{q-\delta}( \abs{K_1} \geq \mu \log \abs{V} ) = 0.
	\]
	Equivalently, for every infinite subset $\mc H \subseteq \mc G$ there exists a further infinite subset $\mc H' \subseteq \mc H$ such that $\lim_{G \in \mc H'}  \mathbb P_{q-\delta}( \abs{K_1} \geq \mu \log \abs{V} ) = 0$. Therefore,
	\[
		\lim_{G \in \mc G} \mathbb P_{q - \delta} \left( \abs{K_1} \geq \mu \log \abs{V} \right) = 0.
	\]
	Since $\eps > 0$ was arbitrary, this establishes that percolation on $\mc G$ has a sharp phase transition. All that remains is to verify the claim.

	\begin{proof}[Proof of claim]
		Fix $\eps > 0$. Let $c_1(d,\eps) > 0$ be the constant from \cref{prop:large_cluster_to_local_connections}. Let $\lambda\left( d , \eps , \frac{\eps^2}{20} \right) < \infty$ be the constant from \cref{prop:local_to_global_connections} (with ``$\eta$'' set to $\eps^2/20$). Let $c_2 > 0$ be the universal constant from \cref{prop:mflb}. We will prove that the claim holds with $\mu := \exp\left( \frac{\lambda}{c_1} + \frac{1}{c_1^2} \right)$ and $c := c_2 \eps$. Let $\mc H \subseteq \mc G$ be an infinite subset, and let $p : \mc H \to (0,1)$ be a sequence satisfying
		\[
			\eta := \liminf_{G \in \mc H} \mathbb P_p( \abs{K_1} \geq \mu \log \abs{V} ) > 0.
		\]
		We say that a statement $A$ holds for \emph{almost every $G$} to mean that the set $\{ G \in \mc H:  A \text{ does not hold for } G \}$ is finite. For almost every $G$,
		\[
			\mathbb P_p( \abs{K_1} \geq \mu \log \abs{V} ) \geq \frac{\eta}{2} \geq \frac{1}{c_1 \abs{V}^{c_1}}.
		\]
		So by \cref{prop:large_cluster_to_local_connections}, noting that $c_1 \log \mu - \frac{1}{c_1} = \lambda$,
		\[
			\min_{u \in B_\lambda} \mathbb P_{p+\eps}( o \leftrightarrow u ) \geq \frac{\eps^2}{20}.
		\]

		For each $G \in \mc H$, define $\gamma(G)$ and $\gamma^+(G)$ as in \cref{prop:local_to_global_connections}. Then by \cref{prop:local_to_global_connections}, thanks to our choice of $\lambda$, for almost every $G$,
		\begin{equation} \label{eq:global_two_point_bound}
			\min_{ u \in B_{\gamma^+} } \mathbb P_{p+2\eps}( o \leftrightarrow u ) \geq e^{-\left( \log \log \gamma^+ \right)^{1/2}}.
		\end{equation}
		Consider a particular $G \in \mc H$ satisfying \cref{eq:non-one-dimensionality}. Then \! $\gamma(G) > e^{(\log \op{diam} G)^{1/9}}$, and by applying the monotone function $x \mapsto e^{(\log x)^9}$ to both sides, $\gamma^+(G) > \op{diam} G$. In particular, $B_{\gamma^+(G)}^G$ is the whole vertex set $V(G)$. We trivially have $\op{dist}_{\mathrm{GH}}\left( \frac{1}{\op{diam} G} G , \frac{1}{\pi} S^1 \right) \leq 1$, because both metric spaces involved have diameter $\leq 1$. So conversely, $\gamma(G) \leq \pi \op{diam} G$, and hence $\gamma^+(G) \leq e^{\left( \log ( \pi \op{diam G} ) \right)^{9} }$. By applying these upper and lower bounds on $\gamma^+(G)$ to \cref{eq:global_two_point_bound}, we deduce that for almost every $G$,
		\[
			\min_{u,v \in V} \mathbb P_{p+2\eps} ( u \leftrightarrow v ) \geq e^{-3 \left( \log \log \left( \pi \op{diam}G \right) \right)^{1/2}} \geq e^{-3 \left( \log \log \left( \pi \abs{V} \right) \right)^{1/2}},
		\]
		where the second inequality follows from the trivial bound $\abs{V} \geq \op{diam} G$.

		For each $G \in \mc H$, define $\delta(G) := (\log \abs{V})^{-1/20}$. For every sufficiently large positive real $x$,
		\[
			2(\log x)^{-1/20} = 2e^{- \frac{1}{20} \log \log x } \leq e^{-3 \left( \log \log \left( \pi x \right) \right)^{1/2}}.
		\]
		Therefore for almost every $G$,
		\begin{equation} \label{eq:good_global_lb}
			\min_{u,v \in V} \mathbb P_{p+2\eps}(u \leftrightarrow v) \geq 2\delta.
		\end{equation}
		By applying \cref{prop:uniqueness}, it follows that for almost every $G$, (since $\delta \leq \eps$)
		\[
			\mathbb P_{p+3\eps}\left( \den{K_1} \geq \delta \text{ and } \den{K_2} \leq \delta^2 \right) \geq 1 - \delta^4.
		\]
		For almost every $G$, we have $\delta^2 < \frac{\delta}{2}$, $\delta^4 < \delta^2$, $\delta \in ( 0 ,\frac{\eps}{10})$, $\abs{V}^{c_2 \eps} \geq \frac{1}{c_2 \eps}$, and by applying Markov's inequality to \cref{eq:good_global_lb}, $\mathbb P_{p+2\eps}( \den{K_o} \geq \delta ) \geq \delta$. So by \cref{prop:mflb}, for almost every $G$,
		\[
			\mathbb P_{p+4\eps}\left( \den{K_1} \geq c_2 \eps \right) \geq 1 - \frac{1}{ \abs{V}^{c_2\eps} }.
		\]
		In particular, $\lim_{G \in \mc H} \mathbb P_{p+4\eps}\left( \den{K_1} \geq c_2 \eps \right) = 1$, as claimed.
	\end{proof}

\newpage 


\section*{Appendix: Details for some claims in \cref{sec:local_connections_to_global_connections}}

In this appendix, we will explain how some of the lemmas in \cref{sec:local_connections_to_global_connections} can be established by minor modifications of existing arguments.

\begin{lem*}[\Cref{lem:green_propagates_orange}]
There exists $n_0(d) < \infty$ such that for all $n \geq n_0$,
\[
	s(\text{$[R(n),R^2(n)]$ is orange}) \leq s(\text{$n$ is green}) + \delta(n).
\]
\end{lem*}

The following argument is essentially contained in the proof of \cite[Proposition 4.5]{easo2023critical}.

\begin{proof}
Let $n \geq 3$ be some scale. Throughout this proof, we will assume that $n$ is large with respect to $d$. Note that the result is trivial if $n > \op{diam} G$, because in that case $B_{m} = B_n$ for all $m \in [R(n),R^2(n)]$. So let us assume to the contrary that $S_n \not=\emptyset$. First consider the case that $n \in \mb L$. Then at any time $t$ when $n$ is green, we know that $\kappa_{\phi(t)}\left( R^2(n) , n \right) \geq \delta(R(n))$, which implies that $[R(n),R^2(n)]$ is already orange at time $t$. So let us assume to the contrary that $n\not\in \mb L$. Define $h := e^{-(\log n)^{100}}$, which therefore satisfies $h \geq \op{Gr}(n)^{-1}$. Pick $p_1 \in (0,1)$ such that $n$ is green at time $\phi^{-1}(p_1)$. Note that $p_1 \geq 1/d$ because by a union bound, using that $S_n \not= \emptyset$ and that $n$ is large with respect to $d$,
\[
	\min_{u \in B_n} \mathbb P_{1/d} (o \leftrightarrow u) \leq \mathbb P_{1/d}(o \leftrightarrow S_n) \leq d (d-1)^{n-1} \cdot \left(\frac{1}{d}\right)^n < \delta(n).
\]
Define $p_2 := \phi( \phi^{-1}(p_1) + \delta(n) )$. In the language of \cite[Section 3]{easo2023critical}, the quantity ``$\delta(p_1,p_2)$" is equal to $\delta(n)$ by construction. Let $u \in B_{R^2(n)}$ be arbitrary, and let $o = u_0, u_1,\ldots,u_k = u$ be a path with $k \leq R^2(n)$. Let $c_1(1),h_0(d,1),c_2,c_3>0$ be the constants from \cite[Proposition 4.1]{easo2023critical} with $D := 1$. Since $n$ is large with respect to $d$, we have $h \leq h_0$, $\delta(n) \leq 1$,
\[
	h^{c_1 \delta(n)^3} = e^{ - c_1 (\log n)^{100} e^{ - 3 ( \log \log n )^{1/2} } } \leq \frac{c_3}{e^{ (\log n)^{81 }} + 1 } = \frac{c_3}{R^2(n)+1} \leq \frac{c_3}{k+1},
\]
and for all $i \in \{0,\ldots,k-1\}$, by Harris' inequality,
\[\begin{split}
	\min \left\{ \mathbb P_{p_1}( x \leftrightarrow y ) : x,y \in B_n(u_i) \cup B_n(u_{i+1}) \right\} &\geq \delta(n) \cdot p_1 \cdot \delta(n) \\
	&\geq \frac{1}{d} e^{- 2 ( \log \log n )^{1/2} } \\
	&\geq 4 e^{-c_1 (\log n)^{100} e^{- 4(\log \log n)^{1/2} } } = 4h^{c_1 \delta(n)^4}.
\end{split}\]
So by \cite[Proposition 4.1]{easo2023critical}, where the sets ``$A_1,\ldots,A_n$'' are the balls $B_n(u_0),\ldots,B_{n}(u_k)$,
\[
	\mathbb P_{p_2}(o \leftrightarrow u) \geq c_2 \delta(n)^2 \geq \delta(R(n)).
\]
Since $u \in B_{R^2(n)}$ was arbitrary, it follows that $[R(n),R^2(n)]$ is orange at time $\phi^{-1}(p_2)$, as required.
\end{proof}

\begin{lem*}[\cref{lem:green_to_orange_via_K_Delta}] 
	For all $c > 0$ there exist $\lambda(d,c),n_0(d,c),K(d,c)< \infty$ such that the following holds for all $n \geq n_0$ with $n \in \mathbb T(c,\lambda)$. For all $t \in \mb R$, if $n$ is orange at time $t$ and $K \Delta_t(n) \leq 1$ then
	\[
		s(\text{$n$ is green}) \leq t + K \Delta_t(n).
	\]
\end{lem*}

This is implicit in the proof of \cite[Proposition 6.1]{easo2023critical}

\begin{proof}
In \cite{easo2023critical}, we made the following definitions: given $d \geq 1$, we wrote $\mc U_d^*$ for the set of all infinite non-one-dimensional unimodular transitive graphs with degree $d$, and given $D \geq 1$ and a transitive graph $G$, we wrote $\mathcal L(G,D)$ for the set of all scales $n \geq 1$ such that $\op{Gr}(m) \leq e^{(\log m)^D}$ for all $m \in [n^{1/3},n]$. Let us now introduce the following variants of these definitions: given $d \geq 1$, write $\mc W_d$ for the set of all (possibly finite) unimodular transitive graphs with degree $d$, and given $D,\lambda \geq 1$, $c > 0$, and a transitive graph $G$, write $\mathcal T(G,D,\lambda,c)$ for the set all of scales $n \in \mathcal L(G,D)$ with $n \leq \op{diam} G$ such that $G$ has $(c,\lambda)$-polylog plentiful tubes throughout an interval of the form $[m_1,m_2]$ with $m_2 \geq m_1^{1+c}$ satisfying $[m_1,m_2] \subseteq [n^{1/3},n^{1/(1+c)}]$. Let \cite[Proposition* 6.1]{easo2023critical} be the result of modifying the statement of \cite[Proposition 6.1]{easo2023critical} as follows:

\begin{enumerate}
	\item Weaken the hypothesis that $G \in \mathcal U_d^*$ to the hypothesis that $G \in \mc W_d$.
	\item Strengthen the hypothesis that $n \in \mathcal L(G,D)$ to the hypothesis that $n \in \mathcal T(G,D,\lambda,1/D)$.
\end{enumerate}

Note that $p_c(G)$ in this statement refers to the usual percolation threshold for an infinite cluster, so in particular, $p_c(G) := 1$ if $G$ is finite. The same proof works because the hypothesis that $G$ was infinite and non-one-dimensional was only used to invoke \cite[Proposition 5.2]{easo2023critical} to establish that there is a constant $c_1(d,D) > 0$ such that for all $\lambda$, whenever $n$ is large with respect to $d,D,\lambda$, if $n \in \mathcal L(G,D)$ then automatically $n \in \mc T(G,D,\lambda,c_1)$. We are just circumventing this application of \cite[Proposition 5.2]{easo2023critical}. Specifically, we can prove \cite[Proposition* 6.1]{easo2023critical} by modifying the proof of \cite[Proposition 6.1]{easo2023critical} as follows:

\begin{enumerate}
	\item Strengthen the condition $n \in \mathcal L(G,D)$ to $n \in \mathcal T(G,D,\lambda,1/D)$ in the definition of $\mathcal A$.
	\item Rather than define $c_1$ and $N$ to be the constants guaranteed to exist by \cite[Proposition 5.2]{easo2023critical}, set $c_1 := 1/D$ and $N := 3$.
	\item Restrict the domain of the definition of $\mathcal P(n)$ from all $n \in \mathcal L(G,D)$ to all $n \in \mathcal T(G,D,\lambda,1/D)$.
	\item Include the hypothesis $n \in \mc T(G,D,\lambda,1/D)$ in the statement of \cite[Lemma 6.8]{easo2023critical}.\footnote{While writing this paper, we noticed the following typo: \cite[Lemma 6.8]{easo2023critical} is missing the hypothesis that $n \in \mc L(G,D)$.}
\end{enumerate}

Taking \cite[Proposition* 6.1]{easo2023critical} for granted, let us now explain how to prove \cref{lem:green_to_orange_via_K_Delta}. Recall that $G$ is a finite transitive graph with degree $d$. Let $c > 0$ be given, and define $D := 101 \vee (1/c)$. Let $\lambda_0(d,D)$ and $c_1(d,D)$ (called ``$c(d,D)$'') be the constants provided by \cite[Proposition* 6.1]{easo2023critical}. Define $\lambda := \lambda_0 \vee \left(100/c_1\right)$. Now let $K_1(d,D,\lambda)$ and $n_0(d,D,\lambda)$ be the corresponding constants provided by \cite[Proposition* 6.1]{easo2023critical}. Define $K := K_1^{1/4}$. By the same argument as in our proof of \cref{lem:green_propagates_orange} above, there exists $n_1(d) < \infty$ such that for all $n_1 \leq n \leq \op{diam} G$ and $t \in \mb R$, if $n$ is orange at time $t$ then $\phi(t) \geq 1/d$. Set $n_2 := n_0 \vee n_1 \vee e^{3^{101}}$. We claim that $\lambda, n_2, K$ have the properties required of the constants called ``$\lambda,n_0,K$'' in the statement of \cref{lem:green_to_orange_via_K_Delta}.

Indeed, suppose that $t \in \mb R$ and $n \geq n_2$ with $n \in \mb T(c,\lambda)$ are such that $n$ is orange at time $t$ and $K \Delta_t(n) \leq 1$. Now apply \cite[Proposition* 6.1]{easo2023critical} with the variables called ``$K,n,b,p_1,p_2$'' in that statement set to our variables $K_1, n , U_t(n), \phi(t), \phi(t + K \Delta_t(n))$. The only hypothesis that is not immediately obvious is that $n \in \mc T(G,D,\lambda,1/D)$. To see this, first note that since $n \in \mb L$ and $n \geq e^{3^{101}}$, every $m \in [n^{1/3},n]$ satisfies
\[
	\op{Gr}(m) \leq \op{Gr}(n) \leq e^{ (\log n)^{100} } \leq e^{ (\log (n^{1/3}))^{101} } \leq e^{(\log m)^{101}} \leq e^{(\log m)^{D}}.
\]
So $n \in \mathcal L(G,D)$. Second, we may assume that $n \leq \op{diam} G$, otherwise the conclusion of \cref{lem:green_to_orange_via_K_Delta} is trivial. Finally, since $n \in \mb T(c,\lambda)$ and $1/D < c$, and the property of having ``$(x,\lambda)$-polylog plentiful tubes" at a given scale gets weaker as we decrease $x$, it follows that $n \in \mc T(G,D,\lambda,1/D)$. Therefore, by applying \cite[Proposition* 6.1]{easo2023critical}, we deduce that 
\[
	\kappa_{\phi(t + K \Delta_t(n))} \left( e^{ (\log n)^{c_1 \lambda} } , n   \right) \geq e^{- 3(\log \log n)^{1/2}}.
\]
In particular, since $c_1 \lambda \geq 100 \geq 81$, 
\[
	\kappa_{\phi(t + K \Delta_t(n))}\left( R^2(n),n \right) \geq \delta(R(n)).
\]
So $s(n \text{ is green}) \leq t+K \Delta_t(n)$, as required.
\end{proof}

\begin{lem*}[\cref{lem:green_at_log_implies_small_delay}]
There exists $n_0(d) < \infty$ such that the following holds for all $n \in \mathbb L$ with $n \geq n_0$. For all $t \in \mathbb R$, if $L(n)$ is green at time $t$ then
\[
	\Delta_t(n) \leq \frac{1}{\log \log n}.
\]
\end{lem*}

This proof is implicit in \cite[Section 6.3]{easo2023critical}.

\begin{proof}
Suppose that $n \in \mb L$. Throughout this proof we will implicitly assume that $n$ is large with respect to $d$. Let $t \in \mb R$ and assume that $L(n)$ is green at time $t$. We may assume that $\lfloor n^{1/3} \rfloor \leq \op{diam} G$, otherwise we trivially have $U_t(n) = \lfloor \frac{1}{8} n^{1/3} \rfloor$ and hence (since $n$ is assumed large) $\Delta_t(n) \leq (\log n)^{-1/5}$. We split the proof into two cases according to whether $L(n) \in \mb L$.

First suppose that $L(n) \not\in \mb L$. By the same argument as in our proof of \cref{lem:green_propagates_orange} above, since $L(n) \leq \op{diam} G$ and $L(n)$ is green at time $t$ (and since $n$ is assumed large), $\phi(t) \geq 1/d$. So by \cite[Corollary 2.4]{easo2023critical}, there exist constants $c(d) > 0$ and $C(d),n_0(d) < \infty$ such that for all $m \geq n_0(d)$,
\[
	\mathbb P_{\phi(t)}\left( \op{Piv}[c \log m , m] \right) \leq C \left( \frac{\log \op{Gr}(m) }{m} \right)^{1/3}.
\]
In particular, since $4 L(n) \leq c \log (n^{1/3})$ and $\op{Gr}(n^{1/3}) \leq \op{Gr}(n) \leq e^{(\log n)^{100}}$,
\[
	\mathbb P_{\phi(t)}\left( \op{Piv} \left[ 4 L(n) , n^{1/3} \right] \right) \leq C \left( \frac{ (\log n)^{100} }{n^{1/3}} \right)^{1/3} \leq \frac{1}{\log n}.
\]
Since we also clearly have $L(n) \leq \frac{1}{8}n^{1/3}$, it follows that $U_t(n) \geq \lfloor L(n) \rfloor$. Since $L(n) \not\in \mb L$, this implies that $\op{Gr}(U_{t}(n)) \geq e^{(\log L(n))^{100}}$. So
\[
	\Delta_t(n) \leq \left( \frac{\log \log n}{ (\log n) \wedge ( \log L(n) )^{100} } \right)^{1/4} \leq \frac{1}{\log \log n}.
\]

Next suppose that $L(n) \in \mb L$. Define $b := \frac{1}{5}\left( R \circ L(n) \wedge \op{Gr}^{-1} \left(R^{-1}(n)\right) \right)$. By \cite[Lemma 2.3]{easo2023critical} (i.e.\! \cite[Lemma 6.2]{contreras2022supercritical}), using the fact that $5b \leq \frac{1}{2} n^{1/3}$,
\[
	\mathbb P_{\phi(t)} \left( \op{Piv}\left[ 4b , n^{1/3} \right] \right) \leq \mathbb P_{\phi(t)}\left( \op{Piv}\left[ 1, \frac{1}{2} n^{1/3} \right] \right) \cdot \frac{ \abs{S_{4b}}^2 \op{Gr}(5b)  }{ \min_{x,y \in S_{4b}} \mathbb P_{\phi(t)} ( x \xleftrightarrow{B_{5b}} y ) }.
\]
By \cite[Lem 2.1]{easo2023critical} (i.e.\! essentially \cite[Proposition 4.1]{contreras2022supercritical}), there is a constant $C(d) < \infty$ such that
\[
	\mathbb P_{\phi(t)}\left( \op{Piv}\left[ 1, \frac{1}{2} n^{1/3} \right] \right) \leq C  \left( \frac{\log \op{Gr}\left(\frac{1}{2} n^{1/3}\right) }{ \frac{1}{2} n^{1/3}} \right)^{1/3}.
\]
By hypothesis, $n \in \mb L$. So we can upper bound $\log \op{Gr}(\frac{1}{2} n^{1/3}) \leq \log \op{Gr}(n) \leq (\log n)^{100}$. Since $L(n)$ is green at time $t$ but $L(n) \in \mb L$, then $\kappa_{\phi(t)}(   R^2 \circ L(n) , L(n) ) \geq \delta ( R \circ L(n))$. Note that $8b \leq R^2 \circ L(n)$ and (using that $L(n) \in \mb L$), $L(n) \leq b$. Therefore, $\min_{x,y \in S_{4b}} \mathbb P_{\phi(t)} ( x \xleftrightarrow{B_{5b}} y ) \geq \delta ( R \circ L(n))$, since we can connect any $x,y \in S_{4b}$ by a path contained in $B_{4b}$ of length at most $8b$, and the $b$-thickened tube around this path is entirely contained in $B_{5b}$. Finally, we can upper bound $\abs{S_{4b}} \leq \op{Gr}(5b) \leq R^{-1}(n)$ by definition of $b$. Therefore,
\[
	\mathbb P_{\phi(t)} \left( \op{Piv}\left[ 4b , n^{1/3} \right] \right) \leq C \left( \frac{ (\log n)^{100} }{ \frac{1}{2}n^{1/3} } \right)^{1/3} \frac{\left(R^{-1}(n)\right)^3 }{ \delta( R \circ L(n) )} \leq \frac{1}{\log n}.
\]
Notice that by our choice of $b$, we have $b \leq \frac{1}{8}n^{1/3}$ and
\[
	\op{Gr}(b) \geq \op{Gr}\left( \frac{1}{5} R\circ L(n) \right) \wedge \op{Gr}\left( \frac{1}{5} \op{Gr}^{-1}\left( R^{-1}(n) \right) \right) \geq \left(\frac{1}{5}R \circ L(n)\right) \wedge \left( R^{-1}(n) \right)^{1/5} = \frac{1}{5}R\circ L(n).
\]
So
\[
	\Delta_t(n) \leq \left( \frac{ \log \log n } { (\log n) \wedge \left( \log \left[ \frac{1}{5} R \circ L(n)  \right] \right) } \right)^{1/4} \leq \frac{1}{\log \log n}.
\]
\end{proof}

\begin{lem*}[\cref{lem:tubes_from_fast_tripling}]
Let $G$ be a unimodular transitive graph of degree $d$. Suppose that
\[
	\operatorname{Gr}(m) \leq e^{(\log m)^D} \quad \text{and} \quad \operatorname{Gr}(3m) \geq 3^5 \operatorname{Gr}(m)
\]
for every $m \in [n^{1-\eps},n^{1+\eps}]$, where $\eps,D,n > 0$. Then there is a constant $c(d,D,\eps) > 0$ with the following property. For every $\lambda \geq 1$, there exists $n_0(d,D,\eps,\lambda) < \infty$ such that if $n \geq n_0$ then $G$ has $(c,\lambda)$-polylog plentiful tubes at scale $n$.
\end{lem*}

\cite[Lemma 5.4]{easo2023critical} is the same statement but with the additional hypothesis that $G$ is infinite. We claim that this additional hypothesis is unnecessary. 

\begin{proof}
\cite[Lemma 5.4]{easo2023critical} is the ultimate conclusion of \cite[Section 5.2]{easo2023critical}. The first result in \cite[Section 5.2]{easo2023critical} that requires $G$ to be infinite is \cite[Lemma 5.16]{easo2023critical}. By inspecting the proof of \cite[Lemma 5.16]{easo2023critical}, we see that this hypothesis is only used in order to apply the elementary bound $\op{Gr}(3mn) \geq n\op{Gr}(m)$ for all $m,n \geq 1$. In fact, in the language of that proof, since we may assume that the constant $c > 0$ satisfies $c \leq 1/10$, say, then the proof only invokes this elementary bound for $m,n$ satisfying $3mn \leq \frac{1}{10}t^{1/2}$. Now this holds whenever $\op{diam}G \geq \frac{1}{10}t^{1/2}$. So \cite[Lemma 5.16]{easo2023critical} holds with the hypothesis ``$G$ is infinite" replaced by the weaker hypothesis ``$\op{diam}G \geq \frac{1}{10}t^{1/2}$". When \cite[Lemma 5.16]{easo2023critical} is applied to establish \cite[Lemma 5.17]{easo2023critical}, the hypothesis ``$\op{diam}G \geq \frac{1}{10}t^{1/2}$" is already implied by the other hypothesis of \cite[Lemma 5.17]{easo2023critical} that $\op{Gr}(3m) \geq 3^{\kappa}\op{Gr}(m)$ for all $n \leq m \leq \frac{1}{2}t^{1/2}$ (and the fact that conclusion of \cite[Lemma 5.17]{easo2023critical} is trivial if there is no integer in $[n, \frac{1}{2} t^{1/2}]$). So in the statement of \cite[Lemma 5.17]{easo2023critical}, we can simply drop the hypothesis that $G$ is infinite.

We can also drop the hypothesis that $G$ is infinite in \cite[Lemmas 5.18 and 5.20]{easo2023critical} because \cite[Lemma 5.18]{easo2023critical} is deduced from \cite[Lemma 5.17]{easo2023critical}, and \cite[Lemma 5.20]{easo2023critical} is deduced from \cite[Lemma 5.18]{easo2023critical}. \cite[Lemma 5.19]{easo2023critical} already does not require $G$ to be infinite. The ultimate proof of \cite[Lemma 5.4]{easo2023critical} only required $G$ to be infinite in order to invoke \cite[Lemma 5.20]{easo2023critical} and (in the radial case) to know that $S_n \not=\emptyset$. The hypothesis that $S_n \not= \emptyset$ is anyway implied by the fact that $\op{Gr}(3m) \geq \op{Gr}(m)$ for some $m \in [n,n^{1+\eps}]$, and as we explained, we can drop the hypothesis that $G$ is infinite in \cite[Lemma 5.20]{easo2023critical}. Therefore we can drop the hypothesis that $G$ is infinite in the statement of \cite[Lemma 5.4]{easo2023critical} too.
\end{proof}

The next claim we will justify is that \cref{lem:two_bad_graphs_make_a_circle} implies \cref{lem:tubes_from_small_tripling_new}. Here are the statements of these results.

\begin{lem*}[\cref{lem:two_bad_graphs_make_a_circle}]
Let $r,n \geq 1$. Let $G$ be a finite transitive graph such that $S_n^{\infty}$ is not $r$-connected. Let $H$ be a (finite or infinite) transitive graph with
$\delta(H) \leq r$ that does not have infinitely many ends. If $B_{50n}^H \cong B_{50n}^G$, then
\[
	\op{dist}_{\mathrm{GH}}\left(\frac{\pi}{\op{diam} G}G,S^1 \right) \leq \frac{200n}{\op{diam} G}.
\]
\end{lem*}

\begin{lem*}[\cref{lem:tubes_from_small_tripling_new}]
Let $G$ be an finite transitive graph of degree $d$. Suppose that $\operatorname{Gr}(3n) \leq 3^\kappa \operatorname{Gr}(n)$, where $n,\kappa > 0$. There exists $C(d,\kappa) < \infty$ such that the following holds if $n \geq C$:

\noindent There is a set $A \subseteq [1,\infty)$ with $\abs{A} \leq C$ such that for every $k \geq 1$ and every $m \in [Ckn,\infty) \backslash \bigcup_{a \in A}[a,2ka]$, if $G$ does not have $(C^{-1}k,C^{-1}k^{-1}m,Ck^Cm)$-plentiful tubes at scale $m$, then
\[
	\op{dist}_{\mathrm{GH}}\left( \frac{\pi}{\op{diam} G} G,S^1 \right) \leq \frac{Cm}{\operatorname{diam}G}.
\]

\end{lem*}

The proof that \cref{lem:two_bad_graphs_make_a_circle} implies \cref{lem:tubes_from_small_tripling_new} is essentially the same as the proof of \cite[Proposition 5.3]{easo2023critical} (i.e.\! \cref{lem:tubes_from_small_tripling_old}), except that $G$ is now assumed to be a finite transitive graph rather than a non-one-dimensional infinite transitive graph. For this reason, the following proof is terse. The argument relies on the structure theory of transitive graphs of polynomial growth. See the proof of \cite[Proposition 5.3]{easo2023critical} for more details and \cite[Section 5.1]{easo2023critical} for more background.

\begin{proof}[Proof of \cref{lem:tubes_from_small_tripling_new} given \cref{lem:two_bad_graphs_make_a_circle}]
Fix $\kappa > 0$. Suppose that $\op{Gr}(3n) \leq 3^{\kappa} \op{Gr}(n)$ for some $n \geq 1$. We will implicitly assume that $n$ is large with respect to $d$ and $\kappa$. Let $H \leq \op{Aut}(G)$, $S \subseteq \Gamma := \op{Aut}(G) / H$, and $C_1(K) < \infty$ be as given by \cite[Theorem 5.5]{easo2023critical} (which is taken from \cite{MR4253426}) with $K := 3^{\kappa}$. Let $G' := \op{Cay}(\Gamma,S)$. For each $k \in \mb N$, let $R_k$ be the set of all relations in $\Gamma$ having word length at most $k$, let $\langle \langle R_k \rangle \rangle$ be the normal subgroup of the free group on $S$ generated by $R_k$, and let $G_k' := \op{Cay}( \langle S \mid R_k \rangle, S )$. By items 7 and 8 of \cite[Theorem 5.5]{easo2023critical},
\[
	\frac{ \op{Gr}' (3n )  } { \op{Gr}'(n) } \leq C_1^2 ( 3 + C_1 )^{C_1}.
\]
In particular, by \cite[Theorem 5.5]{easo2023critical} again (and using that $n$ is large), every transitive graph whose $3n$-ball is isomorphic to the $3n$-ball in $G'$ is necessarily finite or infinite with polynomial growth. In particular, such graphs have at most finitely many ends. Now by \cite[Theorem 1.1]{EHStructure}, there exists $C_2(\kappa,d) < \infty$ such that
\[
	\abs{\left\{ i \in \mb N : i \geq \log_2 n \text{ and } \langle \langle R_{2^{i+1}} \rangle \rangle  \not= \langle \langle R_{2^{i}} \rangle \rangle \right\}} \leq C_2.
\]
Let $A := \left\{ 2^i : i \in \mb N \text{ and } i \geq \log_2 n \text{ and } \langle \langle R_{2^{i+10}} \rangle \rangle \not= \langle \langle R_{2^{i}} \rangle \rangle \right\}$, and note that $\abs{A} \leq 10 C_2$. Let $k \geq 1$ and $m \in [2  kn,\infty) \backslash \bigcup_{a \in A} [a,2ka]$ be arbitrary. By construction of $A$ (and \cite[Lemma 5.6]{easo2023critical}), the balls of radius (say) $50n$ in $G_{\frac{m}{k}}'$ and $G'$ are isomorphic. Note that $\delta \left(G_{\frac{m}{k}}' \right) \leq \frac{m}{k}$, and since the $3n$-ball in $G_{\frac{m}{k}}'$ is isomorphic to the $3n$-ball in $G'$, the graph $G_{\frac{m}{k} }'$ has at most finitely many ends. Consider an arbitrary pair $m_1,m_2 \in \mb N$ satisfying $\frac{m}{k} \leq m_1 \leq m_2 \leq 3m$. By \cref{lem:two_bad_graphs_make_a_circle} applied with the pair ``$(G,H)$" equal to $(G',G_{\frac{m}{k}}')$, either (1) the exposed sphere $S_{m_2}^{\infty}(G')$ is $\lceil \frac{m}{k} \rceil$-connected, or (2) 
\begin{equation} \label{eq:qi_to_circle_from_cay}
\op{dist}_{\mathrm{GH}} \left( \frac{\pi}{\op{diam} G'} G',S^1 \right) \leq \frac{200m_2}{\op{diam} G'}.
\end{equation}
In case (1), we deduce by the proof of \cite[Lemma 2.7]{contreras2022supercritical} (which was behind \cite[Lemma 5.8]{easo2023critical}) that for all $u,v \in S_{m_2}^{\infty}(G')$ there exists a path from $u$ to $v$ in $G'$ that is contained in $\bigcup_{x \in S_{m_2}^\infty(G')} B_{2m_1}(x)$ and has length at most $3m_1 \op{Gr}(3m_2)/\op{Gr}(m_1)$. Now consider case (2). The existence of a $(1,C_1n)$-quasi-isometry from $G$ to $G'$ implies that $\abs{\op{diam} G - \op{diam} G'} \leq 3C_1n$ and $\op{dist}_{\mathrm{GH}}(G,G') \leq C_1n$. (For the latter, see exercise 5.10 (b) in \cite{MR4696672}, for example.) We may assume without loss of generality that $\op{diam} G \geq 100 C_1 n$, say, otherwise our claim is trivial. By combining these simple bounds with \cref{eq:qi_to_circle_from_cay}, we deduce that $\op{dist}_{\mathrm{GH}}( \frac{\pi}{\op{diam} G} G,S^1) \leq \frac{C_3 m}{\op{diam} G}$ for some constant $C_3(\kappa,d) < \infty$.

We now run the rest of the proof of \cite[Proposition 5.3]{easo2023critical}, after the application of \cite[Lemma 5.8]{easo2023critical}, as it is written. This establishes that there is a constant $C_4(\kappa,d) < \infty$ such that for all $k \geq 1$ and $m \in [C_4 k n , \infty) \backslash \bigcup_{a \in A} [a,2ka]$, either (A) there exists $m_2 \in \left[\frac{10}{9} m , \frac{12}{9}m\right]$ such that $S_{m_2}^{\infty}(G')$ is not $ \lceil \frac{m}{k} \rceil$-connected, or (B) $G$ has $(C_4^{-1} k ,C_4^{-1}k^{-1}, C_4 k^{C_4} m)$-plentiful tubes at scale $m$. (Technically, as written, the radial case of the proof of \cite[Proposition 5.3]{easo2023critical} invokes the existence of a bi-infinite geodesic in $G$. All that is really required is a geodesic of length $\geq 24m$. So it suffices to know that $\op{diam} G \geq 24 m$, say, which we may anyway assume without loss of generality otherwise the conclusion holds trivially.) By above, if case (A) holds, then case (2) holds, and hence $\op{dist}_{\mathrm{GH}}( \frac{\pi}{\op{diam} G} G,S^1) \leq \frac{C_3 m}{\op{diam} G}$. Therefore the set of scales $A$ is as required.
\end{proof}

The next claim we will justify is that Timar's proof \cite{timar2007cutsets} of Benjamini-Babson \cite{MR1622785} yields the following statement, which is phrased slightly differently to usual, in terms of \emph{sets of} vertices, \emph{vertex} cutsets, and (extrinsic) \emph{diameter} rather than length of generating cycles.

\begin{lem*}[\cref{lem:timar_improved}]
Let $G$ be a graph. Let $A$ and $B$ be sets of vertices. Let $\Pi$ be a minimal $(A,B)$-cutset that does not disconnect $A$ or $B$. Then $\Pi$ is $\delta(G)$-connected.
\end{lem*}

\begin{proof}
Suppose that $\Pi = \Pi_1 \sqcup \Pi_2$ is a non-trivial partition of $\Pi$. By minimality of $\Pi$, there exist paths $\gamma_1$ avoiding $\Pi_2$ and $\gamma_2$ avoiding $\Pi_1$ that both start in $A$ and end in $B$. Let $\gamma_A$ be a path from the startpoint of $\gamma_1$ to the startpoint of $\gamma_2$ that avoids $\Pi$, and let $\gamma_B$ be a path from the endpoint of $\gamma_1$ to the endpoint of $\gamma_2$ that avoids $\Pi$. Let $\{C_i: i \in I\}$ be a set of cycles of diameter $\leq \delta(G)$ such that $\gamma_1 + \gamma_2 + \gamma_A + \gamma_B = \sum_{i \in I} C_i$. Let $J$ be the set of all indices $i \in I$ such that $C_i$ visits $\Pi_1$, and define
\[
	\zeta := \gamma_1 + \sum_{i \in J} C_i = \gamma_2 + \gamma_A + \gamma_B + \sum_{i \in I \backslash J} C_i.
\]
From either expression for $\zeta$, we see that $\zeta$ has exactly two odd-degree vertices, one in $A$ and the other in $B$. So $\zeta$ contains a path from $A$ to $B$, and hence contains an edge incident to $\Pi$. From the second expression for $\zeta$, we see that $\zeta$ does not contain an edge incident to $\Pi_1$. So $\zeta$ must contain an edge incident to $\Pi_2$. By construction, $\gamma_1$ avoids $\Pi_2$. So by the first expression for $\zeta$, there must exist a cycle $C_{i}$ with $i \in J$ that visit $\Pi_2$. Since this $C_i$ also visits $\Pi_1$ (by definition of $J$) and has diameter at most $\delta(G)$, it follows that $\operatorname{dist}(\Pi_1,\Pi_2) \leq \delta(G)$.
\end{proof}

\printbibliography
\end{document}